\documentclass[letterpaper, 11pt,  reqno]{amsart}

\usepackage{amsmath,amssymb,amscd,amsthm,amsxtra, esint}
\usepackage{tikz} 

\usepackage{mathtools}
\usepackage{color}
\usepackage[implicit=true]{hyperref}

\usepackage{verbatim}

\usepackage{cases}

\usepackage[left=32mm, right=32mm, 
bottom=27mm]{geometry}

\allowdisplaybreaks[2]

\usepackage{color}

\definecolor{gr}{rgb}   {0.,   0.69,   0.23 }
\definecolor{bl}{rgb}   {0.,   0.5,   1. }
\definecolor{mg}{rgb}   {0.85,  0.,    0.85}
\definecolor{yl}{rgb}   {0.8,  0.7,   0.}
\definecolor{or}{rgb}  {0.7,0.2,0.2}

\usetikzlibrary{shapes.misc}
\usetikzlibrary{shapes.symbols}
\usetikzlibrary{shapes.geometric}

\tikzset{
	ddot/.style={circle,fill=white,draw=black,inner sep=0pt,minimum size=0.8mm},
	>=stealth,
	}

\tikzset{
	ddot2/.style={circle,fill=black,draw=black,inner sep=0pt,minimum size=0.8mm},
	>=stealth,
	}

\newtheorem{theorem}{Theorem} [section]

\newtheorem{lemma}[theorem]{Lemma}
\newtheorem{proposition}[theorem]{Proposition}
\newtheorem{remark}[theorem]{Remark}

\newtheorem{corollary}[theorem]{Corollary}

\DeclareMathOperator*{\supp}{supp}

\newcommand{\1}{\hspace{0.5mm}\text{I}\hspace{0.2mm}}
\newcommand{\II}{\text{I \hspace{-2.8mm} I} }
\newcommand{\III}{\text{I \hspace{-2.9mm} I \hspace{-2.9mm} I}}
\newcommand{\IV}{\text{I \hspace{-2.9mm} V}}

\newcommand{\I}{\mathcal{I}}

\newcommand{\noi}{\noindent}
\newcommand{\Z}{\mathbb{Z}}
\newcommand{\R}{\mathbb{R}}
\newcommand{\C}{\mathbb{C}}
\newcommand{\T}{\mathbb{T}}
\newcommand{\bul}{\bullet}

\let\Re=\undefined\DeclareMathOperator*{\Re}{Re}
\let\Im=\undefined\DeclareMathOperator*{\Im}{Im}

\let\P= \undefined
\newcommand{\P}{\mathbf{P}}
\newcommand{\PP}{\mathbb{P}}

\newcommand{\Q}{\mathbf{Q}}

\newcommand{\E}{\mathbb{E}}

\renewcommand{\L}{\mathcal{L}}

\newcommand{\F}{\mathcal{F}}

\newcommand{\al}{\alpha}
\newcommand{\be}{\beta}
\newcommand{\dl}{\delta}

\newcommand{\nb}{\nabla}

\newcommand{\Dl}{\Delta}
\newcommand{\eps}{\varepsilon}
\newcommand{\kk}{\kappa}
\newcommand{\g}{\gamma}
\newcommand{\G}{\Gamma}
\newcommand{\ld}{\lambda}
\newcommand{\Ld}{\Lambda}
\newcommand{\s}{\sigma}

\newcommand{\ft}{\widehat}
\newcommand{\Ft}{{\mathcal{F}}}
\newcommand{\wt}{\widetilde}
\newcommand{\cj}{\overline}

\newcommand{\dt}{\partial_t}

\newcommand{\ta}{\theta}
\newcommand{\ze}{\zeta}

\renewcommand{\l}{\ell}
\newcommand{\om}{\omega}
\renewcommand{\O}{\Omega}

\newcommand{\U}{\Theta}

\newcommand{\les}{\lesssim}
\newcommand{\ges}{\gtrsim}

\newcommand{\jb}[1]
{\langle #1 \rangle}

\newcommand{\ind}{\mathbf 1}

\newcommand{\V}{\mathcal{V}}

\newcommand{\D}{\mathcal{D}}

\newcommand{\N}{\mathbb{N}}

\renewcommand{\H}{\mathcal{H}}

\newtheorem*{ackno}{Acknowledgements}

\newcommand{\too}{\longrightarrow}

\newcommand{\rhoo}{\vec{\rho}}

\newcommand{\deff}{\stackrel{\textup{def}}{=}}

\newcommand{\Ph}{\P_\eps^{\textup{high}}}
\newcommand{\Pl}{\P_\eps^{\textup{low}}}
\newcommand{\Phh}{\P_\eps^{\textup{high}, \theta}}
\newcommand{\Pll}{\P_\eps^{\textup{low}, \theta}}

\numberwithin{equation}{section}
\numberwithin{theorem}{section}

\begin{document}

\title[On the convergence of the 2-$d$ singular SNLW]
{Smoluchowski-Kramers approximation for singular stochastic wave
equations in two dimensions}

\author[Y.~Zine]
{Younes Zine}


\address{
Younes Zine\\
School of Mathematics\\
The University of Edinburgh\\
and The Maxwell Institute for the Mathematical Sciences\\
James Clerk Maxwell Building\\
The King's Buildings\\
Peter Guthrie Tait Road\\
Edinburgh\\ 
EH9 3FD\\
United Kingdom\\
and Chair of Probability and PDEs\\
EPFL SB MATH\\
PROPDE\\
MA C2 595\\
CH-1015 Lausanne\\
Switzerland}

\email{younes.zine@epfl.ch}



\begin{abstract}We study a family of nonlinear damped wave equations indexed by a parameter $\eps >0$ and forced by a space-time white noise on the two dimensional torus, with polynomial and sine nonlinearities. We show that as $\eps \to 0$, the solutions to these equations converge to the solution of the corresponding two dimensional stochastic quantization equation. In the sine nonlinearity case, the convergence is proven over arbitrary large times, while in the polynomial case, we  prove that this approximation result holds over arbitrary large times when the parameter $\eps$ goes to zero even with a lack of suitable global well-posedness theory for the corresponding wave equations.
\end{abstract}



\maketitle
%


\tableofcontents

\section{Introduction}\label{SEC1}

\subsection{Smoluchowski-Kramers approximation} For $\eps>0$, we look at the following stochastic damped nonlinear wave equation on $\T^2 = (\R / 2 \pi \Z)^2$,

\noi
\begin{align}
\begin{cases}
\eps^2 \dt^2 u_\eps +  \dt u_\eps + (1- \Dl) u_\eps +  \mathcal{N}(u_\eps)  \,  = \Phi \xi,    \\
(u_\eps,\dt u_\eps)|_{t=0}=(\phi _0,\phi _1),
\end{cases}
\quad (x,t) \in \T^2 \times \R_+.
\label{SNLWa}
\end{align}

\noi
In the above, $\mathcal{N}$ is a nonlinearity, $\Phi$ is a bounded operator on $L^2(\T^2)$ and $\xi(x, t)$ denotes a real-valued (Gaussian) space-time white noise on $\T^2\times \R_+$
with the space-time covariance formally given by
\[ \E\big[ \xi(x_1, t_1) \xi(x_2, t_2) \big]
= \dl(x_1 - x_2) \dl (t_1 - t_2).\]

It has been proven in various contexts that the solution $u_\eps$ to \eqref{SNLWa} converges to the solution $u$ of the following parabolic equation as $\eps \to 0$:

\noi
\begin{align}
\begin{cases}
\dt u + (1 - \Dl) u + \mathcal{N}(u) = \Phi \xi,    \\
u |_{t=0}= \phi_0,
\end{cases}
\quad (x,t) \in \T^2 \times \R_+.
\label{SQEa}
\end{align}

\noi
This type of convergence result is known as the Smoluchowski-Kramers approximation. In the infinite dimensional setting, various works have proven the Smoluchowski-Kramers approximation property on some bounded domains and for equations stochastically forced by some random noises \cite{CF1,CF2,CF3,CF4,CG,CS1,CS2,CS3,ChF}. Once the approximation is established other important questions arise: how do relevant properties for the second and first order systems compare? For instance, the evolutions of the invariant states \cite{CF2,CF4,CG} and the large deviation and exit problems \cite{CS2,ChF} have also been studied. However in dimension larger than two, the available results only considered random noises which were smooth in space. Indeed, although, in the polynomial case, the well-posedness theory for \eqref{SQEa} has been known for some time since the work of Da Prato and Debussche \cite{DPD}, it is only until recently that well-posedness for \eqref{SNLWa} is well understood. See \cite{GKO1,GKO2,GKOT,OOR,OOT1,OOT2,OWZ}. In the case of a sine nonlinearity; namely, for the so-called sine-Gordon model, the well-posedness issue for the wave and heat equations has only been studied recently \cite{HS, ORSW1, ORSW2, CHS}.

In this paper, we investigate the Smoluchowski-Kramers approximation problem for infinite dimensional systems stochastically forced by a white noise for the first time. We note that Fukuizimi, Hoshino and Inui studied independently in a recent paper \cite{FHI} a similar problem: the non-relativistic limit problem for the (complex) stochastic nonlinear damped wave equation in two dimensions with the Gibbs measure initial data. See Remark \ref{RMK:FHI} for some additional comments on the differences between our results and the ones in \cite{FHI}.

In the deterministic setting (i.e. with $\Phi \equiv 0$ in \eqref{SNLWa} and \eqref{SQEa}) such problems have already been studied in \cite{IM} on Euclidean domains. Note that on $\R^d$ ($d \ge 2$), the scaling $(x,t) \mapsto (\eps^{-1}x, \eps^{-2}t)$ correlates the asymptotics $\eps \to 0$ and $t \to \infty$ in \eqref{SNLWa} (with $\Phi \equiv 0$ and with $1 - \Dl$ replaced by $-\Dl$). In this regard, it has been observed, that solutions to nonlinear wave equations tend to get closer to solutions to corresponding linear heat equations over large times. See for instance \cite{HO,MN,Narazaki,Wakasugi}.

We now informally discuss why the Smoluchowski-Kramers approximation occurs. Consider the following homogeneous linear damped wave equation: 
\begin{align}
\begin{cases}
\eps^2 \dt^2 u_\eps +  \dt u_\eps + (1 - \Dl) u_\eps  = 0    \\
(u_\eps,\dt u_\eps)|_{t=0}=(\phi_0,\phi_1).
\end{cases}
\label{X1}
\end{align}

\noi
By taking the spatial Fourier transform, we have 
\begin{align}
\eps^2 \dt^2 \ft u_\eps (n) +  \dt \ft u_\eps(n)  + \jb{n}^2 \ft  u_\eps(n)   = 0    
\label{X2}
\end{align}

\noi
for $n \in \Z^2$.
The roots of the characteristic polynomial $\eps^2 \Ld^2 + \Ld + \jb{n}^2 = 0$ are given by 
\begin{align}
\Ld_\eps^\pm (n)
= \frac{-1 \pm \sqrt{ 1- 4 \jb{n}^2 \eps^2}}{2\eps^2}.
\label{ld1}
\end{align}

\noi
Note that we have $\Ld_\eps^\pm (n) \in \R$ if and only if $\jb{n} \leq (2\eps)^{-1}$. In the low frequency regime $\jb{n} \leq (2\eps)^{-1}$,
the solution to \eqref{X2} with 
\begin{align}
(\ft u_\eps(n) ,\dt \ft u_\eps(n))|_{t=0}=(\ft \phi_0(n),\ft \phi_1(n))
\label{X3}
\end{align}

\noi
is given by 
\begin{align}
\ft  u_\eps(n, t)  = e^{-\frac{t}{2\eps^2}} \cosh (\ld_\eps(n) t) \ft \phi_0(n)
+ e^{-\frac{t}{2\eps^2}} \frac{\sinh (\ld_\eps(n) t)}{\ld_\eps(n)}
\bigg( \frac{1}{2\eps^2} \ft \phi_0(n) + \ft \phi_1(n)\bigg), 
\label{L1}
\end{align}

\noi
where $\ld_\eps(n)$ is defined by 
\begin{align}
\ld_\eps(n)
= \frac{ \sqrt{ 1- 4 \jb{n}^2 \eps^2}}{2\eps^2}.
\label{ld2}
\end{align}

\noi
In the high frequency regime $\jb{n} >  (2\eps)^{-1}$,
the solution to \eqref{X2} with initial data \eqref{X3}
is given by

\begin{align}
\ft  u_\eps(n, t)  = e^{-\frac{t}{2\eps^2}} \cos (\ze_\eps(n) t) \ft \phi_0(n)
+ e^{-\frac{t}{2\eps^2}} \frac{\sin (\ze_\eps(n) t)}{\ze_\eps (n)}
\bigg( \frac{1}{2\eps^2} \ft \phi_0(n) + \ft \phi_1(n)\bigg), 
\label{L2}
\end{align}

\noi
where $\ze_\eps(n)$ is defined by 
\begin{align}
\ze_\eps(n)
= \frac{ \sqrt{  4 \jb{n}^2 \eps^2 - 1}}{2\eps^2}.
\label{ld3}
\end{align}

Let $\Pl$ and $\Ph$ be the sharp projections
onto the (spatial) frequencies 
 $\{ n \in \Z^2: \jb{n} \leq (2\eps)^{-1}\}$
and  $\{ n \in \Z^2: \jb{n} >  (2\eps)^{-1}\}$, respectively, 
defined by 
\begin{align}
\Pl  f =  \F^{-1} (\ind_{\jb{n} \le (2\eps)^{-1}} \ft f(n))
\qquad \text{and}\qquad 
\Ph  f =  \F^{-1} (\ind_{\jb{n} > (2\eps)^{-1}} \ft f(n)).
\label{proj1}
\end{align}

\noi
Then, define the operator $S_\eps(t)$ and $\D_\eps(t)$ by setting
\begin{align}
S_\eps(t) = 
 \frac{\sinh (\ld_\eps(\nb) t)}{\ld_\eps(\nb)} \Pl
 +  \frac{\sin (\ze_\eps(\nb) t)}{\ze_\eps(\nb)} \Ph
\label{L3}
\end{align}

\noi
and 
\begin{align}
\D_\eps(t) = e^{-\frac{t}{2\eps^2}} S_\eps(t) .
\label{L4}
\end{align}

\noi
\begin{remark}\rm \label{RMK:set}
At this point, the operator $\D_\eps$ is only defined for $\eps \in I$ where $I$ is the set $I:= (0,\infty) \setminus \{  \frac{1}{2 \jb{n}} : n \in \Z^2 \}$. However, we show in Lemma \ref{LEM:F3} that the map $\eps \in I \mapsto \D_\eps$ can be extended to $(0,\infty)$. In what follows, since we are interested in the behaviour of the solutions to \eqref{SNLWa} near $\eps =0$, we will consider that all quantities are defined for $\eps \in [0,1]$.
\end{remark}

\noi
From  \eqref{L1} and \eqref{L2}, we see that 
the solution $u_\eps$ to \eqref{X1} is given by 
\begin{align}
u_\eps(t) = \dt \D_\eps(t) \phi_0 + \D_\eps(t) (\eps^{-2} \phi_0 + \phi_1).
\label{L5}
\end{align}

\noi
Furthermore, by the Duhamel principle, 
the solution $u_\eps$ to the following nonhomogeneous
linear damped wave equation:
\begin{align}
\begin{cases}
\eps^2 \dt^2 u_\eps +  \dt u_\eps + (1 - \Dl) u_\eps  = F\\
(u_\eps,\dt u_\eps)|_{t=0}=(\phi_0,\phi_1).
\end{cases}
\label{X4}
\end{align}

\noi
is given by 
\begin{align}
u_\eps(t) = \dt \D_\eps(t) \phi_0 + \D_\eps(t) (\eps^{-2} \phi_0 + \phi_1)
+ \int_0^t \eps^{-2}\D_\eps(t -t') F(t') dt'
\label{L6}
\end{align}

From \eqref{ld1} with a Taylor expansion, we have
\begin{align}
\begin{split}
\Ld_\eps^+ (n)
& = \frac{-2\jb{n}^2} {1+ \sqrt{ 1- 4 \jb{n}^2 \eps^2}}
= - \jb{n}^2 + O(\jb n ^4 \eps^2) \too - \jb n ^2,\\
\Ld_\eps^- (n)
& = \frac{-2\jb{n}^2} {1- \sqrt{ 1- 4 \jb{n}^2 \eps^2}}
\too - \infty
\end{split}
\label{ld4}
\end{align}

\noi
in the regime $\jb{n} = o(\eps^{-\frac12})$ as $\eps \to 0$, 
namely in the regime 

\begin{align}
\sqrt{ 1- 4 \jb{n}^2 \eps^2} =  1 - 2 \jb n ^2 \eps^2 + O(\jb n ^4 \eps^4)
\label{taylor1}
\end{align}

\noi
as $\eps \to 0$.
Let $\chi$ be a smooth non-negative function such that $\chi \equiv 1$ on $\{ x \in \R : |x| \le 1\}$ and $\supp (\chi) \subset \{ x \in \R: |x| \le 2\}$.  Hence, if we denote by $\P_{N}$ ($N \in \R$) the smooth projection onto (spatial) frequencies $\{ n \in \Z^2 : \jb n \le N \}$ defined by

\noi
\begin{align}
\P_{N} f := \F^{-1} ( \chi_{N}(n)  \ft f (n) ),
\label{proj1b}
\end{align}

\noi
with $\chi_N = \chi( \frac{ \jb{\cdot}}{N} )$,\footnote{Here, we introduce smooth Fourier cut-offs as they are required in the proof of Lemma \ref{LEM:green} below.} we have, at a formal level:
\begin{align}
\begin{split}
 \eps^{-2} \D_\eps(t) \P_{\eps^{-\frac12+\ta}}  
& = \frac{e^{\Ld_\eps^+(\nb) t } - e^{\Ld_\eps^-(\nb) t }}{\sqrt{ 1- 4 \jb{\nb}^2 \eps^2}}
 \P_{\le \eps^{-\frac12+\ta}}  
 \too P_0(t)  \P_{ \eps^{-\frac12+\ta}}  , 
\label{L7}\\
\dt \D_\eps(t)   \P_{ \eps^{-\frac12+\ta}}  
& = \bigg( \Big( 1 - \frac{1}{\sqrt{1 - 4 \jb{\nb}^2 \eps^2}} \Big) \frac{e^{\Ld^+_\eps(\nb)t}}{2}\\
& \hphantom{XX}
+ \Big( 1 + \frac{1}{\sqrt{1 - 4 \jb{\nb}^2 \eps^2}} \Big) \frac{e^{\Ld^-_\eps(\nb)t}}{2}  \bigg) \P_{ \eps^{-\frac12+\ta}}  \\
& \too 0
\end{split}
\end{align}

\noi
for any $0 < \ta \ll 1$, 
where 
\begin{align}
P_0(t) = e^{(\Dl - 1) t}.
\label{Lheat}
\end{align}

\noi
See Lemma \ref{LEM:F2} for a rigorous justification of \eqref{L7}. 

\noi
\begin{remark}\rm
By looking more carefully into the kernels involved, one can prove that the convergence \eqref{L7} occurs in the regime $\jb n = o(\eps^{-1})$. Namely, the convergence in \eqref{L7} is valid with $\P_{ \eps^{-\frac12+\ta}}$ replaced by $\P_{ \eps^{-1+\ta}}$, see Lemma \ref{LEM:F2}
\end{remark}

\noi
\begin{remark}
\rm Note that (the proof of) Lemma \ref{LEM:F2} shows in particular that $ \P_{\eps^{- 1 + \ta}} \big( \eps^{-2} \D_\eps - P_0 \big)$ and $\P_{\eps^{- 1 + \ta}} \dt \D_\eps$ - viewed as operators from $H^s(\T^2)$ to itself for any $s \in \R$ - both converge to zero as $\eps \to 0$ only pointwisely in time. In order to get a uniform-in-time convergence, i.e. in $L^{\infty} \big( [0,T]; H^s(\T^2)\big)$ for any $ T >0$ and $s \in \R$, one must work with the operator $\P_{\eps^{- \frac12 + \ta}} \big( \eps^{-2} \D_\eps + \dt \D_\eps \big)$ which enjoys some extra cancellation. See Lemma \ref{LEM:F2} and Corollary \ref{COR:OP}.
\end{remark}

\noi
Therefore,
we see that $u_\eps$ in \eqref{L6} formally converges
to 
\begin{align}
u(t) = P_0(t) \phi_0 
+ \int_0^t P_0(t -t') F(t') dt', 
\label{L8}
\end{align}

\noi
satisfying the nonhomogeneous linear heat equation:
\begin{align}
\begin{cases}
  \dt u + (1 - \Dl) u  = F\\
u|_{t=0}=\phi_0.
\end{cases}
\label{X5}
\end{align}

\noi
Thus, we expect that the solution to \eqref{SNLWa} converges to the solution of the stochastic quantization equation \eqref{SQEa} as $\eps \to 0$.

\subsection{Results and outline of the approach}

For any $\eps > 0$, we introduce the stochastic convolutions by

\noi
\begin{align}
& \Psi_\eps= \sqrt 2 \int_{0}^t \eps^{-2}\D_\eps(t -t') dW(t') \label{sto1}, \\
& \Psi_0 = \sqrt 2 \int_{0}^t P_0(t-t') dW(t') \label{sto2}.
\end{align}

\noi
where  $W$ denotes a cylindrical Wiener process on $L^2(\T^2)$, defined on some probability space $(\O,\mathbb{P})$:
\begin{align}
W(t)
: = \sum_{n \in \Z^2 } B_n (t) e_n
\label{W1}
\end{align}

\noi
and  
$\{ B_n \}_{n \in \Z^2}$ 
is defined by 
$B_n(t) = \jb{\xi, \ind_{[0, t]} \cdot e_n}_{ x, t}$.
Here, $\jb{\cdot, \cdot}_{x, t}$ denotes 
the duality pairing on $\T^2\times \R$.
As a result, 
we see that $\{ B_n \}_{n \in \Z^2}$ is a family of mutually independent complex-valued\footnote
{In particular, $B_0$ is  a standard real-valued Brownian motion.} 
Brownian motions conditioned so that $B_{-n} = \cj{B_n}$, $n \in \Z^2$. 
By convention, we normalized $B_n$ such that $\text{Var}(B_n(t)) = t$.

In this paper we look at \eqref{SNLWa} and \eqref{SQEa} in the (singular) white noise setting $\Phi \equiv \sqrt 2 \operatorname{Id}$ with two types of nonlinearities: the polynomial case and the sine nonlinearity. Namely, $\mathcal{N}(u) = u^k$ for some integer $k \ge 2$ or $\mathcal{N}(u) = \g \sin(\be u)$ for some $\g, \be >0$. The corresponding equations are (for $\eps >0)$,

\noi
\begin{align}
& \begin{cases}
\eps^2 \dt^2 u_\eps +  \dt u_\eps + (1- \Dl) u_\eps +  u_\eps^k  \,  = \sqrt 2 \xi,    \\
(u_\eps,\dt u_\eps)_{|t=0}=(\phi _0,\phi _1),
\end{cases}
\quad (x,t) \in \T^2 \times \R_+.
\label{SNLW} \\ \nonumber \\
&
\begin{cases}
\eps^2 \dt^2 u_\eps + \dt u_\eps + (1- \Dl)  u_\eps   + \g  \sin(\be u_\eps) = \sqrt 2 \xi\\
(u_\eps, \dt u_\eps)_{|_{t = 0}} = (\phi_0, \phi_1) , 
\end{cases}
\qquad (x, t) \in \T^2 \times \R_+,
\label{SdSG}
\end{align}

\noi
while their parabolic counterparts (i.e. for $\eps=0$) read

\noi
\begin{align}
& \begin{cases}
\dt u + (1 - \Dl) u + u^k = \sqrt 2 \xi,    \\
u_{|t=0}= \phi_0,
\end{cases}
\quad (x,t) \in \T^2 \times \R_+
\label{SQE} \\ \nonumber \\
& \begin{cases}
\dt u + (1- \Dl)  u +  \g  \sin(\be u) = \sqrt 2 \xi\\
u_{|t=0} = \phi_0, 
\end{cases}
\qquad (x, t) \in \T^2 \times \R_+.
\label{SSG}
\end{align}

\noi

Given $\eps \in [0,1]$ and $N \in \N$, 
we define the truncated stochastic convolution $\Psi_{\eps,N} = \P_{ N} \Psi_{\eps}$, 
solving the truncated linear stochastic  wave equation/heat equation (for $\eps =0$):

\noi
\begin{align}
\eps^2 \dt^2 \Psi_{\eps,N} +  \dt \Psi_{\eps,N} + (1- \Dl) \Psi_{\eps,N}  = \sqrt 2  \P_{N} \xi,
\label{convo}
\end{align}

\noi
with the zero initial data.
Here,   $\P_{N}$ is as in \eqref{proj1b}. Then, $\Psi_{\eps,N}$ is represented by the following formula:

\noi
\begin{align}
\Psi_{\eps,N} = \sqrt 2 \int_{0}^t \eps^{-2}\D_\eps(t -t') d (\P_N W)(t'),
\label{convo10}
\end{align}

\noi
where $W$ is as in \eqref{W1}. For each fixed $\eps \in (0,1]$, $x \in \T^2$ and $t \geq 0$,  
we see from \eqref{convo10} and \eqref{ld3} that  $\Psi_{\eps,N}(x, t)$
 is a mean-zero real-valued Gaussian random variable with variance
\begin{align}
\begin{split}
\s_{\eps,N}(t)  \stackrel{\text{def}}{=} \E \big[\Psi_{\eps,N}(x, t)^2\big]
& = C_\eps(t) +  \sum_{\substack{n \in \Z^2\\ \eps^{-1} \ll \jb n  \les N}} 
\int_0^t \bigg[\frac{\sin( \ze_\eps(n) (t - t'))}{\ze_\eps(n)} \bigg]^2 dt'
\\
& \sim  \log N
\end{split}
\label{sig}
\end{align}

\noi
\noi
for some constant $C = C_\eps(t)$ and  $N \gg \eps^{-1}$ and $\ze_\eps$ as in \eqref{ld3}. Note that the implicit constant in \eqref{sig} depends on $\eps$ and $t$. See Lemma \ref{LEM:cov} for precise bounds on $\s_{\eps,N}$.
We point out that the variance $\s_{\eps,N}(t)$ is time-dependent.
For any $t > 0$, 
we see that  $\s_{\eps,N}(t) \to \infty$ as $N \to \infty$, 
which can be used to show  that $\{\Psi_{\eps,N}(t)\}_{N \in \N}$
is almost surely unbounded in $W^{0, p}(\T^2)$ for any $1 \leq p \leq \infty$. Similar comments apply to $\Psi_{0,N}$ and its variance $\s_{0,N}$. 

For notational convenience, we denote by $u_0$ and $w_0$ the solutions to \eqref{SQE} and \eqref{SSG} respectively and we view them as solutions to \eqref{SNLW} and \eqref{SdSG} with $\eps = 0$. For $N \in \N$ and $\eps \in [0,1]$, let $u_{\eps,N}$ and $w_{\eps,N}$ denote the solution to \eqref{SNLW} and \eqref{SdSG} where the rough noise $\xi$ is replace by the regularized noise $\P_{N}\xi$. Proceeding with the following decomposition of $u_{\eps,N}$ the solutions to \eqref{SNLW} and \eqref{SdSG} (\cite{McKean, BO96, DPD}):

\noi
\begin{align}
u_{\eps,N} =  v_{\eps,N} + \Psi_{\eps,N}.
\label{H1}
\end{align}

\noi
In the case of the polynomial nonlinearity \eqref{SNLW}, this decomposition leads to

\noi
\begin{align}
\eps^2 \dt^2 v_{\eps,N} +  \dt v_{\eps,N} + (1 - \Dl) v_{\eps,N} +  \sum_{\ell=0}^k {k\choose \ell} \Psi_{\eps,N}^\ell v_{\eps,N}^{k-\ell} = 0 \label{SNLW2}.
\end{align}

\noi
Due to the deficiency of regularity, 
the power $\Psi_{\eps,N}^\ell$ does not converge to any limit  as $N\to \infty$. 
This is where we  introduce
the Wick renormalization.
Namely, 
we replace  $\Psi^\ell_{\eps,N}$
by  its Wick ordered counterpart: 
\begin{align}
:\!\Psi^\l_{\eps,N}(x, t) \!: \, \stackrel{\text{def}}{=} H_\ell(\Psi_{\eps,N}(x, t);\sigma_{\eps,N}(t)), 
\label{Herm1}
\end{align}

\noi
where 
$H_\l(x; \s )$ is the Hermite polynomial of degree $\l$ with variance parameter $\s$.
See Section~\ref{SEC:2}.
Then, for each $\l \in \N$, the Wick power
$ :\! \Psi_{\eps,N}^\l \!:$
converges 
to a limit, denoted by 
$:\! \Psi_{\eps}^\l  \!: \,$, 
in $C([0,T];W^{-\s,\infty}(\T^2))$
for any $\eps \in [0,1]$, $\s > 0$, and $T > 0$, 
almost surely. See Proposition \ref{PROP:sto} below.

Similarly, in the case of the sine nonlinearity \eqref{SdSG}, the decomposition \eqref{H1} gives

\noi
\begin{align}
\eps^2 \dt^2 v_{\eps,N} +  \dt v_{\eps,N} + (1 - \Dl) v_{\eps,N} + \g \Im \big(  e^{i \be v_{\eps,N}}  e^{i \be \Psi_{\eps,N}}   \big) = 0. \label{SdSG2}
\end{align}

\noi
Since $\Psi_{\eps,N}$ converges to a distribution of negative regularity, $e^{i \be \Psi_{\eps,N}}$ will experience some ``averaging effect" and will converge to $0$ in the sense of distributions. Hence, in order to see a non-trivial behavior in the limit $N \to \infty$, one needs to have $\g  \to \infty$ in \eqref{SdSG2}. More precisely, we define $\g_{\eps,N} = \g_{\eps,N}(\be)$ by 

\noi
 \begin{align}\label{CN}
 \g_{\eps,N} = e^{\frac{\be^2}{2}\s_{\eps,N}},
 \end{align}

\noi
and the so-called imaginary Gaussian multiplicative chaos $\U_{\eps,N}$ by 
 \begin{align}
 \U_{\eps,N}(t,x)  = \,:\!e^{ i\be\Psi_{\eps,N}(t,x)}\!:\,\,  
 \deff
\g_{\eps,N} e^{i \be\Psi_{\eps,N}(t,x)}
=  e^{\frac{\be^2}2 \s_{\eps,N}} e^{i \be \Psi_{\eps,N}(t,x)}, 
\label{Ups}
 \end{align} 
 
\noi
where 
 $\g_{\eps,N}$ and $\s_{\eps,N}$ are as in \eqref{CN} and \eqref{sig}. Using stochastic analysis, one can then show that $\U_{\eps,N}$ converges almost surely to a distribution $\U_{\eps}$ in $C \big( [0,T]; W^{- \frac{\be^2}{4 \pi} - \s, \infty }(\T^2) \big)$, for any $\s, T > 0$. 
 
\noi
\begin{remark}\rm
We actually show that our stochastic objects $:\! \Psi_{\eps}^\l  \!: $ and $\U_{\eps}$ are continuous in $\eps \in [0,1]$. For the latter this is achieved by tracking down the exact dependence in $\eps \in [0,1]$ of the covariance functions $\big\{ \E [ \Psi_{\eps}(t,x) \Psi_{\eps}(t,y) ] \big\}_{\eps \in [0,1]}$. See Subsection \ref{SUBSEC:sin1}.
\end{remark}
 
These renormalizations give rise 
to the renormalized versions of \eqref{SNLW2} and \eqref{SdSG2}:

\noi
 \begin{align}
& \eps^2 \dt^2 v_{\eps,N} + \dt v_{\eps,N} + (1 - \Dl) v_{\eps,N} +  \sum_{\ell=0}^k {k\choose \ell} :\!\Psi_{\eps,N}^\l\!:  v_{\eps,N}^{k-\ell} = 0, \label{SNLW11} \\
& \eps^2 \dt^2 v_{\eps,N} + \dt v_{\eps,N} +(1-\Dl)v_{\eps,N}  + \Im \big\{e^{i \be  v_{\eps,N} }\U_{\eps,N} \big\}=0, \label{SdSG11}
\end{align}

\noi
for $\eps \in [0,1]$. By taking a limit as $N \to \infty$, 
we then obtain the limiting equations:

\noi
 \begin{align}
&\eps^2 \dt^2 v_\eps + \dt v_\eps + (1 - \Dl) v_\eps +  \sum_{\ell=0}^k {k\choose \ell} :\!\Psi_\eps^\l\!:  v_\eps^{k-\ell} = 0,
\label{SNLW4}\\
& \eps^2 \dt^2 v_{\eps} + \dt v_{\eps} +(1-\Dl)v_{\eps}  + \Im \big\{e^{i \be  v_{\eps} }\U_{\eps} \big\}=0. \label{SdSG3} \
\end{align}

\noi
Given the almost-sure space-time regularity of the Wick powers
$:\!\Psi_\eps^\l\!:$, $\l = 1, \dots, k$, and of the imaginary Gaussian multiplicative chaos $\U_{\eps}$, standard deterministic analysis using the product estimates (Lemma~\ref{LEM:bilin})
yield local well-posedness of \eqref{SNLW4} and \eqref{SdSG3}.

Recalling the decomposition~\eqref{H1}, 
this argument also shows that 
the solution $u_{\eps,N} = \Psi_{\eps,N} + v_{\eps,N}$, with $v_{\eps,N}$ solving \eqref{SNLW11}, to the renormalized 
equation with the regularized noise $\P_{N} \xi$

\noi
\begin{align} 
 \eps^2 \dt^2 u_{\eps,N} + \dt u_{\eps,N} + (1 -  \Dl)  u_{\eps,N}   + :\!u_{\eps,N}^k\!: \, = \sqrt 2 \P_{N} \xi \label{SNLW10}
\end{align}

\noi
where the renormalized nonlinearity $:\!u_{\eps,N}^k\!: $ is interpreted as

\noi
\begin{align*}
& :\!u_{\eps,N}^k \!: 
= \, :\!(\Psi_{\eps,N} + v_{\eps,N})^k \!: \, 
= 
 \sum_{\ell=0}^k {k\choose \ell} :\!\Psi_{\eps,N}^\l\!:  v_{\eps,N}^{k-\ell},
\end{align*}

\noi
converge almost surely to the stochastic process $u_\eps
 = \Psi_\eps + v_\eps$,  where $v_\eps$ satisfies \eqref{SNLW4}.

Similarly, the solution $u_{\eps,N} = \Psi_{\eps,N} + v_{\eps,N}$, with $v_{\eps,N}$ solving \eqref{SdSG11}, to the renormalized equation with regularized noise $\P_N \xi$

\noi
\begin{align}
\eps^2 \dt^2 u_{\eps,N} + \dt u_{\eps,N} + (1 -  \Dl)  u_{\eps,N}   + :\!  \sin \big( \be u_{\eps,N} \big)  \!: \, = \sqrt 2 \P_{N} \xi \label{SdSG10}
\end{align}

\noi
where the nonlinearity is interpreted as 

\noi
\begin{align}
:\! \sin(\be u_{\eps,N})\!: = \g_{\eps,N} \sin ( \be( \Psi_{\eps,N}+ v_{\eps,N})) =\Im \big\{e^{i \be  v_{\eps,N} }\U_{\eps,N} \big\},
\end{align}

\noi
converges almost surely to the stochastic process $u_\eps = \Psi_\eps + v_\eps$ where $v_\eps$ solves \eqref{SdSG3}.

It is in this sense that we say that 
the following renormalized versions of the polynomial models \eqref{SNLW} and \eqref{SQE}

\noi
\begin{align}
\eps^{2}\dt^2 u_\eps + \dt u_\eps + (1 -  \Dl)  u  _\eps +  :\!u_\eps^k\!: \, & = \sqrt 2 \xi  \label{rSNLW},\\
 \dt u +  (1 -  \Dl)  u  +  :\!u^k\!: \,& = \sqrt 2 \xi  \label{rSNLW1},
\end{align}

\noi
and the sine-Gordon models \eqref{SdSG} and \eqref{SSG}

\begin{align}
 \eps^2 \dt^2 u_{\eps} + \dt u_{\eps} + (1 -  \Dl)  u_{\eps}   + :\! \sin(\be u_{\eps})\!: \, & = \sqrt 2  \xi,\label{rSdSG}\\
\dt u + (1 -  \Dl)  u   + :\! \sin(\be u) \!: \, & = \sqrt 2   \xi \label{rSdSG1},
\end{align}

\noi
are locally well-posed for $\eps \in [0,1]$
(and for  initial data of suitable regularity).

We can now state the main results of the paper. Firstly, we show the following local existence and Schmoluchowski-Kramers approximation for polynomial nonlinearities. Let $\H^s(\T^2) := H^s(\T^2) \times H^{s-1}(\T^2) $ for any $s \in \R$.

\noi
\begin{theorem}\label{THM1}Fix $k \ge 2$ an integer. Let $(\phi_0, \phi_1) \in \H^s(\T^2)$ for $s > \frac{2k-3}{2k-2}$. There exists a random time $T = T(\om) $ almost surely positive such that for each $\eps \in (0,1]$, there exists a solution $u_\eps$ to \eqref{rSNLW} with initial data $(\phi_0,\phi_1)$ and a solution $u$ to \eqref{rSNLW1} with initial data $\phi_0$ which belong to the class $C \big( [0,T]; H^{-\s}(\T^2) \big)$, for any $\s > 0$.

Moreover, $\{u_\eps\}_{\eps \in (0,1]}$ converges almost surely to the solution $u$ in $C \big( [0,T]; H^{-\s}(\T^2) \big)$ as $\eps \to 0$.
\end{theorem}

\noi
\begin{remark} \rm
We emphasize here that the convergence result in Theorem \ref{THM1} means that there exists a set $\O_0 \subset \O$ (where $(\O,\mathbb{P})$ is the underlying probability space on which the noise $\xi$ is defined) such that $\mathbb{P}(\O_0) = 1$ and 

\noi
\begin{align*}
\| u^\om_\eps - u^\om _0\|_{C_T H^s_x} \to 0,
\end{align*}
for any $\om \in \O_0$ as $\eps \to 0$ and with $T$ and $s$ as in Theorem \ref{THM1}. The subsequent convergence results are also proved in that fashion. It is achieved by showing that the stochastic objects are continuous in $(\eps,t) $ by using a bi-parameter Kolmogorov continuity criterion, see Lemma \ref{LEM:kol}, and by solving a fixed point argument for $(\eps,t) \mapsto v_\eps(t)$ \eqref{H1} in spaces of the form $C \big( [0,1 ] \times [0,T]; H^s(\T^2) \big)$ of functions both continuous in $\eps \in [0,1]$ and in time $t \in [0,T]$.

This is in contrast with the literature where such convergence results are obtained in probability (which only implies the almost sure convergence up to a subsequence of $\eps$'s going to 0). See for instance \cite{CF2, FHI}.
\end{remark}

In the non-singular case with a polynomial nonlinearity (for instance \eqref{SNLWa} with a colored noise in space or in one space dimension) then the result of last theorem can be extended to arbitrary large time intervals since the solutions are known to be global. See for instance \cite{CF1, CF2}. In our singular setting however, the convergence of Theorem \ref{THM1} cannot be established over longer times because of a lack of a global well-posedness theory for \eqref{rSNLW} and $k > 3$ (the solutions are known to be global in time for $k=3$; see Remark \ref{RMK:cubic} below). Since, for $k$ odd, the solution $u_0$ to \eqref{rSNLW} is known to exist globally in time by an argument of Mourrat and Weber \cite{MW1} (see also \cite{Trenberth}), we can show {\it asymptotic large time well-posedness} for \eqref{rSNLW} (for $\eps >0$). More precisely, since $u_\eps$ gets closer to $u_0$ as $\eps \to 0$, we can extend the existence time and the convergence of our local solutions $u_\eps$ over larger times as $\eps \to 0$. This is the purpose of the following theorem.

\noi
\begin{theorem}\label{THM:GWP}
Let $k \ge 2$ be an odd integer. Fix a \textup{(}deterministic\textup{)} target time $T>0$. Let $(\phi_0, \phi_1) \in \H^s(\T^2)$ for $s > \frac{2k-3}{2k-2}$. There exists an almost surely positive random variable $\eps_0 = \eps_0(\om) $ such that for each $\eps \in [0,\eps_0]$ the solutions $u_\eps$ and $u$ to \eqref{rSNLW} and \eqref{rSNLW1}, respectively, constructed in Theorem \ref{THM1} exist up to time $T$.

Furthermore, $\{u_\eps\}_{\eps \in (0,\eps_0]}$ converges almost surely to $u_0$ in $C \big( [0,T]; H^{-\s} (\T^2) \big)$ as $\eps \to 0$, for any $\s >0$.
\end{theorem}

\begin{remark} \label{RMK:cubic} \rm
Fix $s > \frac45$. In \cite{GKOT}, the authors proved that \eqref{rSNLW} for $k=3$ and $\eps =1$ is globally well-posed in $\H^{s}(\T^2)$. By modifying their argument, one could likely show global well-posedness in $\H^{s}(\T^2)$ for $\eps \in (0,1]$. Since the stochastic quantization equation \eqref{SNLW} for $\eps = 0$ is globally well-posed in $H^{s}(\T^2)$ as well; see \cite{MW1, Trenberth}, the Smoluchowski-Kramers approximation proved in Theorem \ref{THM1} holds globally in time, i.e. in $C\big( [0,T]; H^{-\s}(\T^2) \big)$ (for any $\s >0$) for any $T>0$ in the cubic case $k = 3$.
\end{remark}

We prove the following result for the sine-Gordon model: 

\noi
 \begin{theorem}\label{THM:sinGWP}
Let  $0<\be^2<2\pi$ and $s > 1 - \frac{\be^2}{4\pi}$. Let $(\phi_0,\phi_1) \in \H^{s }(\T^2)$. Fix $T>0$. Then, for each $\eps \in (0,1]$, there exists a solution $u_\eps$ to the stochastic sine-Gordon equation ~\eqref{rSdSG} with initial data given by $(\phi_0,\phi_1)$ and a solution $u_0$ to \eqref{rSdSG} \textup{(and $\eps=0$)} with initial data given by $\phi_0$ which belong to the class $C( [0,T];  H^{-\s}(\T^2))$, for any $\s>0$. 

Moreover, $\{ u_\eps \}_{\eps \in (0,1]}$ converges to $u_0$ in $ C( [0,T];  H^{-\s}(\T^2))$ as $\eps \to 0$.
 \end{theorem} 
 
 We conclude this section with a few remarks which complement the last statements. 
 
\noi
\begin{remark}\rm
Per usual when one uses the decomposition \eqref{H1}, the solutions constructed in Theorems \ref{THM1}, \ref{THM:GWP} and \ref{THM:sinGWP} are unique in classes of the form $\Psi_\eps + C([0,T]; H^{1-\s} (\T^2) )$ for any $\eps \in [0,\eps_0]$ for some appropriate $\eps_0 \in [0,1]$, $T>0$ and $\s >0$.
\end{remark}

\noi
\begin{remark}\label{RMK:FHI} \rm
In \cite{FHI}, Fukuizimi, Hoshino and Inui studied a similar convergence problem. They considered the following complex-valued equation for $n \in \N$:

\noi
\begin{align}
\begin{cases}
\eps^2 \dt^2 u_\eps +2 \al  \dt u_\eps + (1- \Dl) u_\eps + : \! u_\eps ^{n+1} \cj{u_\eps}^{n} \! :  \,  = 2 \sqrt{\Re(\al)}  \xi,    \\
(u_\eps, \eps \dt u_\eps)|_{t=0}=(\phi _0,\phi _1) \sim \rhoo_{n},
\end{cases}
\quad (x,t) \in \T^2 \times \R_+,
\label{FHI1}
\end{align}

\noi
$\eps > 0$ and $\al \in \C$ with $\Re(\al) > 0$ and $\Im(\al) \neq 0$ and where $\rhoo_n$ is the Gibbs measure associated with \eqref{FHI1} (see \cite{OT2} for a construction of $\rhoo_n$). The authors then looked at the limits $\eps \to 0$ and $\Im(\al) \to 0$, respectively known as the non-relativistic and ultra relativistic limits. As for the non-relativistic limit, the authors construct a set $\O_0$ with $\rhoo_n(\O_0) = 1$ such that $\{u_{j^{-1}}\}_{j \in \N}$, the family of solutions to \eqref{FHI1} with $\eps(j) = \frac1 j$, and with initial data $(\phi_0, \phi_1) \in \O_0$ converges to the following stochastic complex Ginzburg-Landau equation (which was studied in \cite{Trenberth}):

\noi
\begin{align}
\begin{cases}
2 \al  \dt u_0 + (1- \Dl) u_0 + : \! u_0 ^{n+1} \cj{u_0}^{n} \! :  \,  = 2 \sqrt{\Re(\al)}  \xi,    \\
u_0|_{t=0}= \phi_0,
\end{cases}
\quad (x,t) \in \T^2 \times \R_+,
\label{FHI2}
\end{align}

\noi
$\PP$-almost surely in $C \big([0,T]; H^{-\s}(\T^2 )\big) $ (for any $T, \s >0$). See \cite[Theorem 3]{FHI}. This result is similar to Theorems \ref{THM1}, \ref{THM:GWP} and \ref{THM:sinGWP}. There are however some differences with these results, which we describe in details below.

\smallskip

\noi
(i) In \cite{FHI}, the authors treated polynomial nonlinearities but we also study the sine nonlinearity - the so-called sine-Gordon model in \eqref{rSdSG}.

\smallskip

\noi
(ii) In \cite{FHI}, they considered Gibbs measure initial data. On the other hand, in Theorems \ref{THM1}, \ref{THM:GWP} and \ref{THM:sinGWP}, we consider deterministic initial data. In the polynomial case, we establish in Theorem \ref{THM:GWP}, a pathwise (asymptotic) global-in-time approximation despite the lack of a global well-posedness theory for our equations (except for the cubic case, see Remark \ref{RMK:cubic} above). Furthermore, for the sine nonlinearity \eqref{SdSG}, we prove a pathwise global-in-time convergence result in Theorem \ref{THM:sinGWP}. 


\smallskip

\noi
(iii) In \cite{FHI}, the authors treated the convergence problem in the coordinates $(u, \eps \dt u)$ with $O(1)$ initial data $(\phi_0, \phi_1)$. On the other hand, we work in the coordinates $(u, \dt u)$ with the same $O(1)$ initial data $(\phi_0, \phi_1)$, which would correspond to the scaled initial data $(\phi_0, \eps^{-1} \phi_1)$ in the coordinates $(u, \eps \dt u)$. This gives rise to some issue when controlling the time derivative of our solutions. See the discussion above Proposition \ref{PROP:LWP2}.

\smallskip

\noi
(iv) In \cite{FHI}, the convergence was proven for the sequence $\{ u_{j^{-1}} \} $ (i.e. along the discrete sequence $\eps (j) = \frac 1 j$). In the current work, we prove the convergence according to the continuous parameter $\eps \in (0,1]$. This is done by applying a bi-parameter Kolmogorov continuity criterion \cite[Theorem 2.1]{Baldi}. This appears to be also new in the literature related to the Smoluchowski-Kramers approximation. See for instance \cite{CF1, CF2}.
%
\end{remark}

The paper is organized as follows. In Section \ref{SEC:2} we review some useful results from Probability theory and some deterministic estimates. In Section \ref{SEC3}, we study the convergence of the linear flows and Duhamel integrals defined in \eqref{O1a}, \eqref{OP1} below.  In Section \ref{SEC2}, we turn our attention to the polynomial model \eqref{rSNLW} and present proofs of Theorems \ref{THM1} and \ref{THM:GWP}. Finally, in Section \ref{SEC6}, we look at the sine-Gordon model \eqref{rSdSG} and prove Theorem \ref{THM:sinGWP}.
\noi
\section{Notations and preliminary lemmas}
\label{SEC:2}

In this section, we introduce some notations and go over basic lemmas.

\noi
\subsection{Notations} Here, we introduce the notations that are used throughout the paper.

\smallskip

\noi
{\bf $\bullet$ Preliminary notations.} We write $ A \les B $ to denote an estimate of the form $ A \leq CB $. 
Similarly, we write  $ A \sim B $ to denote $ A \les B $ and $ B \les A $ and use $ A \ll B $ 
when we have $A \leq c B$ for small $c > 0$. We may write $A \les_\ta B$ for $A \leq C B$ with $C = C(\ta)$ if we want to emphasize the dependence of the implicit constant on some parameter $\ta$. Given two functions $f$ and $g$ on $\T^2$, we write 

\noi
\begin{align*}
f \approx g,
\end{align*}

\noi
if there exists two constants $c_1, c_2 \in \R$ such that $f(x) + c_1 \le g(x) \le f(x) + c_2$ for any $x \in \T^2 \setminus \{0\}$.

We introduce the set of parameters $I$ given by 

\noi
\begin{align}
I = (0,\infty) \setminus \Big\{ \frac{1}{2 \jb n}: n \in \Z^2 \Big\}.
\label{setI}
\end{align}

\noi
As discussed in Remark \ref{RMK:set}, the set $I$ is the natural set of $\eps$'s for which the operator $\D_\eps$ is a priori well-defined. In Lemma \ref{LEM:F3} below, we however show that $\eps \mapsto \D_\eps$ can be naturally extended to the half-line $(0, \infty)$.

Lastly, given a set $A$ we denote by $\ind_A$ the indicator function of $A$.

\smallskip

\noi
{\bf $\bul$ Fourier transforms.} We set for $n \in \Z^2$,

\noi
\begin{align*}
e_n(x) \deff \frac{1}{2\pi} e^{in \cdot x}.
\end{align*}

\noi
for the orthonormal Fourier basis in $L^2(\T^2)$. The spatial Fourier transform of the function $f$ is denoted by $\ft f$ (or by $\mathcal{F}(f)$) and is defined by

\noi
\begin{align*}
\ft f (n) = \int_{\T^2} f(x) e_n (x) dx, \quad n \in \Z^2.
\end{align*}

\noi
We also denote by $\F^{-1}(\{a_n\}_{n\in \Z^2})$ the inverse Fourier transform of the sequence $\{a_n\}_{n \in \Z^2}$ given by

\noi
\begin{align*}
\F^{-1}(\{a_n\}_{n\in \Z^2})(x) = \sum_{n\in\Z^2} a_n e_n(x).
\end{align*}

\noi
The convolution product is then given by 

\noi
\begin{align*}
(f \ast g)(x) = \frac{1}{2 \pi} \int_{\T^2} f(y) g(x-y) dy,
\end{align*}

\noi
so that $\Ft (f \ast g) = \Ft (f) \Ft(g)$.

Next, we recall the Poisson summation formula; see \cite{Grafakos1}. Let $f \in L^1 (\R^2)$ such that (i) $|f(x)| \les \jb{x}^{-2 - \eta}$ for some $\eta >0$ and any $x \in \R^2$, and (ii) $\sum_{n \in \Z^2} |\ft f (n)| < \infty$. Then, we have

\noi
\begin{align}
\sum_{n \in \Z^2} \ft f (n) e_n (x) = \sum_{\l \in \Z^2} f(x + 2 \pi \l),
\label{poisson}
\end{align}

\noi
for any $x \in \R ^2$.

\smallskip

\noi
{\bf $\bul$ Frequency projections.} In addition to the projectors $\Ph$, $\Pl$ and $\P_{N}$ defined in \eqref{proj1} and \eqref{proj1b} respectively, we also define $\Pll$ and $\Phh$ by

\noi
\begin{align}
\begin{split}
\Pll  f & =  \F^{-1} (\ind_{\jb{n} \leq (1+\ta) \cdot (2\eps)^{-1}} \ft f(n)), \\
\Phh  f & =  \F^{-1} (\ind_{\jb{n} > (1+\ta) \cdot (2\eps)^{-1}} \ft f(n)),
\end{split}
\label{proj2}
\end{align}

\noi
for some $ 0 < \ta \ll 1$, which is chosen to be much smaller than other fixed (i.e. not $\eps$) parameters.

\smallskip

\noi
{\bf $\bul$ Function spaces.} For $s \in \R$, the space $H^s(\T^2)$ denotes the usual $L^2 (\T^2)$-based Sobolev space and we define $\H^s(\T^2)$ by $\H^s(\T^2) := H^s(\T^2) \times H^{s-1}(\T^2) $. We also use shortcut notations such as $L^{\infty}_T H^s_x$ and $C_{\eps,T} H^s_x$ (for functions of the form $f = f(\eps,t,x)$) for $L^{\infty} \big( [0,T]; H^s(\T^2) \big)$ and $C \big( [0,1] \times [0,T]; H^s(\T^2) \big)$ respectively, etc.

\subsection{Preliminary results from stochastic analysis}


In this subsection, 
we recall some basic tools from probability theory and Euclidean quantum field theory
(\cite{Kuo, Nu, Shige, Simon}).
First, 
recall the Hermite polynomials $H_k(x; \s)$ 
defined through the generating function:
\begin{equation*}
F(t, x; \s) \stackrel{\text{def}}{=}  e^{tx - \frac{1}{2}\s t^2} = \sum_{k = 0}^\infty \frac{t^k}{k!} H_k(x;\s).
 \end{equation*}
	
\noi
For readers' convenience, we write out the first few Hermite polynomials:
\begin{align*}
\begin{split}
& H_0(x; \s) = 1, 
\quad 
H_1(x; \s) = x, 
\quad
H_2(x; \s) = x^2 - \s,   
\quad
 H_3(x; \s) = x^3 - 3\s x.
\end{split}
\end{align*}
	
\noi
Note that the Hermite polynomials verify the following standard identity:

\noi
\begin{align*}
H_k(x+y;\s) = \sum_{\l = 0}^k x^{k-\l} H_{\l}(y; \s). 
\end{align*}

Next, we recall the Wiener chaos estimate.
Let $(H, B, \mu)$ be an abstract Wiener space.
Namely, $\mu$ is a Gaussian measure on a separable Banach space $B$
with $H \subset B$ as its Cameron-Martin space.
Given  a complete orthonormal system $\{e_j \}_{ j \in \N} \subset B^*$ of $H^* = H$, 
we  define a polynomial chaos of order
$k$ to be an element of the form $\prod_{j = 1}^\infty H_{k_j}(\jb{x, e_j})$, 
where $x \in B$, $k_j \ne 0$ for only finitely many $j$'s, $k= \sum_{j = 1}^\infty k_j$, 
$H_{k_j}$ is the Hermite polynomial of degree $k_j$, 
and $\jb{\cdot, \cdot} = \vphantom{|}_B \jb{\cdot, \cdot}_{B^*}$ denotes the $B$-$B^*$ duality pairing.
We then 
denote the closure  of the span of
polynomial chaoses of order $k$ 
under $L^2(B, \mu)$ by $\mathcal{H}_k$.
The elements in $\H_k$ 
are called homogeneous Wiener chaoses of order $k$.
We also set
\[ \H_{\leq k} = \bigoplus_{j = 0}^k \H_j\]

\noi
 for $k \in \N$.

Let $L = \Dl -x \cdot \nabla$ be 
 the Ornstein-Uhlenbeck operator.\footnote{For simplicity, 
 we write the definition of the Ornstein-Uhlenbeck operator $L$
 when $B = \R^d$.}
Then, 
it is known that 
any element in $\mathcal H_k$ 
is an eigenfunction of $L$ with eigenvalue $-k$.
Then, as a consequence
of the  hypercontractivity of the Ornstein-Uhlenbeck
semigroup $U(t) = e^{tL}$ due to Nelson \cite{Nelson2}, 
we have the following Wiener chaos estimate
\cite[Theorem~I.22]{Simon}.
See also \cite[Proposition~2.4]{TTz}.

\begin{lemma}\label{LEM:hyp}
Let $k \in \N$.
Then, we have
\begin{equation*}
\|X \|_{L^p(\O)} \leq (p-1)^\frac{k}{2} \|X\|_{L^2(\O)}
 \end{equation*}
 
 \noi
 for any $p \geq 2$
 and any $X \in \H_{\leq k}$.
\end{lemma}

We recall the following property of Wick products; see \cite[Theorem $\1$.3]{Simon} for a proof.

\begin{lemma}\label{LEM:Wick}
Let $f$ and $g$ be jointly Gaussian random variables with variances $\s_f$
and $\s_g$.
Then, we have 
\begin{align*}
\E\big[ H_k(f; \s_f) H_m(g; \s_g)\big] = \dl_{km} k! \big\{\E[ f g] \big\}^k.
\end{align*}
\end{lemma}

We now state a two-dimensional version of the usual Kolmogorov continuity criterion for readers' convenience. Before doing so, let us recall that given an index set $T$, we say that a stochastic process $\{ \wt X_t \}_{t \in T}$ is a modification of another process $\{ X_t \}_{t \in T}$ (both defined on the same probability space $(\O, \PP)$) if for any $t \in T$, we have $\PP \big( X_t = \wt X_t \big) =1$.

\noi
\begin{lemma}[bi-parameter Kolmogorov continuity criterion]\label{LEM:kol}
Let $I \times J$ be two intervals of $\R$. Let $\{ X(\eps,t) \}_{(\eps, t) \in I \times J}$ be a stochastic process defined on a probability space $(\O, \PP)$ and taking values in a complete metric space $(E,d)$. Assume that there exist real numbers $q, \al >0$ and $A >0$ such that 

\noi
\begin{align*}
\E \big[ d \big( X(\eps_1, t_1), X(\eps_2, t_2)  \big)^q \big] \le A   \| (\eps_1 - \eps_2, t_1 - t_2) \|_2^{2 + \al},
\end{align*}

\noi
for any $(\eps_1, \eps_2) \in I^2$ and $(t_1,t_2) \in J^2$ and where $\| \cdot \|_2$ denotes the canonical Euclidean norm on $\R^2$. Then, there exists a modification $\wt X$ of $X$ whose sample paths are H\"older continuous on $I \times J$ with exponent $\dl$ for any $0 < \dl = \dl(\al,q)$. Namely, we have

\noi
\begin{align*}
d \big( X^\om (\eps_1, t_1), X^\om (\eps_2, t_2)  \big) \les_{\om} A^{\g}  \| (\eps_1 - \eps_2, t_1 - t_2) \|_2^{\dl},
\end{align*}

\noi
for some $\g >0$ and for any $\om \in \O_0$, where $\O_0 \subset \O$ is a full $\PP$-probability set. Furthermore, we have the following tail estimate:

\noi
\begin{align*}
\PP \bigg( \frac{d \big( X^\om (\eps_1, t_1), X^\om (\eps_2, t_2)  \big)}{A^{\g}  \| (\eps_1 - \eps_2, t_1 - t_2) \|_2^{\dl}} > M  \bigg) \les M^{-q}.
\end{align*}
\end{lemma}

\noi
The proof of Lemma \ref{LEM:kol} follows from a slight modification of the argument in the proof of \cite[Theorem 2.1]{Baldi}. See \cite{Bass} for the proof of the tail estimate.

\subsection{Deterministic estimates}
We recall the following product estimates.
See \cite{GKO1} for a proof.

\begin{lemma}\label{LEM:bilin}
 Let $0\le s \le 1$.

\smallskip

\noi
\textup{(i)} Suppose that 
 $1<p_j,q_j,r < \infty$, $\frac1{p_j} + \frac1{q_j}= \frac1r$, $j = 1, 2$. 
 Then, we have  
\begin{equation*}  
\| \jb{\nb}^s (fg) \|_{L^r(\T^2)} 
\les \Big( \| f \|_{L^{p_1}(\T^2)} 
\| \jb{\nb}^s g \|_{L^{q_1}(\T^2)} + \| \jb{\nb}^s f \|_{L^{p_2}(\T^2)} 
\|  g \|_{L^{q_2}(\T^2)}\Big).
\end{equation*}

\smallskip

\noi
\textup{(ii)} 
Suppose that  
 $1<p,q,r < \infty$ satisfy the condition:
$\frac1p+\frac1q\leq \frac1r + \frac{s}2 $.
Then, we have
\begin{align*}
\big\| \jb{\nb}^{-s} (fg) \big\|_{L^r(\T^2)} \les \big\| \jb{\nb}^{-s} f \big\|_{L^p(\T^2) } 
\big\| \jb{\nb}^s g \big\|_{L^q(\T^2)}.  
\end{align*}
%
%
\end{lemma}

Note that
while  Lemma \ref{LEM:bilin} (ii) 
was shown only for 
$\frac1p+\frac1q= \frac1r + \frac{s}2 $
in \cite{GKO1}, 
the general case
$\frac1p+\frac1q\leq \frac1r + \frac{s}2 $
follows from the inclusion $L^{r_1}(\T^2)\subset L^{r_2}(\T^2)$
for $r_1 \geq r_2$.

\smallskip

We record the following elementary result on the regularizing effect of the heat semi-group in Sobolev spaces whose proof is a simple consequence of Plancherel's identity.
\noi
\begin{lemma}\label{LEM:heat}
Let $\al, \be \in \R$ with $\al \ge \be$ and $t >0$. We have

\noi
\begin{align*}
\| P_0(t) f \|_{H^\al (\T^2) } \les t^{- \frac{\al- \be}{2}} \| f \|_{H^\be(\T^2)},
\end{align*}
where $P_0$ is as in \eqref{Lheat}.
\end{lemma}

In the current $L^2_x$-setting, the proof of Lemma \ref{LEM:heat} is an easy application of Plancherel's identity. We include it here for completeness.

\begin{proof} By Plancherel's identity, we have the following bound:

\noi
\begin{align*}
\| P_0(t) f \|_{H^\al (\T^2) } & = \| \jb n ^{\al} e^{-t \jb n ^2} \ft f (n) \|_{\l ^2_n} \\
& \les  t^{- \frac{\al- \be}{2}} \| \jb n ^{\al} \jb n ^{\be - \al} \ft f (n) \|_{\l ^2_n} = t^{- \frac{\al- \be}{2}} \| f \|_{H^\be(\T^2)},
\end{align*}

\noi
where we used the bound $e^{-x} \les x^{- \frac{\al - \be}{2}}$ for any $x >0$.
\end{proof}

\section{Convergence of the deterministic objects}\label{SEC3}

In this section, we study the convergence as $\eps \to 0$ and prove bounds for various deterministic and stochastic objects depending on $\eps \ge 0$. This is needed to prove Theorems \ref{THM1} and \ref{THM:GWP}. More precisely, we first study the deterministic linear objects \eqref{OP1} and \eqref{O1a} defined below. We then proceed with the construction of the Wick powers \eqref{Herm1} uniformly in $\eps \in [0,1]$ and $N \in \N$.

We introduce the following Duhamel operators for $\eps > 0$:

\noi
\begin{align}
\begin{split}
& \I_\eps(F)(t) = \int_{0}^t \eps^{-2}\D_\eps(t -t') F(t') dt' \\
& \I_0(F)(t) = \int_{0}^t P_0(t-t') F(t') dt',
\end{split}
\label{OP1}
\end{align}

\noi
for a space-time function $F$ and $t \in \R$, with $P_0$ and $\D_\eps$ as in \eqref{Lheat} and \eqref{L4}.

In this section, we study the convergence of the linear flow \eqref{L5} and the Duhamel operator \eqref{OP1}.

Let $P_\eps$ be the linear operator associated to the homogeneous linear solution \eqref{X1}. For two distributions $\phi_0$, $\phi_1$, we have

\noi
\begin{align}
P_{\eps}(t)(\phi_0,\phi_1) :=  (\eps^{-2} + \dt)  \D_\eps(t) \phi_0 + \D_\eps(t)\phi_1,
\label{O1a}
\end{align}

\noi
for $t \ge 0$. Hereafter, we identify the operator $P_0$ with $(P_0,0)$ defined by 

\noi
\begin{align*}
(P_0,0)(\phi_0, \phi_1)(t) := P_0(t) \phi_0,
\end{align*}

\noi
for $t\ge 0$.

The main result of this section is the following proposition which provides bounds for both $P_\eps$ and $\I_\eps$.

\noi
\begin{proposition}\label{PROP:det}
Let $s \in \R$ and $T > 0$. For any $\eps \in [0,1]$, we have the following bounds:

\noi
\begin{align}
\begin{split}
& \sup_{\eps \in [0,1]} \|  P_{\eps} (\phi_0, \phi_1) \|_{C_T H^s_x} \les  \|  (\phi_0, \phi_1) \|_{\H^s_x}, \\
& \sup_{\eps \in [0,1]} \| \I_\eps(F) \|_{C_T H^s_x} \les T^\frac12 \|  F \|_{L^{\infty}_T H^{s-1}_x},
\end{split}
\label{OPunif}
\end{align} 

\noi
for any functions $\phi_0, \phi_1$ and $F$. Moreover, for any $0 < \ta \ll 1$, we have

\noi
\begin{align}
\begin{split}
 \| (P_\eps - P_0)(\phi_0, \phi_1) \|_{C_T H^s_x} & \les \eps^{\frac{\ta}{2}} \|  (\phi_0, \phi_1) \|_{\H^{s+\ta}_x}, \\
 \| (\I_\eps- \I_0)(F)  \|_{C_T H^s_x} & \les T^{\frac12} \eps^{\frac{\ta}{2}} \|  F \|_{L^{\infty}_T  H^{s-1+\ta}_x}.
\end{split}
\label{detconv}
\end{align}

\noi
for any functions $\phi_0, \phi_1$ and $F$.
\end{proposition}

We deduce from Proposition \ref{PROP:det}, the following bounds on the operators $\{ P_\eps \}_{\eps \in [0,1]}$ and $\{ \I_\eps \}_{\eps \in [0,1]}$, viewed as continuous objects in $\eps \in [0,1]$.

\noi
\begin{corollary}\label{COR:det}
Let $s \in \R$, $T > 0$ and $\ta >0$. We have the following bounds:

\noi
\begin{align}
& \| P_\eps(t) (\phi_0,\phi_1) \|_{C_{\eps,T} H^s_x } \les \|  (\phi_0, \phi_1) \|_{\H^{s}_x}, \label{C11} \\
& \| \I_\eps(t) \big(F(\eps, \cdot) \big) \|_{C_{\eps,T} H^s_x } \les T^{\frac12}  \| F \|_{C_{\eps,T}H^{s-1}_x}, \label{C12}
\end{align}

\noi
for any functions $\phi_0, \phi_1$ and $F$.
\end{corollary}


\noi
\begin{remark}\rm
By combining using the bounds in Proposition \ref{PROP:det} it is easy to prove the convergence of the solution of the nonhomegeneous linear damped wave equation \eqref{X4} to that of linear heat equation \eqref{X5} as $\eps \to 0$, provided the forcing term $F$ is smooth enough. This makes the formal derivation of the Smoluchowski-Kramers approximation in Section \ref{SEC1} rigorous.
\end{remark}

We prove several useful lemmas first and postpone the proof of Proposition \ref{PROP:det} and Corollary \ref{COR:det} to the end of the section.

Let $\{ \ft{\D_\eps} (n,t) \}_{n \in \Z^2}$ be the symbol associated to the multiplier $\D_\eps(t)$ defined in \eqref{L4}. We define $ \eta \in \R^2 \mapsto  \ft{\D_\eps} (\eta,t) $ its natural extension to $\R^2$ given by

\noi
\begin{align}
\ft{\D_\eps} (\eta,t) := \begin{cases}
 e^{- \frac{t}{2 \eps^2}} \frac{\sinh( \ld_\eps(\eta) t ) }{\ld_\eps(\eta)} \, & \text{if $\jb \eta \le (2 \eps)^{-1} $} \\
e^{- \frac{t}{2 \eps^2}} \frac{\sin( \ze_\eps(\eta) t ) }{\ze_\eps(\eta)} \,  &\text{if $\jb \eta > (2 \eps)^{-1} $}\end{cases},
\label{De1}
\end{align}

\noi
where $ \eta \mapsto \ld_\eps(\eta) $ and $\eta \mapsto \ze_\eps(\eta)$ are the obvious extensions to $\R^2$ of the functions $\ld_\eps$ and $\ze_\eps$ defined in \eqref{ld2} and \eqref{ld3}, respectively. 

From \eqref{De1}, \eqref{ld2} and \eqref{ld3}, it might seem that $(\eps, t, \eta) \mapsto \ft \D_\eps (\eta, t)$ is ill-defined on the hypersurface $\{ (\eps, \eta ) \in (0, \infty) \times \R^2  : \eps = \frac{1}{2 \jb \eta} \}$. We however prove in the next lemma that $(\eps, t, \eta) \mapsto \ft \D_\eps (\eta, t)$ can in fact be extended to a smooth function on $(0,\infty) \times \R_+ \times \R^2$ and provide a control on its derivatives in $\eps$ and $\eta$.

For an integer $p \ge 1$ and a multi-index $\al \in \{1,2\}^p$, we denote by $|\al| = p$ its length. 

\noi
\begin{lemma}\label{LEM:F3}
Recall the definition of the set $I$ in \eqref{setI}. The function $(\eps,t, \eta) \in I \times \R_+ \times \R^2 \mapsto \ft{\D_\eps}(\eta,t) $ can be extended to a $C^\infty$ function on $(0,\infty) \times \R_+ \times \R^2$. Moreover, we have the following bound:

\noi
\begin{align}
&  | \partial^{\al}_{\eta} \ft{  \D_\eps }(\eta,t) | \les_{\al} e^{- \frac{t}{2 \eps^2}} t^{|\al| +1} \eps^{- |\al|} \sum_{p = 1}^{|\al|} (1 + |t \eps^{-1} \eta|^p) \big( \ind_{\jb \eta \le (2\eps)^{-1}}  \, e^{t \ld_\eps(\eta)}  + \ind_{\jb \eta > (2\eps)^{-1}} \big) \label{fun14} \\
& | \partial_\eps \ft{  \D_\eps }(\eta,t)  | \les  \eps^{-5} \jb \eta ^2,  \label{B1}
\end{align}

\noi
for any $\eps \in (0,\infty)$, $\eta \in \R^2$, $t \ge 0$ and multi-index $\al$. We highlight that the implicit constants in \eqref{fun14} and \eqref{B1} are uniform in the parameters $(\eps, t, \eta) \in I$.
\end{lemma}

We refer the reader to \cite[Lemma 2.1.5]{thesis} for a proof of Lemma \ref{LEM:F3}.

In the next lemma, we prove bounds on $\ft \D_\eps(n,t)$ \eqref{L4} (and its time derivative) which are uniform in $\eps >0$. By Plancherel's identity, this will be sufficient to obtain uniform bounds as \eqref{OPunif} (at least for $\eps >0$).

\noi
\begin{lemma}\label{LEM:F1}
Fix $0 < \ta \ll 1$. We have the following bounds:

\noi
\begin{align}
\eps^{- 2} | \ft{\D_\eps} (n,t) | \les \begin{cases}
e^{- \ta t \jb n ^2}   &  \textup{for $\jb n \leq (1+\ta)(2 \eps)^{-1}$} \\
 e^{- \frac{t}{2 \eps^2} } \eps^{-1} \jb{n}^{-1}  & \textup{otherwise
 },
\end{cases}
\label{mul1}
\end{align}

\noi
\begin{align}
| \dt \ft{\D_\eps} (n,t) | \les \begin{cases}
e^{- \ta t \jb n ^2}   &  \textup{for $\jb n \leq (1+\ta)(2 \eps)^{-1}$} \\
 e^{- \frac{t}{2 \eps^2} }   & \textup{otherwise
 },
\end{cases}
\label{mul2}
\end{align}

\noi
with implicit constants independent of $\eps > 0$ and $t > 0$. 
\end{lemma}

We infer from \eqref{mul1} and \eqref{mul2} that both $\eps^{- 2} \ft{\D_\eps} (n,t)$ and $\dt \ft{\D_\eps} (n,t)$ behave like the heat propagator $P_0$ in the low-frequency regime $\jb n \les \eps^{-1}$. However, in the high-frequency regime $\jb n \gg \eps^{-1}$, they essentially behave like (a scaled version of) the damped propagator $\D_{\eps = 1} (n,t)$ (or its time derivative).

\begin{proof} By the smoothness of the map $(\eps,t, \eta) \mapsto \ft{\D_\eps}(\eta,t)$ discussed in Lemma \ref{LEM:F3} above, it suffices to prove \eqref{mul1} and \eqref{mul2} with $\jb n \neq (2 \eps)^{-1}$, which we assume in the remaining of the proof.

We first prove \eqref{mul1}. Fix $t \ge 0$ and $\eps >0$. By \eqref{ld2} and \eqref{L3}, we have

\noi
\begin{align}
\eps^{- 2} \ft{\D_\eps} (n,t) = \eps^{-2} e^{- \frac{t}{2 \eps^2}}  \frac{\sinh (\ld_\eps(n) t)}{\ld_\eps(n)},
\label{M1}
\end{align}

\noi
in the regime $\jb n < (2 \eps)^{-1}$. Note that by the inequality $\sqrt{1-x} \leq 1 - \frac x2$ for $0 \le x \le 1$, we get 

\noi
\begin{align}
e^{- \frac{t}{2 \eps^2}} e^{\ld_\eps(\eta)t} \le e^{- \jb \eta ^2 t}.
\label{upheat}
\end{align}

\noi
Hence, by \eqref{upheat}, we have that

\noi
\begin{align}
e^{- \frac{t}{2 \eps^2}} \sinh ( t \ld_\eps (n)) \les e^{- \jb n ^2 t},
\label{upsinh}
\end{align}

\noi
for $\jb n \leq (2\eps)^{-1}$. By using the inequality $\ld_\eps(n) \ges_\ta \eps^{-2} $ for $\jb n  \le (1 - \ta) (2\eps)^{-1}$, we then obtain by \eqref{upsinh} and \eqref{M1}

\noi
\begin{align}
\eps^{- 2} | \ft{\D_\eps} (n,t) | \les e^{-t \jb{n}^2}.
\label{M2}
\end{align}

\noi
for $\jb n \le (1 - \ta) (2\eps)^{-1}$.

We now estimate $ \eps^{- 2} \ft{\D_\eps} (n,t)$ for $(1 - \ta) (2\eps)^{-1} < \jb n < (2 \eps)^{-1}$.

\medskip
\noi
$\bullet$
{\bf Case 1:} $\ld_\eps(n)t \le 1$.
\quad
In this regime, by using the bounds $| \!  \sinh(x) | \les |x| $ for $0 \le x \le 1$ and $e^{-y} \les y^{-1}$ for $y > 0$ and \eqref{M1} with $\jb n < (2\eps)^{-1}$, we get

\noi
\begin{align}
\begin{split}
\eps^{- 2} | \ft{\D_\eps} (n,t) | & \les \eps^{-2 } t \cdot  e^{- \frac{t}{2 \eps^2}}  \\
& \les e^{- \frac{t}{4 \eps^2}} \les e^{- t \jb n ^2}.
\end{split}
\label{M3}
\end{align}

\medskip
\noi
$\bullet$
{\bf Case 2:} $\ld_\eps(n)t > 1$.
\quad
In this case, we note that we have $\sqrt{1-x} \le 1 - \ta (1 + \frac x 2)$ for $1- \ta \le x \le 1$ and $0 < \ta \ll 1$ which implies 

\noi
\begin{align}
e^{- (1-\ta) \frac{t}{2 \eps^2}} \sinh ( \ld_\eps (n)t) \les e^{- \ta  t \jb n ^2 },
\label{upsinh2}
\end{align}

\noi
for $ (1- \ta)  (2\eps)^{-1} < \jb n < (2\eps)^{-1}$. Using \eqref{M1}, \eqref{upsinh2} and the inequality $e^{-y} \les y^{-1}$ for $y > 0$, we then get

\noi
\begin{align}
\begin{split}
\eps^{- 2} | \ft{\D_\eps} (n,t) | & \les \eps^{-2 } t \cdot e^{- \ta \frac{t}{2 \eps^2} } \cdot e^{- \ta t \jb{n}^2} \\
& \les e^{- \ta t \jb{n}^2}.
\end{split}
\label{M4}
\end{align}

\noi
Hence, from \eqref{M3} and \eqref{M4} we deduce 

\noi
\begin{align}
\eps^{- 2} | \ft{\D_\eps} (n,t) | \les e^{- \ta t \jb{n}^2}
\label{M5}
\end{align}

\noi
for $ (1- \ta)  (2\eps)^{-1} < \jb n < (2\eps)^{-1}$.

For $\jb n > (2\eps)^{-1}$, we have from \eqref{L3} with \eqref{ld3} and \eqref{L4}:

\noi
\begin{align}
\eps^{- 2}  \ft{\D_\eps} (n,t)  = \eps^{-2} e^{- \frac{t}{2 \eps^2}}  \frac{\sin (\ze_\eps(n) t)}{\ze_\eps(n)}.
\label{M6}
\end{align}

\noi
By using the inequalities $| \! \sin(x) |\les |x|$ for $x \in \R$ and $e^{-y} \les y^{-1}$ for $y > 0$, \eqref{M6} and $\jb n \sim \frac{1}{2 \eps^2}$, we have

\noi
\begin{align}
\begin{split}
\eps^{- 2} | \ft{\D_\eps} (n,t) | & \les \eps^{-2} t \cdot e^{ - \frac{t}{2 \eps^2}} \\
& \les e^{ - \frac{t}{4 \eps^2}} \les e^{- \frac{t}{10} \jb n ^2}.
\end{split}
\label{M7}
\end{align}

\noi
for $ (2 \eps)^{-1} < \jb n \le (1 + \ta) (2 \eps)^{-1} $.

For $\jb n > (1+\ta) \cdot (2 \eps)^{-1}$, we have $\ze_\eps(n) \ges_\ta  \eps^{-1} \jb n$. Thus, we get the following bound from from \eqref{M6},

\noi
\begin{align}
\eps^{- 2} | \ft{\D_\eps} (n,t) | \les  e^{- \frac{t}{2 \eps^2} } \eps^{-1} \jb{n}^{-1}
\label{M8}
\end{align}

\noi
Collecting \eqref{M2}, \eqref{M5}, \eqref{M7} and \eqref{M8} yields \eqref{mul1} 

We now prove \eqref{mul2}. Note that from \eqref{M1} and \eqref{M6}, we have 

\noi
\begin{align}
\dt \ft{\D_\eps}(n,t) = - \frac{1}{2 \eps^2} \ft{\D_\eps}(n,t) + e^{- \frac{t}{2 \eps ^2}} \dt \ft{S_\eps}(n,t)
\label{M9}
\end{align}

\noi
where $\{ \ft{S_\eps}(n,t) \}_{n \in \Z^2}$ is the Fourier symbol associated to $S_\eps(t)$ defined in \eqref{L3}.

By \eqref{mul1}, it suffices to estimate the contribution of $e^{- \frac{t}{2 \eps ^2}} \dt \ft{S_\eps}(n,t) $. By \eqref{L3}, we have

\noi
\begin{align}
\dt \ft{S_\eps}(n,t)  = \cosh( \ld_\eps(n) t) \ind_{\jb n < (2 \eps)^{-1}} + \cos( \ze_\eps(n) t) \ind_{\jb n > (2 \eps)^{-1}}.
\label{M10}
\end{align}

\noi
Hence, by \eqref{upheat}, we get

\noi
\begin{align}
e^{- \frac{t}{2 \eps ^2}} \dt \ft{S_\eps}(n,t) \les \begin{cases} e^{-t \jb n ^2} & \text{ if $\jb n \le (2 \eps)^{-1}$} \\
 e^{-\frac{t}{2\eps^2}} & \text{ otherwise},
\end{cases}
\label{M11}
\end{align}

\noi
which concludes the proof of \eqref{mul2}.	
\end{proof}

In the next lemma, we study the behavior near $\eps = 0$ of the symbols $\ft \D_{\eps}(n,t)$. 

\begin{lemma}\label{LEM:F2}
Fix $0 < \ta \ll1$. The following estimates hold:

\noi
\begin{align}
\big| \eps^{- 2}  \ft{\D_\eps} (n,t) - e^{-t \jb n ^2} \big| \ind_{\jb n \les \eps^{- 1 + \ta}} \les e^{- \frac{t}{2 \eps^2}} + \eps^{2 \ta}  e^{- \frac t 2  \jb n ^2},
\label{mul3}
\end{align}

\noi
and

\noi
\begin{align}
\big| (\eps^{- 2} + \dt) \ft{\D_\eps} (n,t) - e^{-t \jb n ^2} \big| \ind_{\jb n \les \eps^{- 1 + \ta}} \les  \eps^{2 \ta}  e^{- \frac t 2 \jb n ^2},
\label{mul4}
\end{align}

\noi
for any $t \ge 0$.
\end{lemma}

\begin{proof} Fix $t\ge 0$ and $\jb n \les \eps^{- 1 + \ta}$. From \eqref{M1} with \eqref{ld2}, we write

\noi
\begin{align}
\eps^{- 2} \ft{\D_\eps} (n,t) = e^{- \frac{t}{2 \eps^2}}  \frac{ e^{\ld_\eps(n) t} - e^{-\ld_\eps(n) t}  }{ \sqrt{1 - 4 \eps^2 \jb n ^2} } =: \1 - \II.
\label{N1}
\end{align}

\noi
Since $\ld_\eps(n)$ is non-negative, we have 

\noi
\begin{align}
| \II | \les e^{- \frac{t}{2 \eps^2}}.
\label{N2}
\end{align}

\noi
Furthermore, we get, using \eqref{ld1}, the inequality $\Ld^+_{\eps}(n)  + \jb n ^2 \le 0$, the asymptotic expansion of $\Ld^+_\eps(n)$ in \eqref{ld4}, the mean value theorem with the inequality $e^{-y} \les y^{-1}$ for $y>0$, we deduce

\noi
\begin{align}
\begin{split}
| \1 - e^{- t \jb n^2} | & \les | \big( \frac{1}{\sqrt{1 - 4 \eps^2 \jb n ^2} } - 1 \big) e^{\Ld^+_\eps(n) t} | + | \big( e^{\Ld^+_\eps(n) t} - e^{- t \jb n^2} \big)  |, \\
& \les \eps^{ 2\ta} e^{-t \jb n ^2} +  t \jb n ^4 \eps^2  e^{- t \jb n ^2} \les \eps^{2 \ta}   e^{- \frac t2  \jb n ^2}.
\end{split}
\label{N3SK}
\end{align}

\noi
Putting \eqref{N1}, \eqref{N2} and \eqref{N3SK} together gives \eqref{mul3}.

We now prove \eqref{mul4}. By using \eqref{M1}, \eqref{ld1}, \eqref{ld2} and \eqref{M9}, we write

\noi
\begin{align}
(\eps^{-2} + \dt) \ft{\D_\eps}(n,t) = \mathcal{P}_\eps (n,t) + \mathcal{R}_\eps (n,t),
\label{N4SK}
\end{align}

\noi
with
\begin{align*}
& \mathcal{P}_\eps (n,t) := \frac{e^{\Ld_\eps ^+(n)t}}{\sqrt{1 - 4 \jb{n}^2 \eps^2}} \\
& \mathcal{R}_\eps (n,t) := \Big( 1 - \frac{1}{\sqrt{1 - 4 \jb{n}^2 \eps^2}} \Big) \frac{e^{\Ld^+_\eps(n)t}}{2} + \Big(  1 -  \frac{1}{\sqrt{1 - 4 \jb{n}^2 \eps^2}}  \Big) \frac{e^{\Ld^-_\eps(n)t}}{2}.
\end{align*}

\noi
By arguing as in \eqref{M3}, we find

\noi
\begin{align}
\begin{split}
& | \mathcal{P}_\eps (n,t)  - e^{-t \jb n ^2}| \les  \eps^{2 \ta}   e^{- \frac t2  \jb n ^2}, \\
& | \mathcal{R}_\eps (n,t) | \les \eps^{ 2 \ta} e^{-t \jb n ^2}.
\end{split}
\label{N5}
\end{align}

\noi
Thus, \eqref{mul4} follows from \eqref{N4SK} and \eqref{N5}.

\end{proof}


In the next lemma, we deduce bounds in Sobolev spaces for the linear propagator $\D_\eps$ from the corresponding estimates at the level of the Fourier symbol $\ft{\D_\eps} (n,t)$ in Lemma \ref{LEM:F1} and Lemma \ref{LEM:F2}.

\noi
\begin{lemma}\label{COR:OP}
Fix $0 < \ta \ll 1$. Recall the definitions of $\Pll$ and $\Phh$ in \eqref{proj2}. The following inequalities hold:

\smallskip
\noi
\textup{(i)} \textup{(}parabolic smoothing\textup{)} Let $\al, \be \in \R$ with $\al \geq \be$ and $t>0$. We have

\noi
\begin{align*}
\big\| \eps^{-2} \Pll \D_\eps(t) f \big\|_{H_x^\al} \les t^{- \frac{\al-\be}{2}}  \| f\|_{H^\be_x},
\end{align*}

\noi
for any function $f$ and with an implicit constant independent of $\eps >0 $ and $t>0$.

\smallskip
\noi
\textup{(ii)} \textup{(}wave smoothing\textup{)} Let $s \in \R$, $\g \in \{ 0,1 \}$ and $t>0$. We have

\noi
\begin{align*}
\big\| \eps^{-2} \Phh \D_\eps(t) f \big\|_{H^s_x} \les e^{-\frac{t}{10 \eps^2}} t^{- \frac{\g}{2}} \| f\|_{H^{s-\g}_x},
\end{align*}

\noi
for any function $f$ and with an implicit constant independent of $\eps >0$ and $t > 0$.

\smallskip
\noi
\textup{(iii)} Let $s \in \R$, $0 \le \g \le 1$ and $t > 0$. We have

\noi
\begin{align*}
\| \D_\eps(t) f \|_{H^s_x} \les \eps^{2 - \g} \| f \|_{H^{s-\g}_x}
\end{align*}

\noi
for any function $f$ and with an implicit constant independent of $\eps >0$ and $t > 0$.

\smallskip
\noi
\textup{(iv)} Let $s \in \R$. We have,

\noi
\begin{align*}
\sup_{\eps, t >0} \big\|  \dt \D_\eps(t) f \big\|_{H^s_x} \les \|f\|_{H^s_x},
\end{align*}

\noi
for any function $f$.

\smallskip
\noi
\textup{(v)} Let $\al, \be \in \R$ with $\al \geq \be$, and $t >0$. We have: 

\noi
\begin{align*}
& \big\|  \P_{\le \eps^{-1+\ta}} \big( (\eps^{-2} + \dt)  \D_\eps (t) - P_0(t)  \big) f \big\|_{H^\al_x} \les \eps^{2\theta}  t^{- \frac{\al-\be}{2}} \|f\|_{H^\be_x},
\end{align*}

\noi
for any function $f$ and with an implicit constant independent of $\eps >0$ and $0 < t \le T$.
\end{lemma}

\begin{proof}
Items (i), (iv) and (v) are direct consequences of \eqref{mul1}, \eqref{mul2} and \eqref{mul4}, respectively (along with Lemma \ref{LEM:heat}). We now look at (ii). If $\g=0$, then (ii) comes directly from \eqref{mul1}. If $\g = 1$, then from \eqref{mul1}, we have

\noi
\begin{align*}
 \big\| \eps^{-2} \Phh \D_\eps(t) f \big\|_{H^\s_x} & \les e^{-\frac{t}{2\eps^2}} \eps^{-1} \| f\|_{H^{s-1}_x} \\
 & \les e^{-\frac{t}{10 \eps^2}} t^{- \frac{1}{2}} \| f\|_{H^{s-1}_x},
\end{align*}

\noi
where we used the inequality $e^{-y} \les y^{-\frac12}$ for $y >0$. This proves (ii) for $\g = 1$. From (i) and (ii) with $\g = 0$, and by interpolation, (iii) follows from the bound

\noi
\begin{align}
\| \D_\eps(t) f \|_{H^s_x} \les \eps \| f \|_{H^{s-1} _x},
\label{lin10}
\end{align}

\noi
which we now prove. From (i), \eqref{proj2} and the restriction $\jb n \les \eps ^{-1}$, we get

\noi
\begin{align}
\| \Pll \D_\eps (t) f \|_{H^s_x} & \les \eps^2 \| \Pll f \|_{H^s_x} \notag \\
& \les \eps \| f \|_{H^{s-1}_x}. \label{lin11}
\end{align}

\noi
Furthermore, from \eqref{mul1} with \eqref{proj2}, we have

\noi
\begin{align}
\| \Phh \D_\eps (t) f \|_{H^s_x} & \les \eps^2 e^{-\frac{t}{2 \eps^2} } \eps^{-1} \| f \|_{H^{s-1}_x} \notag \\
& \les \eps \| f \|_{H^{s-1}_x}. \label{lin12}
\end{align}

\noi
Hence, combining \eqref{lin11} and \eqref{lin12} gives \eqref{lin10}.
\end{proof}

\noi
\begin{remark}\label{RMK:mulstr} \rm
From \eqref{mul2}, we actually get the following stronger bounds for $\al , \be, s \in \R$ with $\al \ge \be$:

\noi
\begin{align*}
& \| \Pll \dt \D_\eps(t) f \|_{H_x^{\al}} \les t^{-\frac{\al - \be}{2}} \|f\|_{H^\be _x}, \\
& \| \Phh \dt \D_\eps(t) f \|_{H_x^{s}} \les e^{- \frac{t}{2 \eps^2}} \| f \|_{H^s_x},
\end{align*}

\noi
for any $\eps, t > 0$ and function $f$.
\end{remark}

We are now ready to prove Proposition \ref{PROP:det}.
\noi
\begin{proof}[Proof of Proposition \ref{PROP:det}]
Let $s \in \R$ and $T> 0$ and fix smooth functions $\phi_0$ and $\phi_1$. We first prove \eqref{OPunif}. By Lemma \ref{LEM:heat}, $P_0$ \eqref{Lheat} cleary satisfies the bound

\noi
\begin{align}
\| P_0 (\phi_0, \phi_1) \|_{ L^{\infty}_T H^s_x} \les \| (\phi_0,\phi_1)   \|_{\H ^s_x}.
\label{pf1}
\end{align}

\noi
Hence, it suffices to prove

\noi
\begin{align}
\sup_{\eps \in (0,1]} \|  P_{\eps} (\phi_0, \phi_1) \|_{L^{\infty}_T H^s_x} \les  \|  (\phi_0, \phi_1) \|_{\H^s_x}.
\label{pf2}
\end{align}

\noi
By \eqref{O1a}, Lemma \ref{COR:OP} (i), (ii), (iii) and (iv) with \eqref{proj2}, we have

\noi
\begin{align}
\begin{split}
& \sup_{\eps \in (0,1]} \| ( \eps ^{-2 } + \dt) \D_\eps (t) \phi_0  \|_{L^{\infty}_T H^s_x} \les \|\phi_0\|_{H^s_x} \\
& \sup_{\eps \in (0,1]}  \| \D_\eps (t) \phi_1  \|_{L^{\infty}_T H^s_x} \les \eps \|\phi_1 \|_{H^{s-1}_x}
\end{split}
\label{pf3}
\end{align}

\noi
The continuity in time of $P_\eps(\phi_0,\phi_1)$ for some fixed $\eps >0 $ and $(\phi_0,\phi_1) \in \H^s(\T^2)$ follows from the dominated convergence theorem and \eqref{pf3}. Combining \eqref{pf3} and \eqref{O1a} gives \eqref{pf2}. This concludes the proof of the first part of \eqref{OPunif}.

By \eqref{OP1}, Minkowski's inequality and Lemma \ref{COR:OP} (i) and (ii), we obtain the following estimate for $\eps >0$:

\noi
\begin{align}
\begin{split}
\| \I_\eps(F) \|_{L^{\infty}_T H^s_x} & \les \int_{0}^T \| \eps^{-2} \D_\eps(t) F(t) \|_{H^s_x} dt \\
& \les \int_{0}^T t^{- \frac12} dt \| F \| _{L^{\infty}_T H^{s-1}_x} \les T^{\frac12} \| F \| _{L^{\infty}_T H^{s-1}_x}.
\end{split}
 \label{pf5}
\end{align}

\noi
The same inequality holds for $\eps=0$. The continuity in time of $\I_\eps (F)$ ($\eps \in [0,1]$) follows from a similar computation. This finishes the proof of \eqref{OPunif}.

The estimate on $P_\eps -P_0$ in \eqref{detconv} follows from Lemma \ref{COR:OP} (v), Lemma \ref{LEM:heat} and similar arguments along with the bound $\| \F^{-1}( \ft{f}  \ind_{\jb n > \eps^{-1 + \ta}} ) \|_{H^s_x} \les \eps^{\frac \ta 2} \|f\|_{H^{s+\ta}_x}$ for any $\ta > 0$ small enough. Similarly, the estimate on $\I_\eps - \I_0$ in \eqref{detconv} follows from \eqref{mul3} in Lemma \ref{LEM:F2}.
\end{proof}

Lastly, we prove Corollary \ref{COR:det}.

\noi
\begin{proof}[Proof of Corollary \ref{COR:det}] The continuity of the map $\eps \in (0,1] \mapsto P_\eps $ is deduced from the smoothness of the map $(\eps,t) \mapsto \ft{\D_\eps}(n,t)$ for each $n \in \Z^2$ (by Lemma \ref{LEM:F3}) and the dominated convergence theorem. The continuity at $\eps =0$ of $\eps \mapsto P_\eps $ follows from \eqref{detconv}. Hence, \eqref{C11} follows from the above and Proposition \ref{PROP:det}. The bound \eqref{C12} follows from similar considerations.
\end{proof}

\section{Polynomial models}\label{SEC2}

\subsection{On the stochastic convolution}

In this subsection, we construct the Wick powers \eqref{Herm1}.

\noi
\begin{proposition}\label{PROP:sto} 
Let $\l \in \N$. Fix any finite $p,q \ge 1$, $T> 0$ and $\s > 0$. Then, the following holds:

\smallskip
\noi
\textup{(i)} Let $\eps \in [0,1]$. The sequence $\{ : \!\Psi^\l_{\eps,N} \! : \}_{N \in \N}$ defined in \eqref{Herm1} is a Cauchy sequence in $L^p \big(  \O; L^q ([0,T]; W^{- \s, \infty} (\T^2) ) \big)$ and thus converges, as $N \to \infty$, to a limiting stochastic process in $L^p \big(  \O; L^q ([0,T]; W^{- \s, \infty} (\T^2) ) \big)$, denoted by $:\!\Psi^\l_{\eps} \!:$.

\smallskip
\noi
\textup{(ii)} The sequence $\{( \eps, t) \mapsto  : \!\Psi^\l_{\eps,N} \! : \}_{N \in \N}$ also converges to the process $(\eps, t) \mapsto : \!\Psi^\l_{\eps} \! :$ in $L^p \big(  \O; L^q ( [0,1] \times [0,T]; W^{- \s, \infty} (\T^2) ) \big)$ and almost surely in $C \big( [0,1] \times [0,T]; W^{- \s, \infty} (\T^2) \big)$ as $N \to \infty$.
\end{proposition}

In \cite{GKO1} and \cite{DPD}, the processes $\{ : \!\Psi^\l_{\eps,N} \! : \}_{N \in \N}$ were constructed for $\eps = 1$ and $\eps = 0$, respectively. The main novelty in Proposition \ref{PROP:sto} lies in (ii), where the stochastic process $(\eps, t) \mapsto : \!\Psi^\l_{\eps} \! :$ is constructed as a continuous function of both $\eps$ and $t$ by using the bi-parameter Kolmogorov continuity criterion (Lemma \ref{LEM:kol}). This implies in particular the convergence of $:  \! \Psi_{\eps}^{\ell} \!  :$ to $:  \! \Psi_{0}^{\ell} \!  :$ along the continuous parameter $\eps \to 0$.

\noi
\begin{proof}
Fix $\l \in \N$, and $\s >0$. Our first goal is to bound the variance:

\noi
\begin{align}
\E \big[  \big( \jb{\nb}^{- \s} : \! \Psi_{\eps,N}^{\ell} (t,\cdot) \! : (x) \big)^2  \big] \les 1,
\label{S1} 
\end{align}

\noi
uniformly in $N \in \N$, $t \ge 0$, $\eps \in [0,1]$, and $x \in \T^2$. 

Fix $N \in \N$, $t\ge 0$, $(x,y) \in (\T^2) ^2$ and $\eps \in (0,1]$ (the case $\eps =0$ in \eqref{S1} follows from similar arguments). By \eqref{convo10}, \eqref{Herm1} and Lemma \ref{LEM:Wick}, we have,

\noi
\begin{align*}
\frac{1}{\ell!} \E \big[  : \! \Psi_{\eps,N}^{\ell} (t,x) \! : \ : \! \Psi_{\eps,N}^{\ell} (t,y) \! :     \big] = \E \big[  \Psi_{\eps,N}(t,x) \Psi_{\eps,N}(t,y)  \big]^\l
\end{align*}

\noi
Applying the Bessel potentials $\jb{\nb_x}^{-\s}$ and $\jb{\nb_y}^{-\s}$ and then setting $x=y$, we see from the previous computation that in order to bound the left-hand-side of \eqref{S1}, we need to bound terms of the form

\noi
\begin{align}
\sum_{ \substack{ n_1, \cdots, n_{\ell} \in \Z^2 \\  \jb {n_j} \les N}} \jb{n_1 + \cdots + n_{\ell}}^{-2 \s} F_1(n_1,t) \cdots F_{\ell}(n_\ell,t),
\label{S2}
\end{align}

\noi
where we write for $1 \le j \le \l$ and $n \in \Z^2$,

\noi
\begin{align*}
& F_j(n,t)  = \E \big[  ( \F ( \Psi_{\eps,N})(n,t)  )^2 \big].
\end{align*}

\noi
Hence, by \eqref{convo10} and \eqref{mul1}, we get

\noi
\begin{align}
\begin{split}
F_j(n,t) & \les \| \ind_{[0,t]}(t')  \eps^{- 2}  \ft{\D_\eps} (n,t)  \|_{L^2_{t'}}^2 \\
& \les \jb{n}^{-2}
\end{split}
\label{S3}
\end{align}

\noi
for $n \in \Z^2$. This gives

\noi
\begin{align*}
\eqref{S2} \les \sum_{ \substack{ n_1, \cdots, n_{\ell} \in \Z^2 \\ \jb{n_j} \les N}} \jb{n_1 + \cdots + n_{\ell}}^{-2 \s}  \prod_{j = 1}^{\l}  \jb{n_j}^{-2} \les 1,
\end{align*}

\noi
and shows \eqref{S1}. 

Let $r > \frac{4}{\s}$ and finite $p, q \ge 1$ with $p \ge q,r$. By Sobolev's and Minkowski's inequalities and Lemma \ref{LEM:hyp} along with \eqref{S1}, we have

\noi
\begin{align}
\begin{split}
\| : \! \Psi_{\eps,N}^{\ell} \! :  \|_{L^p(\O) L^q_T W^{-\s, \infty}_x } & \les \| : \! \Psi_{\eps,N}^{\ell}  \! :  \|_{L^p(\O) L^q_T W^{-\frac{\s}{2}, r}_x } \\
& \leq \big\| \| \jb{\nb}^{- \frac{\s}{2}} : \! \Psi_{\eps,N}^{\ell}  \! :  \|_{L^p(\O)} \big\|_{ L^q_T L^r_x } \\
& \les p^{\frac{\l}{2}} \big\| \| \jb{\nb}^{- \frac{\s}{2}} : \! \Psi_{\eps,N}^{\ell}  \! :  \|_{L^2(\O)} \big\|_{ L^q_T L^r_x } \\
& \les T^{\frac{1}{q}} p^{\frac{\l}{2}} \les_{T,p,\l} 1.
\end{split}
\label{S4}
\end{align}

\noi
Using the inclusion $L^{p_2}(\O) \subset L^{p_1}(\O)$ for $p_1 \le p_2$, we obtain a similar bound for any finite $p \ge 1$. Let $p \ge 1$ be finite and $M \ge N$. By similar arguments, we also get

\noi
\begin{align}
\big\| : \! \Psi_{\eps,N}^{\ell} \! : - : \! \Psi_{\eps,M}^{\ell} \! :  \big\|_{L^p(\O) L^q_T W^{-\s, \infty}_x } \les N^{- \g},
\label{S5}
\end{align}

\noi
for some small $\g >0$. The bound \eqref{S5} shows that $\{ : \! \Psi_{\eps,N}^{\ell} \! : \}_{N \ge 1}$ is a Cauchy sequence in $L^p \big( \O; L^q \big([0,T];  W^{-\s, \infty}(\T^2) \big) \big) $, for any finite $p,q \ge 1$ and $\s >0$. Thus, it converges to some limit denoted by $: \! \Psi_{\eps}^{\ell} \! :$. This shows the first part of the statement, i.e. item (i).

Before proceeding with the proof of (ii), we note that by arguing as in \eqref{S4} and \eqref{S5}, we can construct a process $(\eps, t) \mapsto : \! \Psi ^\l \! : (\eps,t) $ as the limit of the sequence of stochastic objects $\{ (\eps,t) \mapsto  : \! \Psi_{\eps,N}^{\ell} \!:\}_{N \ge 1}$ in ~$L^p \big( \O; L^2 \big( [0,1] \times [0,T];  W^{-\s, \infty}(\T^2) \big) \big) $. Furthermore, this construction is coherent with that of $\{: \! \Psi_{\eps}^{\ell} \! :\}_{\eps \in [0,1]}$ in the sense that $:\! \Psi ^\l\!: (\eps,t)  = \, \, \, :\! \Psi_{\eps}^{\ell} \!:(t)$ in $L^p \big( \O; L^2 \big( [0,1] \times [0,T];  W^{-\s, \infty}(\T^2) \big) \big) $. This can indeed be observed by using the dominated convergence theorem and the uniformity of the bound \eqref{S5} in $\eps \in [0,1]$. This ensures that the process that we are going to construct below indeed corresponds to $\{: \! \Psi_{\eps}^{\ell} \! :\}_{\eps \in [0,1]}$. Furthermore, by arguing as in (the proof of) \cite[Proposition 3.2]{OPTz}, one can prove by using the Borel-Cantelli lemma that the convergence of $\{ (\eps,t) \mapsto  : \! \Psi_{\eps,N}^{\ell} \!:\}_{N \ge 1}$ to $(\eps, t) \mapsto : \! \Psi ^\l \! : (\eps,t) $ holds in $L^2 \big( [0,1] \times [0,T];  W^{-\s, \infty}(\T^2) \big) \big) $, almost surely.

We now prove (ii) and investigate the continuity in $(\eps,t)$ of our stochastic objects. Let $h_1, h_2 \in \R$. We define the operators $\dl_{h_1, h_2} $, $\dl_{h_1}^1$ and $\dl_{h_2}^2$ by 

\noi
\begin{align}
\begin{split}
\dl_{h_1,h_2}X (\eps, t) & = X(\eps + h_1, t+h_2) - X(\eps,t) \\
\dl_{h_1}^1 X(\eps,t) & = X(\eps+h_1,t) - X(\eps,t) \\ \dl_{h_2}^2 X(\eps,t) & = X(\eps,t+h_2) - X(\eps,t),
\end{split}
\label{diff}
\end{align}

Fix $\eps \in [0,1]$ and $t \in [0,T]$. Let $h_1, h_2 \in \R$ such that $\eps + h_1 \in [0,1]$ and $t + h_2 \ge 0$. Let $\l \in \N$, $\s >0$. We aim to show the bound

\noi
\begin{align}
\E \big[  \big( \jb{\nb}^{- \s} \dl_{h_1,h_2} : \! \Psi_{\eps,N}^{\ell} (t,) \! : (x) \big)^2  \big] \les \|(h_1,h_2)\|_2^{\g},
\label{S6}
\end{align}

\noi
for some $\g > 0$ and uniformly in all parameters. In \eqref{S6}, $\| \cdot \|_2$ denotes the Euclidean norm on $\R^2$. 

Let $(x,y) \in (\T^2)^2$. We only treat the case $(\eps,\eps + h_1) \in (0,1]^2$ as the case $\eps = 0$ or $\eps + h_1=0$ follows from similar considerations. Expanding the expression 

\noi
\begin{align}
\frac{1}{\ell!} \E \big[  \dl_{h_1,h_2} : \! \Psi_{\eps,N}^{\ell} (t,x) \! : \dl_{h_1,h_2} : \! \Psi_{\eps,N}^{\ell} (t,y) \! :     \big]
\label{S7}
\end{align}

\noi
yields 

\noi
\begin{align}
\begin{split}
\eqref{S7} & = \Big( \E \big[ \Psi_{\eps+h_1,N} (t+h_2,x) \Psi_{\eps+h_1,N} (t+h_2,y)  \big]^\l \\
& \qquad -  \E \big[ \Psi_{\eps,N} (t,x) \Psi_{\eps+h_1,N} (t+h_2,y)  \big]^\l \Big)  \\
& \quad + \Big( \E \big[ \Psi_{\eps,N} (t,x) \Psi_{\eps,N} (t,y)  \big]^\l \\
& \qquad - \E \big[ \Psi_{\eps+h_1,N} (t+h_2,x) \Psi_{\eps,N} (t,y)  \big]^\l \Big) \\
& =: \1 + \II.
\end{split}
\label{S8}
\end{align}

\noi
We have 

\noi
\begin{align}
\begin{split}
\1 & = \E\big[ \dl_{h_1,h_2} \Psi_{\eps,N} (t,x) \Psi_{\eps+h_1,N} (t+h_2,y)  \big] \\
& \qquad \qquad \times \sum_{j = 0}^{\l - 1} \Big( \E \big[ \Psi_{\eps+h_1,N} (t+h_2,x) \Psi_{\eps+h_1,N} (t+h_2,y)  \big]^j \\
& \qquad \qquad \qquad \E \big[ \Psi_{\eps,N} (t,x) \Psi_{\eps+h_1,N} (t+h_2,y)  \big]^{\l-1 -j} \Big).
\end{split}
\label{S9}
\end{align}

\noi
A similar expression holds for $\II$. Thus, by reasoning as before, in order to estimate $\E \big[  \big( \jb{\nb}^{- \s} \dl_{h_1,h_2} : \! \Psi_{\eps,N}^{\ell} (t,) \! : (x) \big)^2  \big]$ we are led to bound sums of the form 

\noi
\begin{align}
\sum_{ \substack{ n_1, \cdots, n_{\ell} \in \Z^2 \\  \jb {n_j} \les N}} \jb{n_1 + \cdots + n_{\ell}}^{-2 \s} G_1(n_1) \cdots G_{\ell}(n_\ell),
\label{S10}
\end{align}

\noi
with $G_j = G_j(n_j,t,\eps, h_1,h_2)$ ($1 \le j \le \l$). We have 

\noi
\begin{align}
\begin{split}
& G_1(n,t,\eps, h_1,h_2) = \E \Big[ \F( \dl_{h_1,h_2} \Psi_{\eps,N}(n,t) ) \F( \Psi_{\eps^{(1)},N}(n,t^{(1)}) )  \Big] \\
& G_j(n,t,\eps, h_1,h_2) = \E \Big[ \F( \Psi_{\eps_1^{(j)},N}(n,t_1^{(j)}) ) \F( \Psi_{\eps_2^{(j)},N}(n,t_2^{(j)}) )  \Big], \quad 2 \le j \le \l,
\end{split}
\label{S11}
\end{align}

\noi
where $(\eps^{(1)}, \eps_1^{(j)}, \eps_2^{(j)} ) \in \{ \eps, \eps + h_1 \}^3$ and $( t^{(1)}, t_1^{(j)}, t_2^{(j)} ) \in \{t, t+ h_2 \}^3$ for $2 \le j \le \l$. As before, we have

\noi
\begin{align}
G_j(n,t,\eps, h_1,h_2) \les \jb n ^{-2},
\label{S12}
\end{align}

\noi
for $2 \le j \le \l$, uniformly in all parameters. Denoting by $\langle \cdot, \cdot \rangle_{L^2} $ the canonical inner product on $L^2(\R)$, we have

\noi
\begin{align}
\begin{split}
& G_1(n,t,\eps, h_1,h_2) \\
& \qquad = \big\langle  \dl_{h_1,h_2} (\eps^{-2} \ft{\D_{\eps}}(n,t-t')), \ind_{[0,\min (t,  t^{(1)})]}(t') (\eps^{(1)})^{-2} \ft{\D_{\eps^{(1)}}}(n,t^{(1)}-t') \big\rangle_{L^2_{t'}}   \\
& \qquad = \big\langle  \dl^1_{h_1} (\eps^{-2} \ft{\D_{\eps}}(n,t+h_2-t')), \ind_{[0,\min (t,  t^{(1)})]}(t') (\eps^{(1)})^{-2} \ft{\D_{\eps^{(1)}}}(n,t^{(1)}-t') \big\rangle_{L^2_{t'}}   \\
& \qquad \quad + \big\langle \eps^{-2} \dl^2_{h_2}  \ft{\D_{\eps}}(n,t-t'), \ind_{[0,\min (t,  t^{(1)})]}(t')(\eps^{(1)})^{-2} \ft{\D_{\eps^{(1)}}}(n,t^{(1)}-t') \big\rangle_{L^2_{t'}}  \\
& \qquad = \III + \IV.
\end{split}
\label{S13}
\end{align}

We now estimate the terms $\III$ and $\IV$. By \eqref{mul1} we have

\noi
\begin{align}
|\III|, |\IV| \les \jb n ^{-2}.
\label{S14}
\end{align}

\noi
Let us assume that $h_1 \ge 0$ for convenience. By \eqref{mul1}, \eqref{B1} and the mean value theorem, we also have the following crude bound:

\noi
\begin{align}
|\III| & \les \big\langle  \dl^1_{h_1} (\eps^{-2}) \, \ft{\D_{\eps}}(n,t+h_2-t'), \ind_{[0,\min (t,  t^{(1)})]}(t') (\eps^{(1)})^{-2} \ft{\D_{\eps^{(1)}}}(n,t^{(1)}-t') \big\rangle_{L^2_{t'}} \notag \\
& \qquad + \big\langle   \eps^{-2} \, \dl^1_{h_1} \big( \ft{\D_{\eps}}(n,t+h_2-t') \big), \ind_{[0,\min (t,  t^{(1)})]}(t') (\eps^{(1)})^{-2} \ft{\D_{\eps^{(1)}}}(n,t^{(1)}-t') \big\rangle_{L^2_{t'}} \notag \\
&  \les h_1 \eps^{-3} \, \|  \ft{\D_{\eps}}(n,t')  \|_{L^{2} _{t'}} \cdot \big\| (\eps^{(1)})^{-2}  \ft{\D_{\eps ^{(1)}}}(n,t')  \big\|_{L^{2} _{t'}} \notag \\
& \qquad \quad + \eps^{-2} h_1 \|  \sup_{\eps_0 \in (\eps, \eps + h_1)} \partial_\eps \ft{\D_{\eps_0}}(n,t')  \|_{L^{\infty} _{t'}} \cdot \big\| (\eps^{(1)})^{-2}  \ft{\D_{\eps ^{(1)}}}(n,t')  \big\|_{L^{1} _{t'}} \notag \\
& \les h_1 \eps^{-7} \jb n ^2. \label{S15}
\end{align}

\noi
The bound \eqref{S15} blows up when $\eps >0$ is much smaller than other parameters and we now obtain a bound which is acceptable for small values of $\eps$. We note that 

\noi
\begin{align*}
\dl^1_{h_1} (\eps^{-2} \ft{\D_{\eps}}(n,t')) = \big( (\eps+h_1)^{-2} \ft{\D_{\eps+ h_1}}(n,t') - e^{-t' \jb n ^2} \big) + \big( e^{- t' \jb n ^2} - \eps^{-2} \ft{\D_{\eps}}(n,t') \big).
\end{align*}

\noi
Hence, by using \eqref{mul3}, we have

\noi
\begin{align}
|\III| \les_T \eps^{2 \ta} + (\eps + h_1)^{2 \ta},
\label{S16}
\end{align}

\noi
for $\jb n \le \min\big( \eps^{-1 + \ta}, (\eps+h_1)^{- 1 + \ta} \big) = (\eps+h_1)^{- 1 + \ta}$. We claim that we have

\noi
\begin{align}
|\III| \les |h_1|^{\g_1} \jb{n}^{-2 + \g_2},
\label{S17}
\end{align}

\noi
for some $\g_1 > 0$ and some small $0 < \g_2 \ll \s$. We may assume $h_1 \eps^{-7}  > h_1^{\g}$ for $0 < \g \ll 1$, for otherwise, interpolating \eqref{S14} and \eqref{S15} gives \eqref{S17}. We then have $\eps \le h_1 ^{\frac{1-\g}{7}}$. If $\jb n \le (\eps+h_1)^{- 1 + \ta}$, we have 

\noi
\begin{align}
\eqref{S16} \les h_1 ^{\frac{2 \ta (1-\g)}{7}}.
\label{S18}
\end{align}

\noi
Interpolating \eqref{S18} with \eqref{S14} then yields \eqref{S17}. Otherwise $\jb n > (\eps+h_1)^{-1 + \ta}$ and hence 

\noi
\begin{align*}
\eqref{S14} \les \jb{n}^{-2 + \g} (\eps + h_1)^{ \g(1 - \ta) } \les \jb{n}^{-2 + \g} h_1^{ \frac{\g(1-\g)}{7} (1 - \ta) },
\end{align*}

\noi
for $0 < \g \ll \s$. This concludes the proof of \eqref{S17}.

We now estimate $\IV$. By using \eqref{mul2}, we have

\noi
\begin{align}
|\IV| \les |h_2| \eps^{-2}.
\label{S19}
\end{align}

\noi
Since we can write 

\noi
\begin{align}
\begin{split}
\eps^{-2} \dl^2_{h_2}  \ft{\D_{\eps}}(n,t') & = \big( \eps^{-2} \ft{\D_\eps}(n,t'+h_2) - e^{- (t'+h_2) \jb n ^2} \big) \\
& \quad + \big( e^{- t' \jb n ^2} - \eps^{-2} \ft{\D_\eps}(n,t') \big) \\
& \quad + \big(   e^{- (t'+h_2) \jb n ^2} -  e^{- t' \jb n ^2} \big),
\end{split}
\label{S20}
\end{align}

\noi
we have, from \eqref{mul3}, \eqref{S20} and the mean value theorem, the bound

\noi
\begin{align}
|\IV| \les \eps^{2 \ta } + |h_2|,
\label{S21}
\end{align}

\noi
for $\jb n \le \eps^{-1 + \ta}$. Combining \eqref{S14}, \eqref{S19} and \eqref{S21} and arguing as in the estimate of the term $\III$, we deduce

\noi
\begin{align}
| \IV | \les |h_2|^{\g_1} \jb{n}^{-2 + \g_2}
\label{S22}
\end{align}

\noi
for some $\g_1 > 0$ and some small $0 < \g_2 \ll \s$.

Thus, we deduce from \eqref{S13}, \eqref{S17} and \eqref{S22}, the estimate

\noi
\begin{align}
| G_1(n,t,\eps, h_1,h_2) | \les \| (h_1, h_2) \|_2^{\g} \, \jb{n}^{-2 + \g},
\label{S23}
\end{align}

\noi
for $0 < \g \ll \s$. Hence, \eqref{S6} follows from \eqref{S10} with \eqref{S11}, \eqref{S12} and \eqref{S23}.

Arguing as in the computations leading to \eqref{S4} and \eqref{S5}, we deduce that for $\l \in \N$, $M \ge N$, finite $p,q \ge 1$, $\eps \in [0,1]$, $t \in [0,T]$ and $h_1, h_2 \in \R$ such that $\eps + h_1 \in [0,1]$ and $t + h_2 \in [0,T]$, the following bounds hold:

\noi
\begin{align}
\begin{split}
 \| \dl_{h_1,h_2} : \! \Psi_{\eps,N}^{\ell} \! : (t)  \|_{L^p(\O) W^{-\s, \infty}_x } & \les_{p,\l} \| (h_1, h_2) \|_2^{\g}, \\
\| \dl_{h_1,h_2} ( : \! \Psi_{\eps,N}^{\ell} \! :  - : \! \Psi_{\eps,M}^{\ell} \! :  ) (t)  \|_{L^p(\O) W^{-\s, \infty}_x } & \les_{p,\l} N^{- \g} \| (h_1, h_2) \|_2^{\g},
\end{split}
\label{S24}
\end{align}

\noi
for $\g >0 $ small enough. 

Fix $n \in \Z^2$. Given the smoothness of $(\eps,t) \mapsto \ft{\D_\eps}(n,t')$, the following integration by parts formula holds almost-surely:

\noi
\begin{align}
\int_0^t \eps^{-2} \ft{\D_\eps}(n,t') dB_n(t') = \eps^{-2} \ft{\D_\eps}(n,t) B_n(t) -  \int_0^t  \eps^{-2} \dt \ft{\D_\eps}(n,t') B_n(t') dt'.
\label{IPP}
\end{align}

\noi
for any $t \ge 0$ and where the Brownian motion $B_n$ is as in \eqref{W1}. Hence, we infer from \eqref{IPP}, \eqref{convo10} and Lemma \ref{LEM:F3} that for each $N \in \N$, the map $(\eps,t)  \mapsto  \Psi_{\eps,N} $ belongs to $C \big( (0,1] \times  [0,T];  H_x^{\infty} \big)$; whence so does $(\eps, t) \mapsto  :\! \Psi_{\eps,N}^{\ell} \!:$ by \eqref{Herm1}. Thus, by applying Lemma \ref{LEM:kol} on $(0,1] \times [0,T]$, we have the following bounds:

\noi
\begin{align}
\begin{split}
 \sup_{N \in \N} \|:\! \Psi_{\eps,N}^{\ell} \! :\|_{L^p(\O) C ( (0,1] \times [0,T]; W^{-\s, \infty}_x  )} & \les 1, \\
\sup_{N \in \N} \sup_{M \ge N} N^{\g} \|: \! \Psi_{\eps,N}^{\ell} \! : - : \! \Psi_{\eps,M}^{\ell} \! :   \|_{L^p(\O) C ( (0,1] \times [0,T]; W^{-\s, \infty}_x  )} & \les 1.
\end{split}
\label{S71a}
\end{align}

\noi
The bounds \eqref{S71a} show that the sequence $\{: \! \Psi_{\eps,N}^{\ell} \! :\}_{N \in \N}$ is Cauchy and hence converges to a limit in $L^p\big(\O; C ( (0,1] \times [0,T]; W^{-\s, \infty}_x  )\big)$. By uniqueness of the almost sure limit in $L^p \big( \O; L^2 \big( (0,1] \times [0,T];  W^{-\s, \infty}(\T^2) \big) \big) $, this limit is given by $(\eps, t) \mapsto : \! \Psi_{\eps}^{\ell} \! :(t)$. Now, \eqref{S71a} and another application of the Borel-Cantelli lemma produces a set $\O_0$ of full $\PP$-probability such that

\noi
\begin{align}
: \! \Psi_{\eps,N}^{\ell} \! : \too : \! \Psi_{\eps}^{\ell} \! :, \quad \text{in  } C ( (0,1] \times [0,T]; W^{-\s, \infty}_x  ),
\label{S71}
\end{align}

\noi
on $\O_0$.

Let us verify the continuity at $\eps = 0$ of our stochastic objects. Fix $N \in \N$. By following the proof of \eqref{S6}, we can show

\noi
\begin{align*}
\E \big[  \big( \jb{\nb}^{- \s} \dl^2_{h_2} ( : \! \Psi_{\eps,N}^{\ell} (t) \! : - \Psi_{0,N}^{\ell} (t) )(x) \big)^2  \big] \les \eps^\g |h_2|^{\g},
\end{align*}

\noi
for each fixed $\eps \in (0,1]$, $t \in [0,T]$ and $h_2 \in \R$ such that $t + h_2 \in [0,T]$ and some $0 < \g, \sigma \ll 1$. From the usual Kolmogorov continuity criterion, we then obtain the bound

\noi
\begin{align}
\big\| \sup_{\eps \in(0,1]} \eps^{-\g}\|  : \! \Psi_{\eps,N}^{\ell} \!  : - : \! \Psi_{0,N}^{\ell} \!  :  \|_{ C_T W^{-\s, \infty}_x }\big\|_{L^p(\O)} \les_p 1,
\label{S70}
\end{align}

\noi
for any $p \ge 1$. Hence, by \eqref{S70} and Chebyshev's inequality, we get a set of full measure $\O_N$ such that 

\noi
\begin{align}
\|  : \! \Psi_{\eps,N}^{\ell} \!  : - : \! \Psi_{0,N}^{\ell} \!  :  \|_{ C_T W^{-\s, \infty}_x }  \les \eps^\g,
\label{S70b}
\end{align}

\noi
for any $\eps \in (0,1]$, on $\O_N$. We define $\O_{\text{final}}$ to be the full probability set

\noi
\begin{align*}
\O_{\text{final}} := \O_0 \cap \bigcap_{N \in \N} \O_N. 
\end{align*}

\noi
Hence, by \eqref{S70b}, we get that for each $N \in \N$, $: \! \Psi_{\eps,N}^{\ell} \!  :$ is continuous at $\eps = 0$ on $\O_{\text{final}}$. We thus deduce from \eqref{S71} that

\noi
\begin{align}
: \! \Psi_{\eps,N}^{\ell} \! : \too : \! \Psi_{\eps}^{\ell} \! :, \quad \text{in  } C ( [0,1] \times [0,T]; W^{-\s, \infty}_x  ),
\label{S72}
\end{align}

\noi
on $\O_{\text{final}}$. This finishes the proof.
\end{proof}

\noi
\subsection{Local theory}\label{SUBSEC:LWP} In this subection, we prove Theorem \ref{THM1}. Fix an integer $k \ge 2$. First, we consider the following {\it enhanced equation} on functions of the variables $(x, \eps, t)$:

\noi
\begin{align}
\begin{cases}
\eps^2 \dt^2 v + \dt v +(1-\Dl)v  +
 \sum_{\ell=0}^k {k\choose \ell} \, \Xi_\l   \, v^{k-\ell}
=0\\
(v,  \ind_{\eps > 0} \dt v) ) |_{t=0} =(\phi_0, \ind_{\eps > 0} \phi_1),
\end{cases}
\quad (x,\eps,t) \in \T^2 \times [0,1] \times \R_+,
\label{SNLW5}
\end{align}

\noi
for given initial data $(\phi _0, \phi _1)$ and a source $(\Xi_1 , \dots, \Xi_k)$
with the understanding that $\Xi_0 \equiv 1$. 

Here, we present a local well-posedness argument for \eqref{SNLW5} based on Sobolev's inequality as in \cite{GKOT}. More precisely, we prove in Proposition \ref{PROP:LWP} below, the existence of a solution $v = v (\eps,t) $ to \eqref{SNLW5} which belongs to $C\big( [0,1] \times [0,T]; H^{\kk}(\T^2) \big)$, for some $\kk >0$. Note that, for each fixed $\eps \in [0,1]$, $v_\eps = v(\eps, \cdot)$ solves \eqref{SNLW4}. By using the convergence of the stochastic objects in Proposition \ref{PROP:sto}, the proof of Theorem \ref{THM1} essentially reduces to proving that $v_\eps$ converges to $v_{\eps = 0}$ as $\eps \to 0$. However, by construction, this immediately follows from the continuity of $\eps \mapsto v (\eps, \cdot)$ at $\eps =0$. Hence, the convergence in Theorem \ref{THM1} is a consequence of the existence of the solution $v$ to \eqref{SNLW5}. We postpone the proof of Theorem \ref{THM1} to the end of this section.

Given $\ta >0$, $T >0$, and $k \ge 2$ define $\mathcal{X}^{k,\ta}_T(\T^2) = \mathcal{X}^{\ta}_T(\T^2)$  by 

\noi
\begin{align*}
\mathcal{X}^\ta_T(\T^2) := \big(C \big( [0,1] \times [0, T]; W^{-\ta, \infty}(\T^2) \big) \big)^{\otimes k},
\end{align*}

\noi
and set
\[ \|\pmb{\Xi}\|_{\mathcal{X}^\ta_T}
=  \sum_{j = 1}^k \| \Xi_j \|_{C([0,1] \times [0, T]; W_x^{-\ta, \infty})},\]

\noi
for $\pmb{\Xi} = (\Xi_1, \Xi_2, \dots, \Xi_k ) \in \mathcal{X}^\s_T(\T^2)$. In what follows, we use the shorthand notation $\mathcal{X}^\ta(\T^2)$ for $\mathcal{X}^\ta_1(\T^2)$. 

We have the following local well-posedness result for \eqref{SNLW5}.

\begin{proposition}\label{PROP:LWP}
Fix an integer $k \geq 2$ and $\dl \le \frac{1}{2(k-1)} $. Let $0 < \ta \ll \dl$. Then, the equation
\eqref{SNLW5} is locally well-posed
in $ \H^{1-\dl} \times \mathcal{X}^{\ta}(\T^2)$.
More precisely, 
given an enhanced data set: 
\begin{align}
 (\phi _0, \phi _1,  \pmb{\Xi}) \in \H^{1-\dl}(\T^2) \times \mathcal{X}^{\ta}(\T^2), 
\label{data1}
\end{align}

\noi
with $\pmb{\Xi} = (\Xi_1, \Xi_2, \dots, \Xi_k )$, there exist $T = T(\|  (\phi_0, \phi_1) \|_{\H^{ 1- \dl}_x} , \|\pmb{\Xi}\|_{\mathcal{X}^{\s}}) \in (0, 1]$
and a unique solution $v = v(\eps,t)$ to \eqref{SNLW5} in the class
\begin{align}
C \big( [0,1] \times [0, T]; H^{1-\dl}(\T^2) \big).
\label{Z1}
\end{align}
In particular, the uniqueness of $v$ 
holds in the entire class \eqref{Z1}.
Furthermore, the solution map 

\noi
\begin{align*}
(\phi_0, \phi_1,
\pmb{\Xi}) \in \H^{1-\dl}(\T^2) \times \mathcal{X}^{\ta}(\T^2)
\mapsto 
v \in C( [0,1] \times [0, T]; H^{1-\dl}(\T^2))
\end{align*}

\noi
is locally Lipschitz continuous.

\end{proposition}


\noi
\begin{proof}
By writing \eqref{SNLW5} in the Duhamel formulation, we
have 

\noi
 \begin{align}
 \begin{split}
v(\eps,t) = \G (v)(\eps,t) \deff
\ &  P_\eps(t)(\phi _0,\phi _1)\\
  & -
\,  \sum_{\ell=0}^k  {k\choose \ell} \I_\eps(t) 
\big( \Xi_\l(\eps,t) \,  v(\eps,t)^{k-\ell}\big) 
\end{split}
\label{SNLW15}
\end{align}

\noi
where the map $\G = \G_{\pmb{\Xi}}$ depends on the enhanced data set
$\pmb{\Xi}$ in \eqref{data1} and $P_\eps$ and $\I_\eps$ are as in \eqref{O1a}, \eqref{OP1}.
Fix  $0 < T \le 1$. We have from Corollary \ref{COR:det},

\noi
\begin{align}
\| P_\eps(t) (\phi _0, \phi _1) \|_{C_{\eps,T} H^{1-\dl}_x } \les  \|  (\phi _0, \phi _1)  \|_{\H^{1-\dl}_x}
\label{Z2}
\end{align}

We first treat the case $\l =0$.
From  Corollary \ref{COR:det} and Sobolev's inequality (twice), we obtain

\noi
\begin{align}
\begin{split}
\| \I_\eps(t)( v^k) \|_{ C_{\eps,T} H^{1-\dl}_x}
&   \les T^{\frac12} \| v^k\|_{C_{\eps,T} H^{-\dl}_x}
\les T^{\frac12} \| v^k\|_{C_{\eps,T} L^{\frac{2}{1+\dl}}_x}
 \les T^{\frac12} \| v\|_{C_{\eps,T} L^{\frac{2k}{1+\dl}}_x}^k\\
& \les T^{\frac12} \| v\|_{C_{\eps,T} H^{1-\dl}_x}^k, 
\end{split}
\label{Z3}
\end{align}

\noi
provided that 
\begin{align*}
0 \leq \dl \leq \frac{1}{k-1}.
\end{align*}

\noi
For $1 \leq \l \leq k-1$, 
it follows from Corollary \ref{COR:det}, Lemma \ref{LEM:bilin} (i) and (ii), and Sobolev's inequality that 

\noi
\begin{align}
\begin{split}
\| \I_\eps(t) (\Xi_\l \,  v^{k-\ell} ) \|_{C_{\eps,T} H^{1-\dl}_x}
& \les T^\frac{1}{2}\| \Xi_\l\,  v^{k-\ell}\|_{C_{\eps,T} H^{-\dl}_x}\\
& \les T^\frac{1}{2}\|\jb{\nb}^{-\dl} \Xi_\l
\|_{C_{\eps,T} L^\frac{2}{\dl}_x}
\| \jb{\nb}^{\dl} v^{k-\ell}\|_{C_{\eps,T} L^{2}_x}\\
& \les T^\frac{1}{2}\|\pmb{\Xi}\|_{\mathcal{X}^{\ta}}
\| \jb{\nb}^{\dl} v\|_{C_{\eps,T} L^{2(k-\l)}_x}^{k-\l}\\
& \les T^\frac{1}{2}\|\pmb{\Xi}\|_{\mathcal{X}^{\ta}}
\|  v\|_{C_{\eps,T} H^{1-\dl}_x}^{k-\l}, 
\end{split}
\label{Z4}
\end{align}

\noi
provided that 
\begin{align}
0 \leq \dl \leq \frac 1{2(k-1)}.
\label{Z5}
\end{align}

\noi
Lastly, again from Corollary \ref{COR:det}, we have

\noi
\begin{align}
\begin{split}
\| \I_\eps(t)(
  \Xi_k)  \|_{C_{\eps,T} H^{1-\dl}_x}
&   \les T^\frac{1}{2} \| \Xi_k\|_{C_{\eps,T} H^{-\dl}_x}
\leq
 T^\frac{1}{2}\|\pmb{\Xi}\|_{\mathcal{X}^{\ta}}.
\end{split}
\label{Z6}
\end{align}

Putting
 \eqref{SNLW15}, \eqref{Z2}, \eqref{Z3}, \eqref{Z4} and \eqref{Z6}
 together, we have
 
\noi
\begin{align*}
\|\G(v)\|_{C_{\eps,T} H^{1-\dl}_x}
&   \le
C_1 \| (\phi_0, \phi_1)\|_{\H_x^{1-\dl}}
+ 
C_2  T^\frac{1}{2}
 \big(1 + \|\pmb{\Xi}\|_{\mathcal{X}^{\ta}}\big)
\big( 1 +  \| v\|_{C_{\eps,T} H^{1-\dl}_x}\big)^k, 
\end{align*}

\noi
as long as \eqref{Z5} is satisfied. 

By similar arguments, the following difference estimate holds:

\noi
\begin{align*}
\|\G(v_1)-\G(v_2)\|_{C_{\eps,T} H^{1-\dl}_x}
&   \le
C_2  T^\frac{1}{2} \|v_1-v_2\|_{C_{\eps,T}H^{1-\dl}_x}
 \big(1 + \|\pmb{\Xi}\|_{\mathcal{X}^{\ta}}\big)
\big( 1 +  \| v\|_{C_{\eps,T} H^{1-\dl}_x}\big)^{k-1}, 
\end{align*}

\noi
as long as \eqref{Z5} is satisfied. Therefore, 
by choosing
$T = T(\| (\phi_0, \phi_1)\|_{\H_x^{1-\dl}}, \|\pmb{\Xi}\|_{\mathcal{X}^{\ta}}) >0$ sufficiently small, 
we conclude that $\G$ is a contraction in the ball 
$B_R \subset  C \big( [0,1] \times [0, T]; H^{1-\dl}(\T^2) \big)$ of radius
$R = 2C_1 \| (\phi_0, \phi_1)\|_{\H^{1-\dl}_x} + 1$.
At this point, the uniqueness holds only in the ball $B_R$
but by a standard continuity argument, 
we can extend the uniqueness to hold
in the entire $C \big([0,1] \times [0, T]; H^{1-\dl}(\T^2) \big)$. The regularity of the map $(\phi_0, \phi_1,
\pmb{\Xi}) \in \H^{1-\dl}(\T^2) \times \mathcal{X}^{\ta}(\T^2)
\mapsto 
v \in C( [0,1] \times [0, T]; H^{1-\dl}(\T^2))$ is easily obtained through similar estimates. We omit details. 
\end{proof}

We now prove Theorem \ref{THM1}. We recall that, with a slight abuse of notations, wave equations for which $\eps = 0$ are viewed as heat equations.

\noi
\begin{proof}[Proof of Theorem \ref{THM1}] Fix $k \ge 2$ and $(\phi_0,\phi_1) \in \H^s(\T^2)$ for $\frac{2k-3}{2k-2} \le s < 1$. Let $0 < \ta \ll 1-s$. 

\medskip

\noi
$\bullet$
{\bf Step 1: Construction of solutions.} Let $\pmb{\Xi} = (\Psi_{\eps}, : \! \Psi_{\eps} ^2 \! :, \cdots, : \! \Psi_{\eps} ^k \! : )_{\eps \in [0,1]}$. On the full probability set $\O_0$ constructed in Proposition \ref{PROP:sto}, we have that $\pmb{\Xi}_N, \pmb{\Xi} \in \mathcal{X}^{\ta}(\T^2)$ for any $N \in \N$ and $\pmb \Xi _N \to \pmb \Xi$ in $\mathcal{X}^{\ta}(\T^2)$ as $N \to \infty$.

By Proposition \ref{PROP:LWP} we get, for each $N \in \N$, a function $v_N = v_N(\eps,t)$ (resp. $v = v(\eps,T)$) which belongs to $C \big(  [0,1] \times [0,T] ; H^s(\T^2)  \big)$ for some almost surely positive time $0 <T \le 1$ (which is uniform in $N \in \N$ since $\sup_{N \in \N} \|\pmb{\Xi}_N\|_{\mathcal{X}^\ta} < \infty$) and that solves \eqref{SNLW5} with data given by $\pmb{\Xi}_N$ (resp. $\pmb{\Xi}$). Furthermore, by the continuity of the map $(\phi_0, \phi_1, \pmb{\Xi}) \mapsto v$ proved in Proposition \ref{PROP:LWP}, we deduce that $v_N$ converges to $v$ in $C ([0,T]; H^s(\T^2))$ as $N \to \infty$ on $\O_0$.

For any $\eps \in [0,1]$ and $N \in \N$, define $v_{\eps,N} = v_N(\eps, \cdot)$ and $v_{\eps} = v(\eps, \cdot)$. By construction, $v_{\eps,N}$ (resp. $v_\eps$) solves \eqref{SNLW11} (resp. \eqref{SNLW4}) with initial data $(\phi_0, \ind_{\eps >0} \phi_1 )$\footnote{Here and in what follows, with a slight abouse of notation, we understand $(\phi_0, \ind_{\eps >0} \phi_1 )$ as $(\phi_0, \phi_1)$ for $\eps >0$ and as $\phi_0$ for $\eps =0$.} and belongs to $C ([0,T]; H^s(\T^2))$. For $\eps \in [0,1]$ and $N \in \N$, let $u_{\eps,N} = \Psi_{\eps,N} + v_{\eps,N}$. Then, $u_{\eps,N}$ is the solution to \eqref{SNLW10} with initial data $(\phi_0, \ind_{\eps >0} \phi_1 )$. Then, by the above and Proposition \ref{PROP:sto}, $u_{\eps,N}$ converges to the process $u_\eps := \Psi_\eps + v_\eps$ in $C ([0,T]; H^{-\s}(\T^2))$, $\s >0$, as $N \to \infty$ on $\O_0$.

\smallskip
\noi
$\bullet$
{\bf Step 2: Convergence.} By the continuity of the map $\eps \mapsto v_\eps $ and Proposition \ref{PROP:sto}, we have that $u_\eps$ converges to $u_{\eps = 0}$ in $C ([0,T]; H^{-\s}(\T^2))$, $\s >0$, as $\eps \to 0$ on $\O_0$. This concludes the proof of Theorem \ref{THM1}.
\end{proof}

\noi
\subsection{Asymptotic global well-posedness}\label{SEC6}

The purpose of this subsection is to prove Theorem \ref{THM:GWP}. We recall the following global well-posedness result from \cite{MW1} adapted to our notations. See also \cite[Theorem 3.9]{TW} and \cite[Proposition 6.1]{Trenberth}.

\noi
\begin{lemma}\label{LEM:gwpheat}
Let $k \ge 2$ be an odd integer. Let $0< s <1$ and $\phi_0 \in H^s(\T^2)$. Let $v = v(\eps,t)$ be the solution to \eqref{SNLW15} with data given by $(\phi_0, \phi_1, \pmb{\Xi} )= (\phi_0, \phi_1, \Psi_{\eps}, : \! \Psi_{\eps} ^2 \! :, \cdots, : \! \Psi_{\eps} ^k \! : )_{\eps \in [0,1]}$ constructed in Theorem \ref{THM1}. Then, the function $v_0 = v(0, \cdot)$ exists globally in time. Moreover, for any fixed $T>0$, we have $v_0 \in C \big( [0,T]; H^s(\T^2) \big)$.
\end{lemma}

In order to prove Theorem \ref{THM:GWP}, we have to extend the existence time of the solution $\{v_\eps\}_{\eps \in (0,\eps_0]}$ to \eqref{SNLW5} restricted to the range $(0, \eps_0]$, for some fixed $\eps >0$ to be chosen later. The main idea to achieve this goal is to combine two ingredients: the global existence of $v_0$ provided by Lemma \ref{LEM:gwpheat} and the fact that $v_0$ approximates (locally-in-time) $v_\eps$ for $\eps \in (0, \eps_0]$ if $\eps_0$ is small enough. Hence, for any fixed target time $T>0$, we hope to use the quantity $\| v_0 \|_{C_T H^s_x}$, $0<s<1$, as an a priori bound on the growth in time of the relevant Sobolev norm of $v_\eps$, $\eps \in (0, \eps_0]$. To do so, we have to iterate the local well-posedness argument of Proposition \ref{PROP:LWP} starting from a small time $T>0$ and solve the following fixed point problem:

\noi
\begin{align}
 \begin{split}
v_\eps(t) = &  (\eps^{-2} + \dt) \D_\eps(t) v_\eps(T) + \D_{\eps}(t) \dt v_\eps(T)   \\
  & -
\,  \sum_{\ell=0}^k  {k\choose \ell} \int_{T}^t \eps^{-2} \D_{\eps}(t-t')
\big( \Xi_{\l}(\eps,t') \,  v_\eps(t')^{k-\ell}\big) dt', \quad t \in [T, + \infty).
\end{split}\label{SNLW6}
\end{align}

\noi
Solving \eqref{SNLW6} requires having a bound on $\dt v_\eps (T)$. However, by taking the time-derivative of \eqref{SNLW15} it is easy to see that the expression for $\dt v_\eps(T)$ contains a term of the form $\eps^{-4} \ft{\D_\eps}(T) \phi_0 $ which cannot be bounded uniformly in $\eps >0$ in any Sobolev space. Thus, we have no uniform in $\eps >0$ control over the $\H^s(\T^2)$ norm of $v_\eps$ in Proposition \ref{PROP:LWP}.
\begin{remark}
We note that the issue discussed in the above is a feature of the convergence problem at hand. Indeed, one can easily re-iterate the local well-posedness argument of Proposition \ref{PROP:LWP} for a {\it fixed value of $\eps \in (0,1]$} by working in Sobolev spaces ($\mathcal H^s(\T^2))_{s \in \R}$ .
\end{remark}

Fortunately, it turns out that the term $\D_{\eps}(t) \dt v_\eps(T)$ in \eqref{SNLW6} can be bounded uniformly in $\eps >0$. It order to capture this effect, we introduce the space $\V^s_{\eps_0}(\T^2)$ defined in \eqref{norm10} below which is a modification of the spaces $\H^s(\T^2)$ suitable for the convergence setting at hand.

In what follows, we first state a well-posedness result for an appropriate variation of \eqref{SNLW5}. Namely, for fixed $\eps_0 \in (0, 1]$, we consider the following problem:

\noi
\begin{align}
\begin{cases}
\eps^2 \dt^2 v + \dt v +(1-\Dl)v  +
 \sum_{\ell=0}^k {k\choose \ell} \, \Xi_\l   \, v^{k-\ell}
=0\\
(v,  \dt  v) ) |_{t=0} =(\phi_0(\eps),  \phi_1 (\eps) ),
\end{cases}
\quad (x,\eps,t) \in \T^2 \times (0,\eps_0] \times \R_+,
\label{SNLW7}
\end{align}

\noi
for given initial data $\eps \mapsto (\phi _0 (\eps), \phi _1 (\eps) )$ and a source $(\Xi_1 , \dots, \Xi_k)$
with the understanding that $\Xi_0 \equiv 1$. 

Given $\eps_0 \in (0,1]$ and $s \in \R$, we define the space $\mathcal{V}^s_{\eps_0}(\T^2)$ by the norm\footnote{Here, for a metric space $X$ and a Banach space and $(Y,\| \cdot \|)$, we denote by $C_b(X,Y)$ the Banach space of bounded and continuous functions endowed with the norm $\|f\|_{L^{\infty}} := \sup_{x \in X} \|f(x)\|$.}

\noi
\begin{align}
\| (\phi_0, \phi_1) \|_{\mathcal{V}^s_{\eps_0}} = \| \phi_0(\eps) \|_{C_b ((0,\eps_0]; H^s_x) } + \| \D_\eps(t) \phi_1(\eps) \|_{C_b ( (0,\eps_0] \times \R_+ ; H^s_x )}.
\label{norm10}
\end{align}

\noi
\begin{proposition}\label{PROP:LWP2}
Fix an integer $k \geq 2$, $\eps_0 \in (0,1]$ and $\dl \le \frac{1}{2(k-1)} $. Let $0 < \ta \ll \dl$. Then, the equation
\eqref{SNLW7} is locally well-posed
in $ \V_{\eps_0}^{1-\dl} (\T^2) \times \mathcal{X}^{\ta}(\T^2)$.
More precisely, 
given an enhanced data set: 
\begin{align}
(\phi _0, \phi _1,  \pmb{\Xi}) \in \V_{\eps_0}^{1-\dl}(\T^2) \times \mathcal{X}^{\ta}(\T^2), 
\label{data1}
\end{align}

\noi
with $\pmb{\Xi} = (\Xi_1, \Xi_2, \dots, \Xi_k )$ and any time $T_0 \in [0,1)$, there exist a time $ T(\|  (\phi_0, \phi_1) \|_{\V_{\eps_0}^{1-\dl}} , \|\pmb{\Xi}\|_{\mathcal{X}^{\ta}}) \in (0, 1]$, independent of $\eps_0$,
and a unique solution $\vec v := (v, \dt v)$ to \eqref{SNLW7} with initial data $(\phi_0,\phi_1)$ at time $T_0$ in the class
\begin{align}
C \big( [T_0, T]; \V_{\eps_0}^{1-\dl}(\T^2) \big).
\label{Z1SK}
\end{align}
In particular, the uniqueness of $\vec v$ 
holds in the entire class \eqref{Z1SK}.
Furthermore, the solution map 

\noi
\begin{align*}
(\phi_0, \phi_1,
\pmb{\Xi}) \in \V_{\eps_0}^{1-\dl }(\T^2) \times \mathcal{X}^{\ta}(\T^2)
\mapsto 
v \in C([T_0, T]; \V_{\eps_0}^{1-\dl}(\T^2))
\end{align*}

\noi
is locally Lipschitz continuous.
\end{proposition}

\noi
\begin{proof} Fix $\eps_0 \in (0,1]$. For simplicity, we assume that $T_0 = 0$. The integral formulation of \eqref{SNLW7} reads

\noi
\begin{align}
\vec v = \G (  v ) =  \begin{pmatrix} \G_1 ( v ) \\ \G_2 ( v) \end{pmatrix} \hspace{-1mm} ,
\label{m1}
\end{align}

\noi
where

\noi
 \begin{align}
 \begin{split}
\G_1 (v)(\eps,t) & = (\eps^{-2} + \dt) \D_\eps(t) \phi_0 (\eps) + \D_\eps(t) \phi_1 (\eps) \\
& \qquad  - \sum_{\ell=0}^k  {k\choose \ell} \I_\eps(t) 
\big( \Xi_\l(\eps,t) \,  v(\eps,t)^{k-\ell}\big),
\end{split}
\label{m2}
\end{align} 

\noi
and 

\noi
\begin{align}
\begin{split}
\G_2 ( v) (\eps, t) & =  (\eps^{-2} + \dt) \dt \D_\eps(t) \phi_0 (\eps) + \dt \D_\eps(t) \phi_1 (\eps) \\
& \qquad  - \sum_{\l = 0}^k \binom{k}{\l} \int_{0}^t \eps^{-2} \dt \D_\eps(t - t') (\Xi_{\l}(\eps,t') v(\eps,t')^{k - \l}) dt'.
\end{split}
\label{m3}
\end{align}

We fix $T>0$. By arguing as in the proof of Proposition \ref{PROP:LWP}, we obtain the following bound on $\G_1 (v)$:
\noi
\begin{align}
\|\G_1(v)\|_{C ( [0,T] \times (0,\eps_0 ] ; H^{1-\dl}_x)}
&   \les
 \| (\phi_0, \phi_1)\|_{\V_{\eps_0}^{1-\dl}}
+ 
  T^\frac{1}{2}
 \big(1 + \|\pmb{\Xi}\|_{\mathcal{X}^{\ta}}\big)
\big( 1 +  \| v\|_{C_{T} \V_{\eps_0}^{1-\dl}}\big)^k, 
\label{G110}
\end{align}

\noi
for any $v \in C \big( [0, T]; \V_{\eps_0}^{1-\dl}(\T^2) \big)$. 

We now turn our attention to $\G_2$. Namely, by \eqref{norm10}, we have to estimate ${(t , \eps, t_1) \mapsto \D_\eps(t_1)  \G_2 (v) ( \eps, t)}$ in $C \big([0,T] \times (0,\eps_0] \times \R_+; H^{1-\dl}  (\T^2) \big)$ for any $v$ in $C \big( [0, T]; \V_{\eps_0}^{1-\dl}(\T^2) \big)$. 

Let $v$ in $C \big( [0, T]; \V_{\eps_0}^{1-\dl}(\T^2) \big)$. We now obtain bounds on $\| \D_\eps (t_1) \G_2 (v) ( \eps, t) \|_{H^{1-\dl}_x}$ that are uniform in $(\eps,t, t_1) \in (0,\eps_0] \times [0,T] \times \R_+$. We fix $(\eps,t, t_1) \in (0,\eps_0] \times [0,T] \times \R_+$. By Lemma \ref{COR:OP} (iv) and the fact that $\D_\eps$ and $\dt \D_\eps$ commute, we  have

\noi
\begin{align}
\| \D_\eps(t_1) \dt \D_\eps (t) \phi_1 (\eps) \|_{H^{1-\dl}_x} & \leq \| \D_\eps(t_1) \phi_1 (\eps) \|_{H^{1-\dl}_x}  \les \| ( \phi_0, \phi_1)  \|_{\V^{1-\dl}_{\eps_0}}. \label{G10}
\end{align}

\noi
We have

\noi
\begin{align}
& \big\| \D_\eps(t_1) \int_{0}^t \eps^{-2} \dt \D_\eps(t - t') (\Xi_{\l}(\eps,t') v(\eps,t')^{k - \l})  dt' \big\|_{H^{1- \dl}_x} \notag \\
\begin{split}
& \quad  \leq  \big\| \eps^{-2} \D_\eps(t_1) \int_{0}^t  \Pll \dt \D_\eps(t - t') (\Xi_{\l}(\eps,t') v(\eps,t')^{k - \l}) dt' \big\|_{H^{1-\dl}_x} \\
& \qquad +  \big\| \D_\eps(t_1) \int_{0}^t \eps^{-2} \Phh \dt \D_\eps(t - t') (\Xi_{\l}(\eps,t') v(\eps,t')^{k - \l}) dt' \big\|_{H^{1-\dl}_x} =: \1 + \II
\label{G100}
\end{split}
\end{align} 

\noi
By Remark \ref{RMK:mulstr} and Lemma \ref{COR:OP} (i) and (ii) and by arguing as in the proof of Proposition \ref{PROP:LWP}, we have

\noi
\begin{align}
\begin{split}
 \1 & \les  \big\|  \int_{0}^t \Pll \dt \D_\eps(t - t') (\Xi_{\l}(\eps,t') v(\eps,t')^{k - \l})  dt' \big\|_{H^{1-\dl}_x} \\
& \les T^{\frac12} \| \Xi_{\l}v^{k - \l} \|_{C ( [0,T] \times (0, \eps_0] ; H^{-\dl}_x )} \\
& \les T^{\frac12}  \big(1 + \|\pmb{\Xi}\|_{\mathcal{X}_T^{\ta}}\big)
\big( 1 +  \| v\|_{C_{T} \V^{1-\dl}_{\eps_0}}\big)^k.
\end{split}
\label{G101}
\end{align}

\noi
Similarly, using Remark \ref{RMK:mulstr} along with Lemma \ref{COR:OP} (iii), we get

\noi
\begin{align}
\begin{split}
\II & \les \eps^{-1} \big\| \int_{0}^t  \Phh \dt \D_\eps(T - t') (\Xi_{\l}(\eps,t') v(\eps,t')^{k - \l}) dt' \big\|_{H^{-\dl}_x} \\
& \les \eps^{-1}  \int_{0}^t e^{- \frac{t-t'}{2 \eps^2}} dt' \, \| \Xi_{\l} v^{k - \l} \|_{C ( [0,T] \times (0, \eps_0] ; H^{-\dl}_x )} \\
&  \les \eps \big(1 + \|\pmb{\Xi}\|_{\mathcal{X}_T^{\ta}}\big) \big( 1 +  \| v\|_{C_{T} \V^{1-\dl}_{\eps_0}}\big)^k.
\end{split}
\label{G102}
\end{align}

\noi
Hence, from \eqref{G100}, \eqref{G101} and \eqref{G102}, we have

\noi
\begin{align}
\begin{split}
& \big\| \D_\eps(t_1) \int_{0}^t \eps^{-2} \dt \D_\eps(t - t') (\Xi_{\l}(\eps,t') v(\eps,t')^{k - \l})  dt' \big\|_{H^{1-\dl}_x} \\
& \qquad \les \big(1 + \|\pmb{\Xi}\|_{\mathcal{X}_T^{\ta}}\big) \big( 1 +  \| v\|_{C_{T} \V^{1-\dl}_{\eps_0}}\big)^k.
\end{split}
\label{G103}
\end{align}

\noi
We also have from Lemma \ref{COR:OP} (i), (ii) and (iii),

\noi
\begin{align}
\| \D_\eps(t_1) \eps^{-2} \dt \D_\eps(t) \phi_0 \|_{H^{1-\dl}_x} &  \le \| \dt \D_\eps(t)  \phi_0 \|_{H^{1-\dl}_x} \les \| ( \phi_0, \phi_1) \|_{\V_{\eps_0}^{1-\dl}}.
\label{G104}
\end{align}

At last, we look at the term $\D_\eps(t_1) \dt^2 \D_\eps(t) \phi_0$. By differentiating \eqref{M10} and by combining Lemma \ref{COR:OP} (i), (ii) and \eqref{M11}, we get

\noi
\begin{align}
\| \D_\eps(t_1) \dt^2 \D_\eps(t) \phi_0 \|_{H^{1-\dl}_x} & \les \eps^{-4} \| \D_\eps(t_1) \D_\eps(t) \phi_0 \|_{H^{1-\dl}_x} + \| \eps^{-2} \D_\eps(t_1) e^{-\frac{t}{2 \eps^2}} \dt S_\eps(t) \phi_0 \|_{H^{1-\dl}_x} \notag \\
& \qquad + \| \D_\eps(t_1) e^{-\frac{t}{2 \eps^2}} \dt^2 S_\eps(t) \phi_0 \|_{H^{1-\dl}_x}  \notag \\
& \les \| \phi_0 \|_{H^{1-\dl}_x} + \| \D_\eps(t) e^{-\frac{t}{2 \eps^2}} \dt^2 S_\eps(t) \phi_0 \|_{H^{1-\dl}_x} \label{G12}
\end{align}

\noi
By \eqref{mul1} and \eqref{L3} and by analyzing the symbols $\ft{\D_\eps}(n,t_1) e^{- \frac{t}{2 \eps^2}} \dt ^2\ft{S_\eps}(n,t) $, $n \in \Z^2$, we easily obtain the following bound:

\noi
\begin{align}
\| \D_\eps(t_1) e^{-\frac{t}{2 \eps^2}} \dt^2 S_\eps(t) \phi_0 \|_{H^{1-\dl}_x} \les \| \phi_0 \|_{H^{1-\dl}_x}. \label{G13}
\end{align}

\noi
  Hence, from \eqref{G12} and \eqref{G13}, we have

\noi
\begin{align}
\| \D_\eps(t_1) \dt^2 \D_\eps(t) \phi_0 \|_{H^{1 - \dl}_x}  \les \| (\phi_0, \phi_1) \|_{\V_{\eps_0}^{1- \dl}}.
\label{G105}
\end{align}

\noi
Combining \eqref{G10}, \eqref{G103}, \eqref{G104} and \eqref{G13} yields

\noi
\begin{align}
\| \D_\eps (t_1) \G_2 (v) ( \eps, t) \|_{H^{1-\dl}_x} \les \| (\phi_0, \phi_1) \|_{\V_{\eps_0}^{1- \dl}} + T^{\frac12} \big(1 + \|\pmb{\Xi}\|_{\mathcal{X}_T^{\ta}}\big) \big( 1 +  \| v\|_{C_{T} \V^{1-\dl}_{\eps_0}}\big)^k,
\label{G106}
\end{align}

\noi
uniformly in $(\eps,t, t_1) \in (0,\eps_0] \times [0,T] \times \R_+$. By \eqref{G106}, the dominated convergence theorem and since the map $(\eps,t) \in [0,\eps_0] \times [0,T] \mapsto  \big( \phi_0 (\eps), \phi_1 (\eps), \{ \Xi_\ell (\eps,t) \}_{1 \le \ell \le k}, v(\eps,t)\big)$ is continuous and the map $(\eps,t) \mapsto \D_\eps(t)$ is smooth (Lemma \ref{LEM:F3}), we deduce the estimate

\noi
\begin{align}
\begin{split}
& \| \D_\eps (t_1) \G_2 (v) ( \eps, t) \|_{C ( [0,T] \times (0,\eps_0 ] ; H^{1-\dl}_x)} \\
& \qquad \qquad \quad  \les \| (\phi_0, \phi_1) \|_{\V_{\eps_0}^{1- \dl}} + T^{\frac12} \big(1 + \|\pmb{\Xi}\|_{\mathcal{X}_T^{\ta}}\big) \big( 1 +  \| v\|_{C_{T} \V^{1-\dl}_{\eps_0}}\big)^k.
\end{split}
\label{G107}
\end{align}

Hence, by \eqref{G110} and \eqref{G107}, we get

\noi
\begin{align}
\| \G(v) \|_{C_T \V_{\eps_0}^{1-\dl}} \les  \| (\phi_0, \phi_1)\|_{\V_{\eps_0}^{1-\dl+\ta}}
+ 
  T^\frac{1}{2}
 \big(1 + \|\pmb{\Xi}\|_{\mathcal{X}^{\ta}}\big)
\big( 1 +  \| v\|_{C_{T} \V_{\eps_0}^{1-\dl}}\big)^k,
\label{G111}
\end{align}

\noi
for any $v \in C \big( [0, T]; \V_{\eps_0}^{1-\dl}(\T^2) \big)$. This proves that $\G$ maps balls of $C \big( [0, T]; \V_{\eps_0}^{1-\dl}(\T^2) \big)$ into themselves. We obtain a difference estimate along the same lines. The remaining of the proof follows as in the proof of Proposition \ref{PROP:LWP}.
\end{proof}

In the next lemma, we analyze the behavior near $\eps = 0$ of the (first coordinate of the) solutions to \eqref{SNLW7} on some (possibly large) time interval $[0,T]$, $T>0$, constructed in Proposition \ref{PROP:LWP2}. Namely, we prove that they converge to $v_0$ given by Lemma \ref{LEM:gwpheat}.
 
\noi
\begin{lemma}[long time approximation]\label{LEM:GWP1}
Fix an integer $k \ge 2$ and $\frac{2k-3}{2k-2} \le s <1$. Fix ${\eps_0 \in (0,1]}$ and let $0 < \ta \ll 1- s$.  Consider the data $(\phi_0, \phi_1) \in \H^{s+\ta} (\T^2)$ \textup{(}and hence ${\eps \in (0,\eps_0]  \mapsto (\phi_0, \ind_{\eps >0} \phi_1) \in \V^{s+\ta} _{\eps_0}}$\textup{)} and $\pmb{\Xi} = (\Xi_1, \Xi_2, \dots, \Xi_k ) \in \mathcal{X}^\ta_T(\T^2)$.

Let $T >0$ and assume that there exists a solution $v$ to \eqref{SNLW7} on $(0,\eps_0] \times [0,T]$ such that

\noi
\begin{align*}
v \in C \big( (0,\eps_0] \times [0,T] ; H^s (\T^2)  \big)
\end{align*} 

\noi
with data $\eps \in (0,\eps_0]  \mapsto (\phi_0, \ind_{\eps >0} \phi_1)$ and ${\pmb{\Xi} = (\Psi_{\eps}, : \! \Psi_{\eps} ^2 \! :, \cdots, : \! \Psi_{\eps} ^k \! : )_{\eps \in [0,\eps_0]}  \in \mathcal{X}^\ta_T(\T^2)}$. Let ${v_0 \in C \big([0,T] ; H^s (\T^2)  \big)}$ be the solution to \eqref{SNLW5} on $ \{0\} \times [0,T]$ given by Lemma \ref{LEM:gwpheat}.

Then, there exists a constant ${K = K( \| v_0 \|_{C_T H^s_x} ,\| v  \|_{C ( (0,\eps_0] \times [0,T];  H_x^{s+\ta})}) >0}$ and $\g >0$ such that

\noi
\begin{align}
\| v(\eps,\cdot) - v_0 \|_{C_T H^s_x} \le K \max \big( \eps^{\g} ,  \|  \pmb \Xi (\eps) - \pmb \Xi (0)  \|_{C_T( W^{-\ta}_x)^{\otimes k}}, \|  (\phi_0, \phi_1) \|_{\H^{s+\ta}_x}\big).
\label{G1}
\end{align}

\noi
for any $\eps \in (0, \eps_0]$.
\end{lemma}

\noi
\begin{proof}
We follow an argument in \cite{OOT1, OOT2}; by adding an exponential weight to the local well-posedness norm, one is able to globalize directly a bound that only holds locally in time in the original norm. For $T > 0$ and let $\ld >0$ to be chosen later, we define the norm $\| \cdot \|_{S_{\ld,T}}$ on $C( [0,T], H^s(\T^2))$ by 

\noi
\begin{align}
\| v \|_{S_{\ld,T}} \deff \| e^{-\ld t} v \|_{C_T H^s_x}
\label{G2}
\end{align}

\noi
Note that we have the inequalities

\noi
\begin{align}
\| v \|_{S_{\ld,T}} \le \| v \|_{C_T H^s_x} \le e^{\ld T} \| v \|_{S_{\ld,T}}.
\label{G3}
\end{align}

\noi
Fix $\eps \in (0, \eps_0]$. Note that $v = v(\eps, \cdot)$ and $v_0$ verify \eqref{SNLW15} (with $\eps = 0$ for $v_0$). By \eqref{G3}, we have

\noi
\begin{align}
\| v(\eps, \cdot) - v_0 \|_{S_{\ld,T}} & \le \| (P_\eps - P_0) (\phi_0,\phi_1)  \|_{S_{\ld,T}} \\
& \qquad  + \sum_{\l = 0}^k \binom{k}{\l} \|  \I_\eps \big( \Xi_\l(\eps, \cdot) v(\eps,\cdot) ^{k - \l}\big)  - \I_0 \big( \Xi_\l(0, \cdot) v_0 ^{k - \l}\big) \|_{S_{\ld,T}} \notag \\
& \les \| (P_\eps - P_0) (\phi_0,\phi_1)  \|_{C_T H^s_x} + \max_{0 \le \l \le k} \|  (\I_\eps-\I_0) \big( \Xi_\l(\eps, \cdot) v(\eps,\cdot) ^{k - \l}\big) \|_{C_T H^s_x}  \notag \\
& \qquad +  \max_{0 \le \l \le k}  \| \I_0 \big( \Xi_\l(\eps, \cdot) v(\eps,\cdot) ^{k - \l} - \Xi_\l( 0, \cdot) v_0^{k - \l} \big) \|_{S_{\ld,T}} \notag \\
& =  \1 + \II + \III. \label{G4}
\end{align}

\noi
By \eqref{detconv}, we get

\noi
\begin{align}
\1 \les \eps^{\frac{\ta}{2}} \|  (\phi_0, \phi_1) \|_{\H^{s+\ta}_x}
\label{G5}
\end{align}

\noi
We also have from \eqref{detconv} and computations similar to those in the proof of Proposition \ref{PROP:LWP}, the following bounds:

\noi
\begin{align}
\begin{split}
\II & \les T^{\frac12} \eps^{\frac{\ta}{2}} \| \Xi_\l(\eps, \cdot) v( \eps, \cdot)^{k - \l} \|_{C ( (0,\eps_0] \times [0,T];  H_x^{s+\ta})} \\
& \les T^{\frac12}  \eps^{\frac\ta2} (1 + \| \pmb{\Xi} \|_{\mathcal{X}_T^{\ta}}) (1 + \| v  \|_{C ( (0,\eps_0] \times [0,T];  H_x^{s+\ta})})^k 
\end{split}
\label{G6}
\end{align}

We now estimate the term $\III$ in the case $\l = 0$ for convenience. From Lemma \ref{LEM:heat} and proceeding as in \eqref{Z3}, we have

\noi
\begin{align*}
\| \I_0 \big(v(\eps,\cdot) ^k - v_0 ^k \big)  \|_{S_{\ld,T}} & = \sup_{0 \le t \le T}  \big\|  \int_{0}^t e^{-\ld (t-t') } P_0(t-t') \big( e^{-\ld t'} (v(\eps,t')^{k} - v_0(t')^k ) \big) dt'  \|_{H^s_x}  \\
& \les \sup_{0 \le t \le T} \int_{0}^t e^{- \ld(t-t')  } (t-t')^{-\frac12} dt' \cdot \| e^{-\ld t} (v(\eps,\cdot)^{k} - v_0^k ) \|_{C_T H^{s-1}_x}  \\
& \les \frac{1}{\sqrt{\ld}} \|v(\eps,\cdot) - v_0 \|_{S_{\ld,T}} (1 + \| v_0 \|_{C_T H^s_x})^{k-1} (1 + \| v \|_{C( (0,\eps_0] \times [0,T]  H^s_x )})^{k-1}.
\end{align*}

\noi
By similar arguments, we get the bound

\noi
\begin{align}
\III \les_{v, v_0} \eps^{\frac\ta2} + \| \pmb \Xi (\eps) - \pmb \Xi (0)  \|_{C_T (W^{-\ta}_x)^{\otimes k}} +  \frac{1}{\sqrt{\ld}} \|v(\eps, \cdot) - v_0 \|_{S_{\ld,T}}.
\label{G7}
\end{align}

Thus, combining \eqref{G4}, \eqref{G5}, \eqref{G6} and \eqref{G7}, we deduce the existence of ${K = K( \| v_0 \|_{C_T H^s_x} ,\| v  \|_{C ( (0,\eps_0] \times [0,T];  H_x^{s+\ta})}) >0}$ such that

\noi
\begin{align}
\| v(\eps,\cdot) - v_0 \|_{S_{\ld,T}} \le K  \eps ^{\frac\ta3} + \| \pmb \Xi (\eps) - \pmb \Xi (0)  \|_{C_T (W^{-\ta}_x)^{\otimes k}} + \frac{K}{\sqrt{\ld}}  \| v_\eps - v_0 \|_{S_{\ld,T}}.
\end{align} 

\noi
This leads to

\noi
\begin{align}
\| v(\eps, \cdot) - v_0 \|_{S_{\ld,T}} \le 2 K \max \big( \eps^{\frac \ta 3} ,  \| \pmb \Xi (\eps) - \pmb \Xi (0)\|_{C_T (W^{-\ta}_x)^{\otimes k}} \big),
\label{G8}
\end{align}

\noi
upon choosing $\ld  = (2 K(\eps_0,T))^2$. We hence deduce \eqref{G1} from the second inequality in \eqref{G3} and \eqref{G8}.

\end{proof}

We now prove Theorem \ref{THM:GWP}. 

\noi
\begin{proof}[Proof of Theorem \ref{THM:GWP}]
Fix $k \ge 2$. For convenience, we only prove Theorem \ref{THM:GWP} for $T=1$. Let $(\phi_0, \phi_1) \in \H^s(\T^2)$ with $\frac{2k-3}{2k-2} < s < 1$. Let $0 < \ta \ll 1 - s$. We first fix a set of full $\PP$-probability $\O_0$ such that on $\O_0$, the following conditions hold: (i) ${\pmb{\Xi} = (\Psi_{\eps}, : \! \Psi_{\eps} ^2 \! :, \cdots, : \! \Psi_{\eps} ^k \! : )_{\eps \in [0,1]} \in \mathcal{X}^\ta(\T^2)}$ (by Proposition \ref{PROP:sto}) and (ii) $v_0 \in C \big([0,1] ; H^s (\T^2)  \big)$ the solution to \eqref{SNLW5} on $ \{0\} \times [0,1]$, with data given by $(\phi_0, \pmb{\Xi})$, given by Lemma \ref{LEM:gwpheat}. For the remaining of the proof, we work on $\O_0$ without any mention to it.

%
%
%
%
%
%

In the rest of the proof, we use the following notation: for $\eps' \in (0,1]$, we denote by $v^{\eps'} = v^{\eps'}(\eps,t,x)$ the solution to the problem \eqref{SNLW7} (with some given initial data) on the set $(x, \eps, t) \in \T^2 \times (0,\eps'] \times \R_+$. We fix $\eps_0 =1$. Consider the problem \eqref{SNLW7} with data $\eps \in (0, \eps_0] \mapsto (\phi_0, \ind_{\eps>0}\phi_1, \pmb \Xi (\eps)) \in \V_{\eps_0}^{s+\ta} \times \mathcal{X}^\ta $. By Proposition \ref{PROP:LWP2}, there exists a solution $v^{\eps_0} \in C \big( [0,T_0]; \V_{\eps_0}^s(\T^2) \big)$ to \eqref{SNLW7} on $(0, \eps_0] \times [0,T_0]$, $T_0>0$, with data $\eps \in (0, \eps_0] \mapsto (\phi_0, \ind_{\eps>0}\phi_1, \pmb \Xi (\eps)) \in \V_{\eps_0}^s \times \mathcal{X}^\ta $. By Lemma \ref{LEM:GWP1} and the assumptions (i) and (ii), we have

\noi
\begin{align}
\|v^{\eps_0}\|_{C((0, \eps_1] \times [0,T_0]; H^s_x )} \le \|v_0\|_{C_{T=1} H^s_x} +1,
\label{C11}
\end{align}

\noi
for $\eps_1 >0$ small enough. Denote by $v^{\eps_1}$ the solution to \eqref{SNLW7} on $(0, \eps_1] \times [0,T_0]$ with the same data. Then, we have $v^{\eps_1}  (\eps, \cdot) = v^{\eps_0} (\eps, \cdot)$ for $\eps \in (0, \eps_1$]. Hence, by arguing as in the proof of Proposition \ref{PROP:LWP2}, we have

\noi
\begin{align}
\|v^{\eps_1}\|_{C([0,T_0]; \V^s_{\eps_1} )}  & = \|v^{\eps_0}\|_{C([0,T_0]; \V^s_{\eps_1} )} \notag \\
&  \les \| (\phi_0, \phi_1)\|_{\H^{s}_x}
+ 
 \big(1 + \|\pmb{\Xi}\|_{\mathcal{X}^{\ta}}\big)
\big( 1 +  \| v_0 \|_{C_{T=1} H_x^{s}}\big)^k, \label{C12}
\end{align}

We can now apply the local well-posedness statement of Proposition \ref{PROP:LWP2} to extend $v^{\eps_1}$ to a larger time interval $[T_0, T_1]$ whose size $T_1 - T_0$ only depends on the (fixed) constants $\| (\phi_0, \phi_1)\|_{\H^{s}_x}$, $\|\pmb{\Xi}\|_{\mathcal{X}^{\ta}}$ and $\| v_0 \|_{C_{T=1} H_x^{s}}$. Thus, iterating this argument allows us to reach the target time $T= 1$ and provides $\eps_\star >0$ such that $v^{\eps_\star}  \in C \big( [0,T]; \V_{\eps_\star}^s(\T^2) \big) $ solves \eqref{SNLW7} on $(0, \eps_\star ] \times [0,1]$ with data $\eps \in (0, \eps_\star] \mapsto (\phi_0, \ind_{\eps>0}\phi_1, \pmb \Xi (\eps))$. By Proposition \ref{LEM:GWP1}, we obtain that $v^{\eps_\star} (\eps, \cdot)$ converges to $v_0$ on $[0,1]$ as $\eps \to 0$. Together with Proposition \ref{PROP:sto}, we have that $\Psi_\eps \to \Psi$ in $C\big([0,1]; H^{-\s}(\T^2)\big)$, $\s >0$ as $\eps \to 0$. This proves the convergence of $u_\eps = \Psi_\eps + v_\eps$ to $u_{0}  = \Psi_0 + v_0$ in $C\big([0,1]; H^{-\s}(\T^2)\big)$, $\s >0$ as $\eps \to 0$.
\end{proof}

\newpage

\section{Sine-Gordon models}

In this section, we prove Theorem \ref{THM:sinGWP}. As in the polynomial case, our argument consists in two steps: construction and convergence as $\eps \to 0$ of stochastic objects -- here essentially $\Theta_\eps$ as in \eqref{Ups} and convergence as $\eps \to 0$ at the level of the deterministic remainders \eqref{SdSG3}.

\subsection{Estimates on spatial covariance functions}\label{SUBSEC:sin1}
We aim to construct the stochastic objects $\{\U_{\eps,N}\}_{N \in \N}$ for $\eps \in [0,1]$ defined in \eqref{Ups}. Due to the non-Gaussianity of $\U_{\eps,N}$, we cannot use standard tools such as hypercontractivity (Lemma \ref{LEM:hyp}) to estimate its high moments as in Proposition \ref{PROP:sto} in order to construct the limiting (as $N \to \infty$) object $\U_\eps$. We instead perform a computation ``by hand" relying on precise estimates for the (spatial) covariance function of the stochastic convolution \eqref{convo10} as in \cite{HS} and \cite{ORSW1, ORSW2} in the cases $\eps =0$ and $\eps = 1$, respectively. We defer the construction and convergence of the stochastic objects $\{\U_{\eps,N}\}_{N \in \N}$ to the next subsection.

In this subsection, we prove estimates on the spatial covariance functions. The main technical challenge here is to study this covariance function for all $\eps \in [0,1]$ and its behavior in the limit $\eps \to 0$; see Lemmas \ref{LEM:cov} and Lemma \ref{LEM:cov2} below.

\smallskip

\noi
{\bf $\bul$ Estimates on covariance functions.} Let $N \in \N$. We define the spatial covariance function of the stochastic convolutions \eqref{convo10} and \eqref{sto1} by
 
 \noi
 \begin{align}\label{covfun}
 \Gamma_{\eps,N} (t,x-y) := \E [\Psi_{\eps,N} (t,x) \Psi_{\eps,N}(t,y)].
 \end{align}
 
 \noi
and
 
 \noi
 \begin{align}
 \Gamma_\eps(t,x-y)  := \E [ \Psi_{\eps}(t,x) \Psi_{\eps}(t,y) ],
 \label{covfun2}
 \end{align}
 
 \noi
respectively. Here,  $t \ge 0$ and $\eps \in [0,1]$ and $\Psi_\eps$ and $\Psi_{\eps,N}$ are as in \eqref{sto1} and \eqref{convo10}, respectively. It is easy to see that we have $\Gamma_{\eps,N} = \P_N ^2 \Gamma_{\eps}$.

The aim of this part of the section is to study the precise dependence in the parameter $\eps \in [0,1]$ of the covariance functions defined in \eqref{covfun} and \eqref{covfun2}; see Lemmas \ref{LEM:cov} and \ref{LEM:cov2}. Before proving the aforementioned, we first introduce some notations and prove a preparatory lemma.

 Fix $0 < \al \le 2$ and let $J_\al$ be the convolution kernel of the Fourier multiplier $\jb{\nb}^{-\al}$. Namely, $J_\al$ is the distribution

\noi
\begin{align}
J_\al(x) := \frac{1}{2 \pi} \sum_{n \in \Z^2} \frac{e_n(x)}{ \jb{n}^\al},
\label{bessel}
\end{align} 

\noi
for $x \in \T^2$. The distribution $J_{2}$ is the Green function\footnote{Namely, the solution to the linear equation $(1-\Dl)G = \dl_0$ on $\T^2$, where $\dl_0$ is the Dirac delta function at the origin.} for $1-\Dl$ and will be denoted by $G$, i.e. we write

\noi
\begin{align}
G(x) := \frac{1}{2 \pi} \sum_{n \in \Z^2} \frac{e_n(x)}{ \jb{n}^2},
\label{besselG}
\end{align} 

\noi
for $x \in \T^2$. We also define the distribution $H$ by

\noi
\begin{align}
H(t,x) := \frac{1}{2 \pi} \sum_{n \in \Z^2} \frac{1 - e^{-2 t \jb n ^2}}{ \jb n ^2 } e_n (x),
\label{heatgreen}
\end{align}

\noi
for $t \in \R_+$ and $x \in \T^2$.

The next lemma describes the behavior of the operator $J_\al$ on the spatial physical side; see \cite[Lemma 2.2]{ORSW1} for a proof.

\begin{lemma}\label{LEM:bessel}
Let $0 < \al < 2$. There exists a smooth function $R$ on $\T^2$ and a constant $c_\al$ such that 

\noi
\begin{align}
J_\al (x) = c_\al |x|^{\al - 2} + R(x)
\end{align}

\noi
for all $x \in \T^2 \setminus \{ 0 \} \cong [-\pi, \pi] \setminus \{ 0\}$.
\end{lemma}

Let $\log_-$ be the function defined on $\R_{>0}$ given by

\noi
\begin{align}
\log_-(r) = \begin{cases} \log(r) & \quad \text{if $0 < r \le 1$}  \\
 0 & \quad \text{if $r >1$}.  \end{cases}
 \label{log_minus}
\end{align}

We record in the following Lemma some useful estimates on the Green function $G$ and the function $H$.

\noi
\begin{lemma}\label{LEM:green}
Let $G$ and $H$ be as in \eqref{besselG} and \eqref{heatgreen}, respectively. Then, the following estimates hold.

\smallskip

\noi
\textup{(i)} Let $N \ge 1$, $0 \le t \le 1$ and $x \in \T^2$. Then, we have

\noi
\begin{align}
& \P_N ^2 G(x) \approx - \frac{1}{2 \pi} \log_- \big( |x| + N^{-1} \big), \label{bb1} \\
& \P_N^2 H(t,x) \approx  - \frac{1}{2 \pi} \log_- \Big( \frac{  |x| + N^{-1}   }{ |x| + t^{\frac12}}  \Big). \label{bb2}
\end{align}

\smallskip

\noi
\textup{(ii)} Let $N_1, N_2 \in \R_+$ such that $N_2 \ge N_1 \ge 1$, $0 \le t \le 1$ and $x \in \T^2$. Then, we have

\noi
\begin{align}
& | (\P_{N_j}^2 G(x) - \P_{N_2} \P_{N_1}) (\P_N^2 G)(x)   | \les \min \big( 1 - \log_-( |x| + N_2^{-1} ) ,  N_1^{-4} |x|^{-4}  \big), \label{bb3} \\
& |( \P_{N_j}^2  -  \P_{N_2} \P_{N_1} ) (\P_N^2 H)(t,x)   | \les \min\Big( 1 - \log_- \Big( \frac{  |x| + N_2^{-2}   }{ |x| + t^{\frac12}}  \Big) ,   N_1^{-4} |x|^{-4}  \Big), \label{bb4}
\end{align}

\noi
uniformly in $N \ge 1$ and for $j=1,2$.
\end{lemma}

\noi
\begin{proof}
The proofs of the estimates in items (i) and (ii) pertaining to the Green function $G$ can be found in \cite[Lemma 2.3, Remark 2.4]{ORSW1}.\footnote{Smooth cutoffs are needed in the argument in \cite{ORSW1} in order to perform integration-by-parts in Fourier space.} Let us prove the estimate \eqref{bb2} on $H$. For each $N \ge 1$, we denote by $G_N$ and $H_N$ the variants of $\P_N^2 G$ and $\P_N^2 H$ given by

\noi
\begin{align*}
& G_N(x) = \frac{1}{2 \pi} \sum_{\substack{n\in \Z^2\\ \jb n \le N}} \frac{1}{ \jb{n}^2}e_n(x) & \text{and} &
& H_N(t,x) = \frac{1}{2 \pi} \sum_{\substack{n\in \Z^2\\ \jb n \le N}} \frac{1-e^{-2t \jb n^2}}{ \jb{n}^2} e_n(x),
\end{align*}

\noi
for $t \in \R_+$ and $x \in \T^2$. Note that $G_N$ (resp. $H_N$) simply consists in $\P_N^2 G$ (resp. $\P_N^2 H$) but where the smooth cutoff function $\chi_N(n)$ is replaced by the sharp cutoff $\ind_{\jb n \le N}$. Moreover, it is easy to see that

\noi
\begin{align}
\big| G_N(x) - \P_N^2 G(x) \big| + \big| H_N(t,x) - \P_N^2 H(t,x) \big| \les 1, \label{same}
\end{align}

\noi
uniformly in $t \ge 0$, $x \in \T^2$ and $N \ge 1$. Hence, in order to get the bound \eqref{bb2}, it suffices to show \eqref{bb2} with $\P_N^2 H$ replaced with $H_N$. Namely, we want to prove 

\begin{align}
H_N(t,x) \approx  -\frac{1}{2\pi} \log_- \Big( \frac{  |x| + N^{-1}   }{ |x| + t^{\frac12} + N^{-1}  }  \Big), \label{bb2b}
\end{align}

\noi
for $t \ge 0$, $x \in \T^2$ and $N \ge 1$.

 We prove the latter in what follows. Fix $0< t \le 1$,\footnote{The claimed bound is evident for $t = 0$.} $x \in \T^2$ and $N \ge 1$. We consider two cases.

\smallskip

\noi
{\bf Case 1: $N \le t^{-\frac12}$.}  Then, by the mean value theorem, we have

\noi
\begin{align*}
|H_N(t,x)| & \les  \Big| \sum_{\substack{n \in \Z^2 \\ \jb n \le N}}    \frac{1 - e^{ - 2 t \jb{n}^2 }}{ \jb{n}^2 } e_n (x) \Big| \\
& \les t \sum_{\substack{n \in \Z^2 \\ \jb n \le N}} 1 \les t N^2 \les 1.
\end{align*}

\noi
This shows \eqref{bb2b} in this case as both sides of the equality are of size $O(1)$.

\smallskip

\noi
{\bf Case 2: $t^{-\frac12} \le N$.} We now break up $H_N(t,x)$ as follows.

\noi
\begin{align}
\begin{split}
H_N(t,x) & = \sum_{\substack{n \in \Z^2 \\ \jb n < t^{-\frac12}}}   \frac{1 - e^{ - 2 t \jb{n}^2 }}{ \jb{n}^2 } e_n (x) + \sum_{\substack{n \in \Z^2 \\ t^{-\frac12} \le \jb n \le N}}   \frac{1 - e^{ - 2 t \jb{n}^2 }}{ \jb{n}^2 } e_n (x)\\
& =: \1 + \II.
\end{split}
\label{alo0}
\end{align}

\noi
By Case 1, we have the estimate

\noi
\begin{align}
| \1 | \les 1.
\label{alo1}
\end{align}

\noi
In order to bound $\II$, we write 

\noi
\begin{align}
\begin{split}
\II & = \sum_{\substack{n \in \Z^2 \\ t^{-\frac12} \le \jb n \le N}}   \frac{1}{ \jb{n}^2 } e_n (x) - \sum_{\substack{n \in \Z^2 \\ t^{-\frac12} \le \jb n \le N}}   \frac{e^{ - 2 t \jb{n}^2 }}{ \jb{n}^2 } e_n (x) \\
& =: \III + \IV.
\end{split}
\label{alo2}
\end{align}

\noi
By writing the sum $\sum_{t^{-\frac12} \le \jb n \le N}$ as the difference $\sum_{\jb n \le N} - \sum_{\jb n < t^{-\frac12}}$ and using \eqref{bb1}, we deduce that

\noi
\begin{align}
\III \approx - \frac{1}{2 \pi} \log \Big( \frac{  |x| + N^{-1}   }{ |x| + t^{\frac12}}  \Big) \approx - \frac{1}{2 \pi} \log_- \Big( \frac{  |x| + N^{-1}   }{ |x| + t^{\frac12}}  \Big).
\label{alo3}
\end{align}

\noi
Here, we have replaced $\log$ by $\log_-$ by noting that $ \frac{  |x| + N^{-1}   }{ |x| + t^{\frac12}}  \ll 1 $ in the regime $t^{-\frac12} \ll N$. Lastly, we estimate $\IV$. Let $n = (n_1, n_2) \in \Z^2$. Let $C(n) \subset \R^2$ be the only square which has $n$ as a corner and lies inside the disc of radius $|n|$. For instance, if $n_1,n_2 >0$, then $C(n) = [n_1 -1, n_1] \times [n_2 - 1, n_2]$. Then, by the fact that the function $r \mapsto \frac{e^{-2t(1+r^2)}}{1+r^2}$ is non-increasing, that $|C(n)| =1$ and a change of variable, we get

\noi
\begin{align}
\begin{split}
|\IV| & = \Big|\sum_{\substack{n \in \Z^2 \\ t^{-\frac12} \le \jb n \le N}}  e_n (x)  \int_{C(n)} \frac{e^{ - 2 t \jb{n}^2 }}{ \jb{n}^2} dy \Big| \\
& \les \sum_{\substack{n \in \Z^2 \\ t^{-\frac12} \le \jb n \le N}}   \int_{C(n)} \frac{e^{ - 2 t \jb{y}^2 }}{ \jb{n}^2} dy \les \int_{t^{-\frac12} \les \jb x \les N}  \frac{e^{ - 2 t \jb{y}^2 }}{ \jb{y}^2 } dx \\
& \les e^{-2t } \int_{t^{-\frac12} \les r \les N} \frac{e^{ - 2 t r^2}}{ r } dr \les \int_{1 \les y \les t^{\frac12} N } \frac{e^{-s^2}}{s} ds \les 1,
\end{split}
\label{alo4}
\end{align}

\noi
uniformly in $N \ge 1$. By collecting \eqref{alo0}, \eqref{alo1}, \eqref{alo2}, \eqref{alo3} and \eqref{alo4}, we have

\noi
\begin{align*}
H_N(t,x) \approx - \frac{1}{2 \pi} \log_- \Big( \frac{  |x | + N^{-1}   }{ |x| + t^{\frac12}}  \Big),
\end{align*}

\noi
which also corresponds to \eqref{bb2b} in this case. 

We now prove \eqref{bb4}. The first bound in \eqref{bb4} is easily proved by arguing as in the above proof of \eqref{bb2}. We hence aim at proving 

\noi
\begin{align}
| (\P_{N_j}^2 H(t,x) -  \P_{N_2} \P_{N_1})(\P_N^2 H)(t,x)   | \les  N_1^{-4} |x|^{-4}, \label{bb5}
\end{align} 

\noi
uniformly in $N \ge 1$.

We fix $j = 2$ for convenience. By the Poisson formula \eqref{poisson}, we have that

\noi
\begin{align}
\begin{split}
& (\P_{N_j}^2 H(t,x) -  \P_{N_2} \P_{N_1})(\P_N H)(t,x) \\
& \qquad \qquad  = \frac{1}{2 \pi} \sum_{n \in \Z^2} (\chi_{N _2}^2(n) - \chi_{N_2}(n) \chi_{N_1}(n)) \chi_N^2(n) \frac{1 - e^{-2 t \jb n ^2}}{ \jb n ^2 } e_n (x) \\
& \qquad \qquad  = \frac{1}{2 \pi} \sum_{m \in \Z^2} f_{N_1, N_2}(t,x + 2\pi m),
\end{split}
\label{bb5b}
\end{align}

\noi
with 

\noi
\begin{align*}
\ft{f_{N_1,N_2}}(t,x) = \chi_N^2(n) (\chi_{N _2}^2(\xi) - \chi_{N_2}(\xi) \chi_{N_1}(\xi)) \frac{1 - e^{-2 t \jb \xi ^2}}{ \jb \xi ^2 }.
\end{align*}

\noi
Moreover, by integration by parts, we have

\noi
\begin{align}
\begin{split}
|f_{N_1, N_2}(t,x)| & \sim | x |^{-4} \Big| \int_{\R^2} e^{i x \cdot \xi} \Dl^2_{\xi} \Big( \chi_N^2(n) (\chi_{N _2}^2(\xi) - \chi_{N_2}(\xi) \chi_{N_1}(\xi)) \frac{1 - e^{-2 t \jb \xi ^2}}{ \jb \xi ^2 } \Big) d \xi \Big| \\
& \les |x|^{-4} N_1^{-4}. 
\end{split}
\label{bb6}
\end{align}

\noi
In the above, we used the fact that the integrand in the \eqref{bb6} is supported in the set $\{ \xi \in \R^2: N_1 \les |\xi| \les N_2 \}$ and the bound $e^{-2t \jb \xi ^2} \les_{p} (t \jb \xi ^2)^{-p}$ to prove that the contribution of derivatives of order $p$ hitting the factor $1 - e^{-2 t \jb \xi ^2}$ is bounded by $N_1^{-p}$ for any $p \in \N$.

Summing in $m \in \Z^2$ for $x \in \T^2\setminus \{0\} = (-\pi, \pi] \setminus \{0\}$ in \eqref{bb5b} together with \eqref{bb6} yields the desired bound \eqref{bb5}.
\end{proof}

%

We first record the following estimates on $\G_{\eps,N}$ as in \eqref{covfun}.

 \noi
 \begin{lemma}\label{LEM:cov}
Fix $\eps \in [0,1]$. We have the following bounds on the covariance function \eqref{covfun}:

\smallskip

\noi
\textup{(i)} For any $N \in \N$, $0 \le t \le 1$ and $x,y \in \T^2$, we have

\noi
\begin{align*}
\Gamma_{0,N}(t,x-y) \approx - \frac{1}{2\pi} \log_- \Big( \frac{  |x - y| + N^{-1}   }{ |x-y| + t^{\frac12}}  \Big).
\end{align*}

\smallskip

\noi
\textup{(ii)} For any $N \le (2 \eps)^{-1}$, $0 \le t \le 1$ and $x,y \in \T^2$, we have

\noi
\begin{align*}
\hspace{1mm} \Gamma_{\eps,N}(t,x-y) \approx -\frac{1}{2\pi} \log_- \Big( \frac{  |x - y| + N^{-1}   }{ |x-y| + t^{\frac12}}  \Big).
\end{align*}

\smallskip

\noi
\textup{(iii)} For any $N > (2 \eps)^{-1}$, $0 \le t \le 1$ and $x,y \in \T^2$, we have

\noi
\begin{align*}
\hphantom{XXXX}  \Gamma_{\eps,N}(t,x-y) &  \approx -\frac{1}{2\pi} \log_- \Big( \frac{  |x - y| + 2 \eps   }{ |x-y| + t^{\frac12}}  \Big)  \\
& \qquad - \frac{1 - e^{- \frac{t}{\eps^2} } }{2\pi} \log_- \Big( \frac{  |x - y| + N^{-1}  }{ |x-y| + 2 \eps}  \Big).
\end{align*}

\smallskip

\noi
\textup{(iv)} For any $\eps > 0$, $0 \le t \le 1$ and $N \les \eps^{-1 + \ta}$, we have

\noi
\begin{align*}
\hspace{-10mm} \big| \Gamma_{\eps,N}(t,x-y) - \Gamma_{0,N}(t,x-y)  \big| \les \eps^{2\ta},
\end{align*}

\noi
where the bound is uniform in $0 \le t \le 1$ and $x,y \in \T^2$.
\end{lemma}
 
 \begin{remark}\rm
 In \cite{HS} and \cite{ORSW1, ORSW2}, the authors (essentially) consider the following variants of the stochastic convolutions \eqref{sto1} and \eqref{sto2}:
 
 \noi
 \begin{align*}
 \wt \Psi_{1}(t) = \sqrt{2} \int_{-\infty}^t \D_1(t -t') dW(t'),
 \end{align*}
 
 \noi
 and
 
 \noi
 \begin{align*}
 \wt \Psi_{0}(t) = \sqrt{2} \int_{-\infty}^t P_0(t-t') dW(t').
 \end{align*} 
 
 \noi
 These choices of stochastic convolutions lead to covariant estimates which are slightly different from those in Lemma \ref{LEM:cov}. Namely, they prove in \cite[Lemma 3.9]{HS} and \cite[Lemma 2.7]{ORSW1} that
 
 \noi
 \begin{align*}
 \E [ \wt \Psi_{\eps}(t,x)  \wt \Psi_{\eps}(t,y) ] \approx - \frac{1}{2\pi} \log( |x-y| ),
 \end{align*}
 
 \noi
 for $\eps \in \{0,1\}$.
 \end{remark} 
  
 \begin{proof} Throughout this proof, we fix $0 \le t \le 1$ and $x, y \in \T^2$. By \eqref{covfun} and \eqref{convo10}, we have that
 
 \noi
 \begin{align}
 \Gamma_{0,N}(t,x-y) = \frac{1}{2 \pi} \sum_{n \in \Z^2} \chi_N(n)^2 \frac{1 - e^{- 2 t \jb n ^2} }{\jb n ^2} e_n(x- y) = \P_N ^2H(t,x-y),
 \label{cov50}
 \end{align}
 
 \noi
 where $H(t,x)$ is as in \eqref{heatgreen}. Hence, (i) follows immediately from \eqref{bb2} in Lemma \ref{LEM:green} (i).
 
We now turn to the proof of (ii). Let us fix an absolute constant $0 < c \ll 1$.

\smallskip

\noi
{\bf $\bul$ Case 1: $N < c (2 \eps)^{-1}$.} In this case, we keep track in the dependence on the constant $c$ of our estimates as this will be useful for the proof of item (iv) below. By \eqref{covfun}, \eqref{convo10} with \eqref{L3} and \eqref{L4}, we have that 
 
\noi
\begin{align}
\begin{split}
 \Gamma_{\eps,N} (t,x-y) & =  \frac{1}{\pi}\sum_{n \in \Z^2} \chi_N(n)^2 \eps^{-4} \int_0 ^t e^{- \frac{t - t'}{\eps^2} } \frac{ \sinh ( (t-t') \ld_\eps(n))^2 }{ \ld_\eps(n) ^2} dt' e_n (x-y) \label{cov1}
  \\
& =  \frac{1}{2\pi} \sum_{n \in \Z^2}  \chi_N(n)^2 \eps^{-4} \int_0 ^t e^{- \frac{t - t'}{\eps^2} } \frac{ \cosh( 2 (t-t') \ld_\eps(n))}{ \ld_\eps(n) ^2} dt' e_n (x-y)  \\
& \quad -   \frac{1}{2\pi} \sum_{n \in \Z^2} \chi_N(n)^2 \eps^{-2} \frac{ 1 - e^{- \frac{t}{\eps^2} } }{ \ld_\eps(n) ^2} e_n (x-y)  \\
& = \1 - \II.
\end{split}
 \end{align}

\noi
Note that in the above, we used the fact that $\chi_N(n) = 0$ if $\jb n > (2\eps)^{-1}$ (provided the constant $c$ is chosen small enough) so that $\P_N^2 S_\eps(t) =  \P_N^2 \frac{\sinh (\ld_\eps(\nb) t) }{\ld_\eps(\nb)}$ for $N < c (2 \eps)^{-1}$.
 
Since we have that 

\noi
\begin{align}
\ld_\eps (n) \sim \eps^{-2},
\label{condi1}
\end{align}

\noi
for $\jb n \le 10 c (2 \eps)^{-1}  $, we get 
 
 \noi
 \begin{align}
  |\II | \les   N^2 \eps^2 \les c^2.
  \label{cov2}
 \end{align}
 
On the other hand, further expanding $\1$ using the formula \eqref{ld2} for $\ld_\eps(n)$ gives
 
 \begin{align}
 \1 & =  \frac{\eps^{-4}}{4 \pi}  \sum_{n \in \Z^2}  \chi_N(n)^2 \frac{1 - e^{ - ( \eps^{-2} - 2 \ld_\eps(n) ) t }}{ \ld_\eps(n)^2 ( \eps^{-2} - 2 \ld_\eps(n) )}e_n (x-y) \nonumber \\
 &  \qquad + \frac{\eps^{-4}}{4 \pi}  \sum_{n \in \Z^2} \chi_N(n)^2 \frac{1 - e^{ - ( \eps^{-2} + 2 \ld_\eps(n) ) t }}{ \ld_\eps(n)^2 ( \eps^{-2} + 2 \ld_\eps(n) )} e_n (x-y) =: \III + \IV \label{cov3}
 \end{align}
 
 \noi
 As in the estimate on $\II$, we also have 
 
 \noi
\begin{align}
| \IV | \les N^2 \eps^2 \les c^2.
\label{cov4}
\end{align}

We now compare $\III$ to $\Gamma_{0,N}(t,x-y)$. By using the Taylor expansion \eqref{taylor1} and the condition $\jb n \ll \eps^{-1}$, we get

\noi
\begin{align}
\eps^{-2} - 2 \ld_\eps(n) = 2 \jb{n}^2 \big(1 + O ( \jb{n}^2 \eps^2 ) \big),
\label{condi4}
\end{align}

\noi
and in particular, it holds that

\noi
\begin{align}
\eps^{-2} - 2 \ld_\eps(n) \sim \jb n^2 
\label{condi2}
\end{align}

\noi
Thus, by the mean value theorem, we have

\noi
\begin{align}
\big|e^{-(\eps^{-2} - 2 \ld_\eps(n))t} - e^{-2 \jb n ^2 t} \big| \les t \jb n ^4 \eps^2 e^{-c' \jb n ^2 t} \les \jb n ^2 \eps^2,
\label{condi3}
\end{align}

\noi
for some constant $c' >0$. In view of the above and \eqref{cov50}, we have that

\noi
\begin{align}
\begin{split}
&  | \III - \Gamma_{0,N}(t,x-y) | \\
 &  = \Big| \sum_{n \in \Z^2}  \chi_N(n)^2 e_n(x-y) \Big( \frac{\eps^{-4}}{4 \pi} \frac{1 - e^{ - ( \eps^{-2} - 2 \ld_\eps(n) ) t }}{ \ld_\eps(n)^2 ( \eps^{-2} - 2 \ld_\eps(n) )} - \frac{1 - e^{- 2 t \jb n ^2} }{2\pi \jb n ^2} \Big) \Big| \\
 &  \le \frac{\eps^{-4}}{4 \pi}  \sum_{n \in \Z^2}  \chi_N(n)^2 \frac{\big|e^{-2t \jb n ^2} - e^{ - ( \eps^{-2} - 2 \ld_\eps(n) ) t }\big|}{ \ld_\eps(n)^2 ( \eps^{-2} - 2 \ld_\eps(n) )} \\
 & \qquad \quad + \frac{\eps^{-4}}{4 \pi}  \sum_{n \in \Z^2}  \chi_N(n)^2 \frac{1 - e^{-2t \jb n ^2}}{ \ld_\eps(n)^2} \Big| \frac{1}{( \eps^{-2} - 2 \ld_\eps(n) )} - \frac{1}{2\jb n^2}\Big| \\
 & \qquad \quad + \frac{\eps^{-4}}{4 \pi}  \sum_{n \in \Z^2}  \chi_N(n)^2 \frac{1- e^{-2t \jb n ^2}}{ 2 \jb n ^2 } \Big| \frac{1}{\ld_\eps(n)^2} - \frac{4}{\eps^{-4}}\Big| =: A + B + C.
\end{split}
\label{cov5}
\end{align}

\noi
By \eqref{condi1}, \eqref{condi4}, \eqref{condi2}, \eqref{condi3}, we have that

\noi
\begin{align}
A + B \les \eps^2 N^2 \les c^2.
\label{XXX1}
\end{align}

\noi
As for the term $C$, we simply note that

\noi
\begin{align*}
\Big| \frac{1}{\ld_\eps(n)^2} - \frac{4}{\eps^{-4}}\Big| = \frac{|\eps^{-4} - (2 \ld_\eps(n))^2|}{\ld_\eps(n)^2 \eps^{-4}} \les  \frac{|\eps^{-2} - 2 \ld_\eps(n)| \cdot \eps^{-2} }{\ld_\eps(n)^2 \eps^{-4}} \ll \jb n^2 \eps^{6},
\end{align*}

\noi
where we used \eqref{condi1} and \eqref{condi2} with $\jb n \ll \eps^{-1}$ in the above. Hence, this easily yields

\noi
\begin{align}
C \les \eps^2 N^2 \les c^2.
\label{XXX2} 
\end{align}

\noi
Thus, by \eqref{cov5}, \eqref{XXX1} and \eqref{XXX2}, we deduce that

\noi
\begin{align}
 | \III - \Gamma_{0,N}(t,x-y) | \les c^2.
 \label{XXX3}
\end{align}

Putting \eqref{cov1}, \eqref{cov2}, \eqref{cov3}, \eqref{cov4}, \eqref{XXX3} and (i) together gives (ii) for ${1 \le N \le c (2\eps)^{-1}}$.

\smallskip

\noi
{\bf $\bul$ Case 2: $c (2 \eps)^{-1} \le N \le (2 \eps)^{-1}$.} We show that the contribution of the frequencies ${c (2\eps)^{-1}  < \jb n  \le c^{-1} (2\eps)^{-1}}$ to $\Gamma_{\eps,N}(t,x-y)$\footnote{Here, we make use of the fact that $\chi_N(n) = 0$ if $\jb n > c^{-1} (2 \eps)^{-1}$ for $c$ small enough.} is bounded by $O(1)$; which concludes the proof of (ii) since the contribution of the modes $ \jb n  < c (2\eps)^{-1}$ can be estimated as in the proof of Case 1. More precisely, we prove that

\noi
\begin{align}
& \Big| \frac{\eps^{-4}}{2 \pi} \sum_{\substack{ n \in \Z^2 \\ c (2\eps)^{-1}  \le \jb{n} \leq (2 \eps)^{-1}}} \chi_N(n)^2 \int_0 ^t e^{- \frac{t - t'}{\eps^2} } \frac{ \sinh ( (t-t') \ld_\eps(n))^2 }{ \ld_\eps(n) ^2} dt' e_n (x-y) \Big| \les 1,
\label{cov9} \\
& \Big| \frac{\eps^{-4}}{2 \pi} \sum_{\substack{ n \in \Z^2 \\ (2\eps)^{-1}  < \jb{n} \le c^{-1} (2 \eps)^{-1}}} \chi_N(n)^2\int_0 ^t e^{- \frac{t - t'}{\eps^2} } \frac{ \sin ( (t-t') \ld_\eps(n))^2 }{ \ld_\eps(n) ^2} dt' e_n (x-y) \Big| \les 1.
\label{XXX4}
\end{align}

\noi
In what follows, we only show \eqref{cov9} as \eqref{XXX4} follows from similar arguments. We divide the proof of the bound \eqref{cov9} in two (sub-)cases as in the proof as in the proof of Lemma \ref{LEM:F1}.

\smallskip

\noi
$\bullet$
{\bf Subcase 1:} $\ld_\eps(n)(t-t') \le 1$.
\quad
In this case, using the inequalities $|\sinh(x) | \les |x|$ for $|x| \les $ 1 and $e^{-y} \les y^{-2}$ for $y > 0$, we get

\noi
\begin{align*}
& \Big| \frac{\eps^{-4}}{2 \pi} \sum_{\substack{ n \in \Z^2 \\ c (2\eps)^{-1}  \le \jb{n} \leq (2 \eps)^{-1}}} \chi_N(n)^2 \int_0 ^t e^{- \frac{t - t'}{\eps^2} } \frac{ \sinh ( (t-t') \ld_\eps(n))^2 }{ \ld_\eps(n) ^2} dt' e_n (x-y) \Big| \\
& \qquad \quad \les \sum_{\substack{ n \in \Z^2 \\ c (2\eps)^{-1}  \le \jb{n} \leq (2 \eps)^{-1}}}  \chi_N(n)^2\eps^{-4} \int_0 ^t e^{- \frac{t - t'}{\eps^2} } (t-t')^2 dt'  \\
& \qquad \quad \les \sum_{\substack{ n \in \Z^2 \\ c (2\eps)^{-1}  \le \jb{n} \leq (2 \eps)^{-1}}}  \chi_N(n)^2 \int_0 ^t e^{- \frac{t - t'}{2 \eps^2} } dt'  \les \sum_{\substack{ n \in \Z^2 \\ c (2\eps)^{-1}  \le \jb{n} \leq (2 \eps)^{-1}}} \chi_N(n)^2 \eps^2 \les N^2 \eps^2 \les 1.
\end{align*}

\noi
$\bullet$
{\bf Subcase 2:} $\ld_\eps(n)(t-t') > 1$.
\quad
Fix $0 < \g \ll c$. We have the following inequality

\noi
\begin{align*}
\sqrt{1 - x} \le 1 - \g ( 1 + \frac{x}{2} )
\end{align*}

\noi
for $c \le x \le 1$. This entails the following bound

\noi
\begin{align}
e^{-(1 - \g) \frac{t}{2 \eps^2}} \sinh( t \ld_\eps(n) ) \les e^{- \g t \jb n ^2}
\label{cov89}
\end{align}

\noi
Hence, as in the previous case and using the inequalities $\ld_\eps(n)(t-t') > 1$, \eqref{cov89} and $e^{-y} \les y^{-2}$, for $y >0$, we have

\noi
\begin{align*}
& \Big| \frac{\eps^{-4}}{2 \pi} \sum_{\substack{ n \in \Z^2 \\ c (2\eps)^{-1}  \le \jb{n} \leq (2 \eps)^{-1}}} \chi_N(n)^2 \int_0 ^t e^{- \frac{t - t'}{\eps^2} } \frac{ \sinh ( (t-t') \ld_\eps(n))^2 }{ \ld_\eps(n) ^2} dt' e_n (x-y) \Big| \\
& \qquad \quad \les \sum_{\substack{ n \in \Z^2 \\ c (2\eps)^{-1}  \le \jb{n} \leq (2 \eps)^{-1}}} \chi_N(n)^2  \eps^{-4} \int_0 ^t (t-t')^2 e^{- \frac{(t - t') \g}{\eps^2} } \big(  e^{- \frac{(t - t')(1- \g)}{\eps^2} }  \sinh ( (t-t') \ld_\eps(n) ) \big)^2 dt'  \\ 
& \qquad \quad \les \sum_{\substack{ n \in \Z^2 \\ c (2\eps)^{-1}  \le \jb{n} \leq (2 \eps)^{-1}}} \chi_N(n)^2  \eps^{-4} \int_0 ^t (t-t')^2 e^{- \frac{(t - t') \g}{\eps^2} } e^{-\g (t-t') \jb n^2} dt' \\ 
& \qquad \quad \les \sum_{\substack{ n \in \Z^2 \\ c (2\eps)^{-1}  \le \jb{n} \leq (2 \eps)^{-1}}} \chi_N(n)^2  \frac{\eps^{-4}}{\jb n^4} \int_0 ^t  e^{- \frac{(t - t') \g}{\eps^2} } dt' \\ 
& \qquad \quad \les \sum_{\substack{ n \in \Z^2 \\ c (2\eps)^{-1}  \le \jb{n} \leq (2 \eps)^{-1}}}  \chi_N(n)^2 \int_0 ^t  e^{- \frac{(t - t') \g}{2 \eps^2} }  dt' \les N^2\eps^2 \les 1.
\end{align*}

\noi
This shows \eqref{cov9} and concludes the proof of (ii).

We now prove (iii). Let us fix $0 < c \ll 1$. We have that

\noi
\begin{align}
\begin{split}
 \Gamma_{\eps,N} (t,x-y) & =  \frac{1}{\pi}\sum_{n \in \Z^2} \chi_N(n)^2 \eps^{-4} \int_0 ^t e^{- \frac{t - t'}{\eps^2} } \ft S_\eps (t-t',n)^2 dt' e_n (x-y) \label{XXX10}
  \\
& =  \frac{\eps^{-4}}{\pi} \sum_{\substack{n \in \Z^2 \\ \jb n < c^{-1} (2\eps)^{-1}}}  \chi_N(n)^2 \eps^{-4} \int_0 ^t e^{- \frac{t - t'}{\eps^2} } \ft S_\eps(t-t',n)^2 dt' e_n (x-y)  \\
& \quad  +  \frac{\eps^{-4}}{\pi} \sum_{\substack{n \in \Z^2 \\  \jb n \ge c^{-1}(2\eps)^{-1}}}  \chi_N(n)^2 \int_0 ^t e^{- \frac{t - t'}{\eps^2} } \frac{\sin(\ze_\eps(n) (t-t'))^2}{\ze_\eps(n)^2} dt' e_n (x-y)  \\
& =: \1 + \II.
\end{split}
\end{align}

\noi
Here, $S_\eps(t)$ and $\ze_\eps(n)$ are as in \eqref{L3} and \eqref{ld3}, respectively and $\ft S_\eps(t,n)$ denotes the Fourier symbol of $S_\eps(t)$. 

By arguing as in the proof of (ii) above, we have
\noi
\begin{align}
\1 \approx -\frac{1}{2\pi} \log_- \Big( \frac{  |x - y| + 2\eps   }{ |x-y| + t^{\frac12}}  \Big).
\label{XXX11}
\end{align}

\noi
as $\1$ essentially corresponds to $\Gamma_{\eps, c^{-1}(2\eps)}(t,x)$. More precisely, the contribution of the modes $\jb n < c (2\eps)^{-1}$ and $ c (2\eps)^{-1} \le \jb n < c^{-1} (2\eps)^{-1}$ to $\1$ may be estimated as in Case 1 and Case 2 in the proof of (ii), respectively.

We now bound $\II$. By using trigonometric formulas and integrating in time, we get 

\noi
\begin{align}
\begin{split}
\II & = \frac{\eps^{-4}}{\pi} \sum_{\substack{n \in \Z^2 \\  \jb n \ge c^{-1}(2\eps)^{-1}}}  \chi_N(n)^2 \int_0 ^t e^{- \frac{t - t'}{\eps^2} } \frac{\sin(\ze_\eps(n) (t-t'))^2}{\ze_\eps(n)^2} dt' e_n (x-y) \\
& = \frac{\eps^{-4}}{2\pi} \sum_{\substack{n \in \Z^2 \\  \jb n \ge c^{-1}(2\eps)^{-1}}}  \frac{\chi_N(n)^2}{\ze_\eps(n)^2}  e_n (x-y) \Big( \eps^2 (1-e^{-\frac{t}{\eps^2}}) \\
& \qquad \qquad \qquad \qquad \qquad \qquad - \int_0 ^t e^{- \frac{t - t'}{\eps^2} } \cos(2\ze_\eps(n) (t-t')) dt' \Big) \\
& = \frac{\eps^{-2}}{2\pi} (1-e^{-\frac{t}{\eps^2}})  \sum_{\substack{n \in \Z^2 \\  \jb n \ge c^{-1}(2\eps)^{-1}}}  \frac{\chi_N(n)^2}{\ze_\eps(n)^2}  e_n (x-y)  \\
& \qquad \qquad \qquad \qquad  - \frac{\eps^{-4}}{2\pi} \Im \sum_{\substack{n \in \Z^2 \\  \jb n \ge c^{-1}(2\eps)^{-1}}}  \frac{\chi_N(n)^2(1- e^{- \frac{t}{\eps^2} + i \ze_\eps(n) t)} )}{\ze_\eps(n)^2(\eps^{-2} - i \ze_\eps(n))} e_n (x-y) \\
& =: \III - \IV.
\label{XXX12}
\end{split}
\end{align}

\noi
Note that by \eqref{ld3}, it holds that 

\noi
\begin{align*}
\ze_\eps(n) \sim \frac{\jb n}{\eps} \gg \eps^{-2},
\end{align*}

\noi
when $\jb n \ge c^{-1}(2\eps)^{-1}$. Hence, we have

\noi
\begin{align}
|\IV| \les \eps^{-1} \sum_{\substack{n \in \Z^2 \\ c^{-1}(2\eps)^{-1} \le \jb n \les N }} \jb n ^{-3} \les c.
\label{XXX13}
\end{align}

On the other hand, by arguing as in \eqref{cov5}-\eqref{XXX3} in Case 1 in the above proof of (ii), one can show that

\noi
\begin{align}
| \III - (1- e^{-\frac{t}{\eps^2}}) (\P_N ^2 - \P_{(2\eps)^{-1}}^2) G(t,x-y) | \les c,
\label{XXX14}
\end{align}

\noi
where $G$ is as in \eqref{besselG}. Thus, (iii) follows from the estimate \eqref{bb1} in Lemma \ref{LEM:green} together \eqref{XXX10}, \eqref{XXX11}, \eqref{XXX12}, \eqref{XXX14} and \eqref{XXX14}, 

Lastly, we briefly discuss the proof of item (iv). For $\eps \sim 1$, then (iv) is immediate in view of (i) and (ii) since $N \les 1$ in this case. Otherwise, for $0 < \eps \ll 1$, the estimate (iv) is a consequence of the bounds \eqref{cov2}, \eqref{cov4} and \eqref{XXX3} with $c = \eps^{\ta} \ll 1$ for $0 < \ta \ll 1$.
 \end{proof}

\begin{lemma}\label{LEM:cov2}
 Let $1 \le N_1 \le N_2$. We have the following bound:
 
 \noi
 \begin{align}
 | \P_{N_j}^2 \Gamma_{\eps,N} (t,x) - \P_{N_1} \P_{N_2} \Gamma_{\eps,N}(t,x)  | \les  N_1^{-\dl} |x|^{- 2 \dl},
 \label{fun1a}
 \end{align}
 
 \noi
uniformly in $N \ge 1$ and for any $\eps \in [0,1]$, $0 \le t \le 1$, $x \in \T^2 \setminus \{ 0 \}$, $j = 1,2$ and $\dl > 0$ small enough.
 \end{lemma}

 \begin{proof}
Fix $j \in \{1,2\}$, $0 \le t \le 1$ and $x \in \T^2 \setminus \{0\}$. The bound \eqref{fun1a} is a consequence of \eqref{cov50} and the bound \eqref{bb4} on $H$ in Lemma \ref{LEM:green}. Indeed, by interpolation and the inequality $| \log_-( y ) | \les_\dl y^{-\dl}$ for any $0 < y \les 1 $ and $\dl >0$, we have that

\noi
\begin{align}
\begin{split}
 | (\P_{N_j}^2 H(t,x) - \P_{N_1} \P_{N_2})(\P_N^2H)(t,x)  | & \les \Big( 1- \log_- \Big( \frac{  |x| + N_2^{-1}   }{ |x| + t^{\frac12}}  \Big) \Big)^{1- \dl} (N_1 ^{-1} |x|^{-1} )^\dl \\
 & \les  \Big( 1 +\Big( \frac{  |x| + t^{\frac12}  }{ |x| + N_2^{-1}  }  \Big)^{-\dl} \Big)^{1- \dl} (N_1 ^{-1} |x|^{-1} )^\dl \\
 & \les N_1^{-\dl} |x|^{-2 \dl}.
 \end{split}
 \label{XXX15}
\end{align} 

Fix $0 < \eps \le 1$. It suffices to prove the bound \eqref{fun1a} in the following four cases:

\begin{itemize}
\item[{\bf (a)}] $0< \eps \le 10^{-\frac{10}{\ta}}$ and $1 \le N_1 , N_2 \le 10^3 \eps^{-1 + \ta}$,

\smallskip

\item[{\bf (b)}] $0< \eps \le 10^{-\frac{10}{\ta}}$ and $N_1, N_2 \ge 10^{-3} \eps^{-1 - \ta}$,

\smallskip

\item[{\bf (c)}] $0< \eps \le 10^{-\frac{10}{\ta}}$ and $10^{-3}  \eps^{-1 + \ta} \le N_1 , N_2 \le 10^3 \eps^{-1 - \ta}$.

\smallskip

\item[{\bf (d)}] $10^{-\frac{10}{\ta}}< \eps \le 1$.
\end{itemize}

\noi
Let us now argue that proving \eqref{fun1a} in the above four cases allows us to obtain \eqref{fun1a} in full generality. For instance, let us assume that $0<\eps \le 10^{-\frac{10}{\ta}}$ and that we have $1 \le N_1 \le 10^{-3} \eps^{-1 + \ta} < 10^{3} \eps^{-1+\ta} < N_2 \le 10^3 \eps^{-1-\ta}$ so that the couple $(N_1,N_2)$ does not belong to the cases (a), (b), (c) or (d). Then, pick $M = \eps^{-1+\ta}$ such that $(N_1,M)$ belongs to Case (a) while $(M,N_2)$ belongs to Case (c). Note that by assumption, the Fourier support of the symbols associated with $\P_{N_1}$, $\P_M$ and $\P_{N_2}$ are included in each other (in that order). This leads to the following identity:

\noi
\begin{align}
\begin{split}
\P_{N_2} - \P_{N_1} \P_{N_2} & = \P_{N_2} - \P_{N_1} = (\P_{N_2} - \P_M) + (\P_M - \P_{N_1}) \\
& = (\P_{N_2} - \P_M \P_{N_2} ) + (\P_M - \P_{N_1} \P_M ).
\end{split}
\label{fun88}
\end{align}

\noi
Hence, from \eqref{fun88} and Case (a) and Case (c), we get (say for $j=2$ in \eqref{fun1a})

\noi
\begin{align*}
 |\P_{N_j}^2 \Gamma_{\eps,N} (t,x) - \P_{N_1} \P_{N_2} \Gamma_{\eps,N}(t,x)  | \les (M^{-\dl} + N_1 ^{-\dl}) |x|^{-\dl} \les N_1^{-\dl} |x|^{-\dl},
\end{align*}

\noi
which proves \eqref{fun1a} in this case. The other scenarios follow in a similar fashion.

\smallskip

We now prove \eqref{fun1a} in the cases (a), (b), (c) and (d).

\smallskip

\noi
{\bf $\bul$ Proof of \eqref{fun1a} in Case (a).} The bound \eqref{fun1a} in this case is a consequence of the \eqref{cov50}, \eqref{XXX15} and Lemma \ref{LEM:cov} (iv) by noting that a small power of $\eps$ can be converted to a small power of $N_1$ in Case (a).   

\smallskip

\noi
{\bf $\bul$ Proof of \eqref{fun1a} in Case (b).} By arguing as in the computations leading to the bounds \eqref{XXX12}-\eqref{XXX14} in the proof of Lemma \ref{LEM:cov} (iii) with $c = N_1^{-\ta^2}$, we have that\footnote{Note that when applying the frequency projection $\P_{N_j}^2 - \P_{N_1} \P_{N_2}$ to $\G_{\eps,N}$, the term $\1$ in \eqref{XXX10} disappears with our current choice for the parameter $c$ since $c^{-1} (2 \eps)^{-1} \ll N_1 $.} 

\noi
\begin{align*}
\big| ( \P_{N_j}^2 - \P_{N_1} \P_{N_2}) \big( \Gamma_{\eps,N}(t,x) - (1 - e^{- \frac{t}{\eps^2}}) G(x) \big)  \big| \les N_1^{- \ta^2}.
\end{align*}

\noi
Thus, \eqref{fun1a} in this case follows from the above and the estimate 

\noi
\begin{align*}
 | (\P_{N_j}^2 - \P_{N_1} \P_{N_2})(\P_N^2G)(x)   |  \les N_1^{-\dl} |x|^{-2 \dl},
\end{align*} 

\noi
for $\dl >0$ small enough, which in turn follows from \eqref{bb3} in Lemma \ref{LEM:green} and  an interpolation argument as in \eqref{XXX15}.
\smallskip

\noi
{\bf $\bul$ Proof of \eqref{fun1a} in Case (c).} By \eqref{covfun2} with \eqref{sto1}, we have
 
 \noi
 \begin{align}
 \begin{split}
  \P_{N_j}^2 \Gamma_{\eps,N} (t,x) - \P_{N_1} \P_{N_2} \Gamma_{\eps,N}(t,x) & = \frac{1}{\pi} \sum_{n \in Z^2} \eps^{-4} \chi_N^2(n)( \chi_{N_j}^2(n) - \chi_{N_1}(n) \chi_{N_2}(n) ) \\
  & \qquad \quad \times  \int_0 ^t  \ft{\D_\eps}( n, t-t') ^2  dt' e_n (x),
 \end{split}
 \label{fun10a}
 \end{align}

\noi
where $\ft{\D_\eps}$ is as in \eqref{De1}. From \eqref{fun10a} and the Poisson summation formula \eqref{poisson}, we get

\noi
\begin{align}
  \P_{N_j}^2 \Gamma_\eps (t,x) - \P_{N_1} \P_{N_2} \Gamma_\eps(t,x) = \frac{1}{\pi} \sum_{\ell \in \Z^2} \rho^\eps_{N_1,N_2}( x + 2 \pi \ell),
 \label{fun15}
\end{align}

\noi
where $\rho^\eps_{N_1,N_2}: \R^2 \to \R^2$ is the smooth function defined on the Fourier side by

\noi
\begin{align*}
\ft{\rho^\eps_{N_1,N_2}}(t,\eta) =  \eps^{-4} \chi_N^2(\eta) ( \chi_{N_j}^2(\eta) - \chi_{N_1}(\eta) \chi_{N_2}(\eta) ) \int_0 ^t \ft{\D_\eps}(\eta,t-t') ^2 dt'
\end{align*}

\noi
Moreover, by noting that $\eta \mapsto \chi_N^2(\eta) \chi_{N_j}(\eta)^2 - \chi_{N_1}(\eta) \chi_{N_2}(\eta)$ is supported on ${\{ N_1 \les |\eta | \les N_2 \}}$, integration by parts, Lemma \ref{LEM:F3} and the Leibniz formula, we have that

\noi
\begin{align}
\begin{split}
|\rho^\eps_{N_1,N_2}(t,x)| & = \frac{1}{2\pi} \Big| \int_{\R^2} \ft{\rho^\eps_{N_1,N_2}}(t,\eta) e^{i \eta \cdot x} d\eta \Big|  \\
& \sim |x|^{- 2m} \Big| \int_{\R^2} \Dl^{m}_{\eta} \big( \ft{\rho^\eps_{N_1,N_2}}(t,\eta) \big) e^{i \eta \cdot x} d\eta \Big| \\
& \les  \eps^{-4} |x|^{- 2m} \sum_{ \substack{ \al_1, \al_2, \al_3 \\ |\al_1| + |\al_2| + |\al_3| = 2m  }} \int_{\R^2} d\eta  \int_0 ^t dt' | \partial^{\al_1}_\eta \ft{D_\eps}(\eta,t-t') | \,  \\
& \qquad \quad \times | \partial^{\al_2}_\eta  \ft{D_\eps}(\eta,t-t')| \big| \partial^{\al_3}_\eta \big( \chi_N^2(\eta)  ( \chi_{N_j}^2(\eta) - \chi_{N_1}(\eta) \chi_{N_2}(\eta) ) \big) \big| \\
& \les \eps^{-4} |x|^{- 2m} \sum_{ \substack{ \al_1, \al_2, \al_3 \\ |\al_1| + |\al_2| + |\al_3| = 2m  }} \int_{ N_1 \les |\eta | \les N_2 } N_1^{-  |\al_3|} \\
& \qquad \quad \times \int_0 ^t e^{- \frac{t'}{\eps^2} } (t')^{|\al_1| + |\al_2| + 2} \eps^{- |\al_1| - |\al_2|} \sum_{1 \le p\le |\al_1| + |\al_2|} (1 + | t' \eps^{-1} \eta |^{p}) \\
& \qquad \quad \times \big( \ind_{\jb \eta \le (2\eps)^{-1}}  \, e^{ 2 t' \ld_\eps(\eta)}  + \ind_{\jb \eta > (2\eps)^{-1}} \big) dt' d\eta.
\end{split}
\label{fun16}
\end{align}

\noi
for any $m \in \N$.

Let us denote by $\1 = \1 (\al_1, \al_2, \al_3,p)$ and $\II = \II (\al_1, \al_2, \al_3,p)$ the contribution of $\jb \eta \le (2\eps)^{-1}  $ and $\jb \eta > (2\eps)^{-1}$ to the (right-most) right-hand-side of \eqref{fun16}, respectively, for a fixed tuple $(\al_1,\al_2,\al_3,p)$. From \eqref{upheat} and the inequality $e^{-y} \les y^{-a}$ for any $y > 0$ and $a >0$, we have that

\noi
\begin{align}
\begin{split}
\1 &  \les  \eps^{-4 - |\al_1| - |\al_2|} |x|^{- 2m}  N_1^{-  |\al_3|}  \int_{ N_1 \les |\eta | \les N_2 } \int_0 ^t e^{-  2 t' \jb \eta ^2} \\
&  \hspace{55mm} \times (t')^{|\al_1| + |\al_2| + 2}(1 + | t' \eps^{-1} \eta |^{p}) dt' d\eta \\
& \les \eps^{-4 - |\al_1| - |\al_2|} |x|^{- 2m}  N_1^{-  |\al_3|} \int_{ N_1 \les |\eta | \les N_2 } \jb{\eta}^{-2 ( |\al_1| + |\al_2| )- 4 } \\
& \hspace{60mm} \times  (1 + ( \eps \jb{\eta})^{-p})  \int_0 ^t e^{- \frac12 t' \jb \eta ^2} dt' d\eta \\
& \les \eps^{-4 - |\al_1| - |\al_2|} |x|^{- 2m}  N_1^{-  |\al_3|} \int_{ N_1 \les |\eta | \les N_2 } \jb{\eta}^{-2 ( |\al_1| + |\al_2| )-6} (1 + ( \eps \jb{\eta})^{-p}) \\
& \les \eps^{-4 - |\al_1| - |\al_2|} |x|^{- 2m}  N_1^{-  |\al_3| - 2 ( |\al_1| + |\al_2|) -  4 + c_1 \ta } \les |x|^{- 2m}  N_1^{- 2m + c_2 \ta},
\end{split}
\label{fun17}
\end{align}

\noi
in view of the the conditions $|\al_1| + |\al_2| + |\al_3| = 2m$ and the condition ${\eps^{-1 + \ta} \les N_1 \le N_2 \les \eps^{-1 - \ta}}@$. Here $c_j = c_j(p)$, $j=1,2$ are fixed constants. We bound similarly

\noi
\begin{align}
\II \les |x|^{- 2m}  N_1^{- 2m +c \ta} \label{fun18},
\end{align}

\noi
for some constant $c = c(p)$. Performing the summation over $(\al_1,\al_2,\al_3,p)$ with \eqref{fun16}, \eqref{fun17} and \eqref{fun18} yields

\noi
\begin{align}
|\rho^\eps_{N_1,N_2}(x)|  \les |x|^{- 2m}  N_1^{- 2m + c \ta},
\label{fun19}
\end{align}

\noi
for any $m \in \N$ and $x \in \R^2 \setminus \{0\}$. Hence, summing in $\ell \in \Z^2$ with $m=2$ in \eqref{fun15} with \eqref{fun19} gives

\noi
\begin{align}
\big|   \P_{N_j}^2 \Gamma_{\eps,N} (t,x) - \P_{N_1} \P_{N_2} \Gamma_{\eps,N}(t,x) \big| \les |x|^{-4} N_1^{- 4 + c \ta},
\label{fun20}
\end{align}

\noi 
for any $x \in \T^2 \setminus \{0\}$ and uniformly in $N \ge 1$.

Thus, interpolating \eqref{fun20} with Lemma \ref{LEM:cov} (ii) or (iii) (depending on the relative sizes of $N$ and $(2\eps)^{-1}$) as in \eqref{XXX15} shows \eqref{fun1a} in this regime and concludes the proof of Case (c).

\smallskip

\noi
{\bf $\bul$ Proof of \eqref{fun1a} in Case (d).} The proof of this case follows from arguments that are similar to the proof of Case (c) above by noting that $\eps \sim_\ta 1$ in the current scenario.
 \end{proof} 
 
 \subsection{Construction and covariance of stochastic objects}\label{SUBSEC:sin2} In the current section, we construct the stochastic objects $\{\U_{\eps,N}\}_{N \in \N}$ for all $\eps \in [0,1]$ and study the convergence of the constructed limits as $\eps \to 0$; see Proposition \ref{PROP:stosin2} below. 
 
 Fix $\eps \in [0,1]$ and $N \in \N$. We define the following functions - which call potentials - for $\eps > 0$ and $N \in \N$,
 
 \noi
 \begin{align}
 \mathcal J_{\eps,N} (t,x) = e^{-2\pi \Gamma_{\eps,N}(t,x)},
 \label{pot1}
 \end{align}
 
 \noi
 for any $0 \le t \le 1$ and $x \in \T^2$. Here $\G_{\eps,N}$ is as in \eqref{covfun}.
 
 In view of the definition of the function $\log_{-}$ in \eqref{log_minus} and Lemma \ref{LEM:cov}, the following identities hold:
 
 \noi
 \begin{align}
  \mathcal J_{\eps,N} (t,x) \sim \begin{cases} \frac{|x| + N^{-1}}{|x| + t^{\frac12}} \quad & \text{if } t^{\frac12} > N^{-1}\\
  1 \quad & \text{if } t^{\frac12} \le N^{-1}  \end{cases},
   \label{pot2}
 \end{align}
 
 \noi
 for $N \le (2\eps)^{-1}$ and 
 
  \noi
 \begin{align}
  \mathcal J_{\eps,N} (t,x) \sim \begin{cases}\Big( \frac{ |x| +2 \eps }{ |x| + t^{\frac12} } \Big) \Big( \frac{ |x| + N^{-1} }{ |x| +  2\eps } \Big)^{1- e^{-\frac{t}{\eps^2}}} \quad & \text{if } t^{\frac12} > 2\eps\\
 \Big( \frac{ |x| + N^{-1} }{ |x| + 2\eps } \Big)^{1- e^{-\frac{t}{\eps^2}}}  \quad & \text{if } t^{\frac12} \le 2\eps  \end{cases},
  \label{pot3}
 \end{align}
 
 \noi
 for $N > (2\eps)^{-1}$.

For $\s_1, \s_2 \in \{1,-1\}$ and $\eps \in [0,1]$, we also define $\mathcal{J}^{\s_1 \s_2}(t,x) := (\mathcal{J}(t,x))^{\s_1 \s_2}$ for $t \ge 0$ and $x \in \T^2$.

The next lemma highlights a ``cancellation of charge" property of the potentials defined above as first noted in \cite{HS} in the context of the parabolic sine-Gordon model ($\eps = 0$) and in \cite{ORSW1} for the hyperbolic sine-Gordon equation ($\eps =1$).

\noi
\begin{lemma}\label{LEM:charge}
Let $\ld > 0$ and $p \in \N$. Given $j \in \{ 1, \cdots, 2p\}$, we set $\s_j = 1$ if $j$ is even and $\s_j = -1$ is odd. Let $\mathfrak S_p$ be the set of permutations of the set $\{1, \cdots, p \}$. Fix $0<t\le 1$, $\eps \in [0,1]$ and $N \in \N$. Then, the following estimate holds:

\noi
\begin{align}
\prod_{1 \le j < k \le 2p } \mathcal{J}_{\eps,N}^{\s_1 \s_2}(t,x_j - x_k) ^\ld  \les \max_{\tau \in \mathfrak S_p} \prod_{1 \le j \le p} \big( |x_{2j} - x_{2 \tau(j) -1 }| + N^{-1} \big)  ^{-\ld},
\label{Ebd}
\end{align}

\noi
for any set of $2p$ points $\{ x_j \}_{1 \le j \le 2p}$ in $\T^2$.
\end{lemma}

\begin{proof}
By proceeding by induction as in \cite[Lemma 2.5]{ORSW1} and by using \eqref{pot2} and \eqref{pot3}, we may get the bound

\noi
\begin{align}
\prod_{1 \le j < k \le 2p } \mathcal{J}_{\eps,N}^{\s_1 \s_2}(t,x_j - x_k) ^\ld  \les \max_{\tau \in \mathfrak S_p} \prod_{1 \le j \le p} \mathcal{J}_{\eps,N}(t,x_j - x_k)  ^{-\ld},
\label{Ebd1}
\end{align}

\noi
for any $0 \le t \le 1$ and any set of $2p$ points $\{ x_j \}_{1 \le j \le 2p}$ in $\T^2$. The estimate \eqref{Ebd} then follows from \eqref{Ebd1} and the bound

\noi
\begin{align*}
\mathcal{J}_{\eps,N}(t,x) ^{-\ld}  \les \big( |x| + N^{-1} \big)  ^{-\ld},
\end{align*}

\noi
for $x \in \T^2$, which follows easily from \eqref{pot2} and \eqref{pot3}.
  \end{proof}

\begin{proposition}\label{PROP:stosin2}
Let $0 < \be^2 < 4 \pi$. Fix any finite $p,q \geq 1$, $T >0$, $\al > \frac{\be^2}{4 \pi}$. Then, the following holds:

\smallskip

\noi
\textup{(i)} Let $\eps \in [0,1]$. The sequence $   \{ \U_{\eps,N}  \}_{N \in \N}$ defined in \eqref{Ups} is a Cauchy sequence in $L^p \big( \O;  L^q ( [0,T]; W^{-\al, \infty}(\T^2 ) \big)$ and thus converges to a limiting stochastic process in $L^p \big( \O;  L^q ( [0,T]; W^{-\al, \infty}(\T^2 ) \big) $, denoted by $ \U_\eps $ as $N \to \infty$.

\smallskip

\noi
\textup{(ii)} The sequence $ \{ (\eps,t) \mapsto \U_{\eps,N} (t)  \}_{N \in \N}$ also converges to $(\eps, t) \mapsto  \U_\eps (t)$ ${L^p \big( \Omega; L^q\big( [0,1] \times [0,T]; W^{-\s, \infty}(\T^2)\big)  \big)}$ and almost surely in $C([0,1] \times [0,T]; W^{-\al, \infty}(\T^2 ))$ as $N \to \infty$.
\end{proposition}

\begin{proof} Fix $T > 0$, $\al > \frac{\be^2}{4 \pi}$, and $0 < \dl \ll 1$. By using Lemma \ref{LEM:charge}, Lemma \ref{LEM:cov2} and by a straightforward adaptation of the argument in the proof of \cite[Proposition 1.1]{ORSW1}, we get

\noi
\begin{align}
& \E \big[ \big| \jb{ \nb }^{\dl - \al} \U_{\eps,N} (t,x)  \big|^{2p} \big] \les_T 1 \label{stosin1} \\
& \E \big[ \big| \jb{ \nb }^{\dl - \al} \big( \U_{\eps,N_2} (t,x) - \U_{\eps,N_1} (t,x)  \big) \big|^{2} \big] \les_T  N_1 ^{- \kappa}, \label{stosin2}
\end{align}

\noi
uniformly in $(\eps,t) \in [0,1] \times [0,T]$, $x \in \T^2$, $N \in \N$ and for $0 < \kappa \ll 1$, and $N_2 \geq N_1$ two integers. 

By arguing as in the proof of Proposition \ref{PROP:sto}, the bounds \eqref{stosin1} and \eqref{stosin2} show that for any fixed $\eps \in [0,1]$, $   \{ \U_{\eps,N}  \}_{N \in \N}$ and $ \{ (\eps,t) \mapsto \U_{\eps,N} (t)  \}_{N \in \N}$ are Cauchy sequences in $L^p \big( \O;  L^q ( [0,T]); W^{-\al, \infty}(\T^2 ) \big) $ and in $L^p \big( \Omega; L^q\big( [0,1] \times [0,T]; W^{-\s, \infty}(\T^2)\big)  \big)$ respectively. This proves item (i) and the first part of item (ii).

We now look at the almost sure convergence of the sequence $ \{ (\eps,t) \mapsto  \Theta_{\eps,N}(t) \}_{N \in \N}$ in item (ii). We claim that the following estimate holds:

\noi
\begin{align}\label{stosin3}
 \E \big[ \big| \jb{ \nb }^{\dl - \al} \big( \U_{\eps,N} (t,x) - \U_{0,N} (t,x)  \big) \big|^{2} \big] \les \eps^{\g} 
\end{align}

\noi
for some small $\g >0$.  In order to prove \eqref{stosin3}, it suffices to show

\noi
\begin{align}\label{stosin4}
 \E \big[ \big| \jb{ \nb }^{\dl - \al} \big( \U_{\eps,N} (t,x) - \U_{0,N} (t,x)  \big) \big|^{2} \big] \les N^{C}  \eps^{\kappa}
 \end{align}

\noi
fo some large absolute constant $C>0$ (depending on $\be$, see below), some small $\kappa >0$ and $N \les \eps^{- 1  + \ta}$. Indeed, if $N^C \les \eps^{- \frac \kappa 2}$, then \eqref{stosin4} yields \eqref{stosin3}. Otherwise, in the case $N^C \gg \eps^{- \frac \kappa 2}$, by combining \eqref{stosin2} and \eqref{stosin4} with an integer $M$ such that $M^{C} \sim \eps^{-\frac \kappa 2}$ gives

\noi
\begin{align*}
 \E \big[ \big| \jb{ \nb }^{\dl - \al} \big( \U_{\eps,N} (t,x) - \U_{0,N} (t,x)  \big) \big|^{2} \big] &  \les \E \big[ \big| \jb{ \nb }^{\dl - \al} \big( \U_{\eps,M} (t,x) - \U_{0,M} (t,x)  \big) \big|^{2} \big] \\
& \quad + \sup_{\eps_0 \in [0,1]} \E \big[ \big| \jb{ \nb }^{\dl - \al} \big( \U_{\eps_0,N} (t,x) - \U_{\eps_0, M} (t,x)  \big) \big|^{2} \big] \\
& \les \eps ^{\frac \kappa 2} + M^{- \kappa} \les \eps^{\gamma},
\end{align*}

\noi
for some small $\g >0$. This proves that \eqref{stosin3} reduces to \eqref{stosin4}.

We now prove \eqref{stosin4}. By Lemma \ref{LEM:bessel}, the H\"older and Cauchy-Schwarz inequalities, we have that

\noi
\begin{align} 
\begin{split}
& \E \big[ \big| \jb{ \nb }^{\dl - \al} \big( \U_{\eps,N} (t,x) - \U_{0,N} (t,x)  \big) \big|^{2} \big] \\
& \quad = \int_{\T^2} \int_{\T^2} J_{\al - \dl} (x-y) J_{\al - \dl}(x-z)   \\
& \qquad  \quad \times \E \big[ \big(  \U_{\eps,N} (t,y)  -  \U_{0,N} (t,y)  \big) \big(  \U_{\eps,N}(t,z) -  \U_{0,N}(t,z) \big)  \big] dy dz  \\
& \quad \les  \big\| \U_{\eps,N} (t,y)  -  \U_{0,N} (t,y)  \big\|_{L^{\infty}_y L^2 (\O)}^2,
\end{split}
\label{stosin5}
\end{align}

\noi
assuming $\al - \dl > 0$. In view of \eqref{Ups}, one can write

\noi
\begin{align*}
\U_{\eps,N} (t,y)  -  \U_{0,N} (t,y) & = e^{ \frac{\be^2}{2} \s_{\eps,N}(t)} \big(  e^{i \be \Psi_{\eps,N}(t,y)} -  e^{i \be \Psi_{0,N}(t,y)}  \big) \\
& \qquad  + \big( e^{ \frac{\be^2}{2} \s_{\eps,N}(t)} - e^{ \frac{\be^2}{2} \s_{0,N}(t)} \big) e^{i \be \Psi_{0,N}(t,y)} \\
& \qquad =: \1 + \II.
\end{align*}

\noi
for any $y \in \T^2$. Here, $\s_{\eps,N}$ is as in \eqref{sig}. Note that we have $\s_{\eps,N} = \Gamma_{\eps,N}(t,0)$, where $\Gamma_{\eps,N}$ is the covariance function \eqref{covfun}. By this observation, Lemma \ref{LEM:cov} (ii)-(iv), the condition $N \les \eps^{-1 + \ta}$, and the mean value theorem, we have

\noi
\begin{align}
\|  \II \|_{L^{\infty}_y L^2 (\O)} \les N^{100\be^2} \eps^{\ta}.
\label{stosin52}
\end{align}

\noi
Besides, by the mean value theorem, Minkowski's and Bernstein's inequalities, and by proceeding as in the proof of \eqref{S24} (for $\l = 1$), we have

\noi
\begin{align}
\begin{split}
\| \1 \|_{L^{\infty}_y L^2 (\O)} & \les N^{100 \be^2} \big\| \Psi_{\eps,N} (t,y)  -  \Psi_{0,N} (t,y)  \big\|_{L^2 (\O) L^{\infty}_y } \\
& \les N^{ 100 \be^2 + 10} \big\| \Psi_{\eps,N} (t,y)  -  \Psi_{0,N} (t,y)  \big\|_{ L^2 (\O) W^{-1,\infty}_y} \les N^C \eps^\ta,
\end{split}
\label{stosin53}
\end{align}

\noi
for some $C > 0$. The bounds \eqref{stosin5}, \eqref{stosin52} and \eqref{stosin53} imply \eqref{stosin4}.

We now claim that the following estimates hold:

\noi
\begin{align}
& \E \big[ \big| \jb{ \nb }^{\dl - \al}  ( \dl_{h_1, h_2} \U_{\eps,N} (t,\cdot)) (x)   \big|^{2} \big] \les \| (h_1, h_2) \|_2 ^{\g} \label{stosin7} \\
& \E \big[ \big| \jb{ \nb }^{\dl - \al} (  \dl_{h_1, h_2} \big( \U_{\eps,N_2} (t,\cdot) -  \U_{\eps,N_1} (t,\cdot) \big)) (x)   \big|^{2}  \big] \les  N_1 ^{- \g}  \| (h_1, h_2) \|_2 ^{\g} \label{stosin8}
\end{align}

\noi
uniformly in $(\eps,t) \in [0,1] \times [0,T]$, $x \in \T^2$ and $N \geq 1$ such that $(\eps + h_1, t + h_1) \in [0,1] \times [0,T]$ and $N_2 \geq N_1$. Since, \eqref{stosin8} follows from \eqref{stosin2} and \eqref{stosin7}, we focus on \eqref{stosin7}. By arguing as we did in the above when reducing \eqref{stosin3} and \eqref{stosin4}, it suffices to prove

\noi
\begin{align}
\E \big[ \big| \jb{ \nb }^{\dl - \al}  ( \dl_{h_1, h_2} \U_{\eps,N} (t,\cdot)) (x)   \big|^{2} \big] \les N^{ 100(1 + \be^2) } \| (h_1, h_2) \|_2 ^{\kappa}, 
\label{stosin9}
\end{align}

\noi
for some small $\kappa > 0$ and uniformly in all parameters. Here, $\dl_{h_1,h_2}$ is as in \eqref{diff}. By arguing as in the proof of \eqref{stosin4} and applying the mean value theorem, the estimate \eqref{stosin9} essentially reduces to proving estimates on the stochastic convolution \eqref{convo10} which are similar to \eqref{S6}; we omit details.

Interpolating the bounds \eqref{stosin7} and \eqref{stosin8} with \eqref{stosin1} and applying the bi-parameter Kolmogorov continuity criterion in Lemma \ref{LEM:kol} together with the Borel-Cantelli lemma as in Proposition \ref{PROP:sto} concludes the proof 

\end{proof}

\subsection{Well-posedness and convergence of the dynamics}\label{SUBSEC:sin3}
In this subsection, we sketch the proof of Theorem \ref{THM:sinGWP}.

We first prove uniform (in $\eps$) local well-posedness in $C([0,1] \times [0,T]; H^{1-\al}_x)$ of $(\eps,t) \mapsto u(\eps,t)$ the solution to \eqref{rSdSG}.

After making the ansatz \eqref{H1}, we are led to study \eqref{SdSG11} and \eqref{SdSG3}. It thus suffices to consider the following system with unknown $(\eps,t) \mapsto v(\eps,t)$.

\noi
\begin{align}
\begin{cases}
\eps^2 \dt^2 v + \dt v +(1-\Dl)v  +
\Im\big( e^{i \be v} \U \big)     
=0\\
(v,  \ind_{\eps > 0} \dt v) |_{t=0} =  (\phi_0, \ind_{\eps > 0} \dt \phi_1),
\end{cases}
\quad (x,\eps,t) \in \T^2 \times [0,1] \times \R_+.
\label{sin}
\end{align}

\noi
for a given (deterministic) distribution $(\eps,t) \mapsto \U(\eps,t)$. We prove the following local well-posedness result:

\noi
\begin{proposition}\label{PROP:sinlwp}
Let  $\, \! 0 \! < \! \al \! < \! \frac12$, $(\phi_0, \phi_1) \in \H^{1- \al}(\T^2)$, and $(\eps,t) \mapsto \U_\eps(t)$ be a distribution in $C ( [0,1]^2 ; W^{ - \al, \infty }( \T^2) )$. Then, there exists a random time $T = T( \| \U \|_{ C( [0,1]^2; W^{- \al, \infty }_x ) }) \in (0,1] $ and a unique solution $v$ to \eqref{sin} in the class ${C( [0,1]\times [0,T] ; H^{1- \al}(\T^2))}$ with initial data $(\phi_0,\phi_1)$. Moreover, the solution map $(\phi_0, \phi_1, \U) \mapsto v$ is continuous.
\end{proposition}

\begin{proof}
The proof of Proposition \ref{PROP:sinlwp} essentially follows that of \cite[Proposition 3.1]{ORSW2} upon very minor modifications. The key observation which allows us to globalize the solutions that are constructed in Proposition \ref{PROP:sinlwp} is that the local existence time $T$ only depends on the distribution $(\eps,t) \mapsto \Theta_\eps(t)$ and not on the initial data since the bounds in the perturbative argument are in fact linear in the nonlinear remainder $v$ solution to \eqref{sin}.
\end{proof}

As in Subsection \ref{SEC6}, the term $\dt v_\eps(t)$, for any fixed $t \in \R_+$, becomes unbounded as $\eps \to 0$. This causes issues when trying to extend the solution constructed in Proposition \ref{PROP:sinlwp} to larger times. The next result aims at fixing this problem.

Recall the definition of the space $\mathcal V^s_{1}(\T^2)$, $s \in \R$, in \eqref{norm10}.

\noi
\begin{proposition}\label{PROP:sinlwp2}
Let $ \! 0 \! < \! \al \! < \! \frac12$, $(\phi_0, \phi_1) \in \H^{1- \al}(\T^2)$, and $(\eps,t) \mapsto \U_\eps(t)$ be a distribution in $C ( [0,1]^2 ; W^{ - \al, \infty }( \T^2) )$. Fix $T_0 \in (0,1]$. Then, there exist a time $T = T( \| \U \|_{ C( [0,1]^2; W^{- \al, \infty }_x ) }) \in (0,1] $ and a unique solution $\vec v := (v, \dt v)$ to \eqref{sin} on $(0,1] \times [T_0,T]$ with initial data $(\phi_0,\phi_1)$ at time $T_0$, in the class $C \big( [T_0, T]; \V_{1}^{1-\al}(\T^2) \big).$ Moreover, the solution map $(\phi_0, \phi_1, \Theta) \mapsto \vec v$ is continuous.
\end{proposition}

\begin{proof} The proof follows from computations that are similar to those written in the proofs of Proposition \ref{PROP:LWP2} and \cite[Proposition 3.1]{ORSW2}; we omit details.
\end{proof}

We now conclude this section with the proof of Theorem \ref{THM:sinGWP}.

\begin{proof}[Proof of Theorem \ref{THM:sinGWP}] Let $\be^2 \in (0,2\pi)$ and fix $\al \in \big( \frac{\be^2}{4\pi}, \frac12 \big)$. We only prove the result for $T=1$ for convenience and consider a fixed inital data $(\phi_0,\phi) \in \H^{1-\al}(\T^2)$. By Proposition \ref{PROP:stosin2} and Proposition \ref{PROP:sinlwp}, we may construct $v$ solution to \eqref{SdSG3} on $[0,1] \times [0,T_0]$ for some small time $T_0 \in (0,1]$. 

By applying Proposition \ref{PROP:stosin2} Proposition \ref{PROP:sinlwp2} (possibly repeatedly) allows us to extend $\{v_{\eps}\}_{\eps \in (0,1]}$ to the whole time interval $[0,1]$ as the existence time in Proposition \ref{PROP:sinlwp2} only depend on the noise $\{\Theta_{\eps}\}_{\eps \in [0,1]}$. 

Similarly, it easy to show that $v_0$, the solution to \eqref{SdSG3} on the slice $\{ \eps = 0 \}$ can also be extended to the time interval $[0,1]$ by repeating the computations in the proof of \cite[Proposition 3.1]{ORSW2} and noting that the local existence time only depends on the noise $\Theta_0$.

We can then prove the convergence of $v_\eps$ to $v_0$ in $C\big([0,1]; H^{1-\al}(\T^2)\big)$ as $\eps \to 0$ by adapting the argument of Lemma \ref{LEM:GWP1} to the current setting. The convergence of $u_\eps := \Psi_\eps + v_\eps$ to $u_0 := \Psi_0 + v_0$ in $C\big([0,1]; H^{-\s}(\T^2)\big)$, $\s >0$, as $\eps \to 0$ then follows from the convergence of the nonlinear remainders and Proposition \ref{PROP:sto}. This finishes the proof.
\end{proof}

\begin{ackno}\rm
The author would like to thank Tadahiro Oh, for suggesting this problem and his support throughout its completion. The author would also like to thank the anonymous referees for their helpful comments. The author was supported by the European Research Council (grant no.~864138 ``SingStochDispDyn'') and by the chair of probability and PDEs at EPFL.
\end{ackno}


\end{document}